\documentclass[12pt]{article}
\usepackage{amssymb,amsmath,amsthm,ascmac}
\usepackage{enumerate}

\newtheorem{thm}{Theorem}

\newtheorem{lem}{Lemma}[section]
\newtheorem{pro}[lem]{Proposition}
\newtheorem{rem}{Remark}[section]

\newcommand{\dis}{\displaystyle}

\newcommand{\R}{{\Bbb R}}
\newcommand{\N}{{\Bbb N}}

\newcommand{\pa}{\partial}

\makeatletter

\@addtoreset{equation}{section}
\makeatother

\usepackage{fancyhdr}
\title{\Large\sf Stability of nondegenerate ODE type blowup for the Fujita type heat equation}
\author{Junichi Harada}

\begin{document}
\maketitle
\thispagestyle{empty}
\large

 \begin{abstract}
 The asymptotic bahavior of blowup solutions to the Fujita type heat equation
 $u_t=\Delta u+|u|^{p-1}u$ is studied.
 This equation admits the ODE type blowup given by
 $u(x,t)=(p-1)^\frac{1}{p-1}(T-t)^{-\frac{1}{p-1}}$.
 It is known that
 nondegenerate ODE type blowup is stable if $p\in(1,\frac{n+2}{n-2})$
 due to Fermanian Kammerer-Merle-Zaag (2000).
 This paper extends their result to more general case.
 \end{abstract}

 \tableofcontents

 \section{Introduction}
 \label{sec_1}
 We are concerned with blowup solutions to the Fujita type heat equation.
 \begin{equation}
 \label{eq_1.1}
 \begin{cases}
 u_t = \Delta u+|u|^{p-1}u
 \qquad
 \text{in } \R^n\times(0,T),\\
 u|_{t=0}=u_0(x),
 \end{cases}
 \end{equation}
 where $n\geq1$ and $p>1$.
 It is known that
 \eqref{eq_1.1} admits a unique local classical solution
 for any bounded continuous initial data $u_0$.
 If there exists $T>0$ such that $\limsup_{t\to T}\|u(t)\|_\infty=\infty$,
 we say that a solution $u(x,t)$ blows up in a finite time $T$.
 A typical blowup solution is given by
 \[
 u(x,t)
 =
 \pm(p-1)^{-\frac{1}{p-1}}
 (T-t)^{-\frac{1}{p-1}}
 \qquad
 (\text{ODE type}).
 \]
 Furthermore
 we call $a\in\R^n$ a blowup point of $u(x,t)$
 if there exists a sequence $\{(a_j,t_j)\}_{j\in\N}\subset\R^n\times(0,T)$
 such that $|u(a_j,t_j)|\to\infty$ as $j\to\infty$.
 To describe the asymptotic behavior of solutions to \eqref{eq_1.1},
 we introduce the similarity variables.
 Let $u(x,t)$ be a blowup solution of \eqref{eq_1.1}
 and put
 \begin{align*}
 w_{a,T}(z,\tau)
 =
 (T-t)^\frac{1}{p-1}
 u(a+z\sqrt{T-t},t)
 \qquad
 \text{with }
 T-t=e^{-\tau},
 \end{align*}
 where $T>0$ is the blowup time
 and $a\in\R^n$ is the blowup point of $u(x,t)$.
 This rescaled function $w_{a,T}(z,\tau)$ satisfies
 \begin{align}
 \label{eq_1.2}
 \pa_\tau
 w_{a,T}
 &=
 \Delta_z
 w_{a,T}
 -
 \tfrac{z}{2}
 \cdot
 \nabla_z
 w_{a,T}
 -
 \tfrac{1}{p-1}
 w_{a,T}
 +
 |w_{a,T}|^{p-1}
 w_{a,T}
 \\
 \nonumber
 &\qquad
 \text{for }
 (z,\tau)\in\R^n\times(-\log T,\infty).
 \end{align}
 The long time behavior of $w_{a,T}(z,\tau)$
 describes the asymptotic behavior of $u(x,t)$ near the blowup point.
 The constant $\pm(p-1)^{-\frac{1}{p-1}}$ give
 a stationary solution of \eqref{eq_1.2},
 which corresponds to the ODE type blowup in the original variables.
 In the pioneering work by Giga-Kohn \cite{Giga-Kohn_1985, Giga-Kohn_1987},
 they proved that
 all blowup solutions satisfy
 \begin{align}
 \label{eq_1.3}
 \lim_{\tau\to\infty}
 w_{a,T}(z,\tau)
 &=
 (p-1)^{-\frac{1}{p-1}}
 \qquad
 \text{locally uniformly on } \R^n
 \end{align}
 for $p\in(1,\frac{n+2}{n-2})$ under some technical assumptions,
 employing the similarity variables
 (see also Giga-Matsui-Sasayama \cite{Giga-Matsui-Sasayama}).
 The convergence \eqref{eq_1.3} is equivalent to
 \begin{align}
 \label{eq_1.4}
 \lim_{t \to T}
 \lim_{|x-a|<K\sqrt{T-t}}
 (T-t)^\frac{1}{p-1}u(x,t)
 &=
 (p-1)^{-\frac{1}{p-1}}
 \\
 \nonumber
 &
 \text{for any fixed }
 K>1.
 \end{align}
 On the other hand,
 it is known that
 there are several blowup solutions whose profile is different from
 \eqref{eq_1.3}
 for the case $p\geq\frac{n+2}{n-2}$.
 In this paper,
 we are interested in the stability of blowup solutions
 that are locally of ODE type
 but have isolated blowup points
 for $p\geq\frac{n+2}{n-2}$.
 To obtain a more precise asymptotic profile of blowup solutions,
 we investigate the dynamics of solutions to the linearized equation
 around $w_{a,T}=(p-1)^{-\frac{1}{p-1}}$. 
 There are two possible asymptotic profiles
 for the solutions to the linearized equation
 (see \cite{Herrero-Velazquez_1D,Velazquez_classification,Filippas-Kohn}).
 \begin{enumerate}[({case} i)]
 \item
 There exists $L\in\{1,2,\cdots,n\}$
 such that
 $\dis\lim_{\tau\to\infty}
 \tau
 \Bigl|
 w_{a,T}(z,\tau)-(p-1)^{-\frac{1}{p-1}}
 +
 \tfrac{(p-1)^{-\frac{1}{p-1}}}{4p\tau}
 \Bigl
 (\sum_{j=1}^Lz_j^2-2n\Bigr)
 \Bigr|=0$
 locally uniformly on $\R^n$.
 \item
 There exists $\gamma>0$ such that
 for any $K>1$ there exists $C(K)>0$ such that
 $\dis\sup_{|z|<K}|w_{a,T}(z,\tau)-(p-1)^{-\frac{1}{p-1}}|<C(K)e^{-\gamma\tau}$.
 \end{enumerate}
 We call nondegenerate ODE type blowup,
 if (case i) holds with $L=n$.
 When a blowup is of nondegenerate ODE type,
 its profile is given by
 \begin{align}
 \label{eq_1.5}
 w_{a,T}(z,\tau)
 \approx
 (p-1)^{-\frac{1}{p-1}}
 -
 \tfrac{(p-1)^{-\frac{1}{p-1}}}{4p\tau}(|z|^2-2n).
 \end{align}
 The solution exhibiting a nondegenerate ODE type blowup
 has only one maximum point near the blowup point at $t\sim T$,
 as in \eqref{eq_1.5}.
 For 1D case,
 the stability of blowup solutions is completely understood
 due to Herrero-Vel\'azquez (Theorem \cite{Herrero-Velazquez_1Dgeneric} p. 385).
 They proved that
 \begin{enumerate}[(rI)]
 \item
 the nondegenerate ODE type blowup is stable under small perturbations
 in the $L^\infty$ norm on the initial data,
 \item
 the degenerate ODE type blowup (case ii) is unstable.
 Precisely let $u(x,t)$ be a blowup solution satisfying (case ii).
 Then
 for any $\epsilon\in(0,1)$,
 there exists a solution $\tilde u(x,t)$
 such that $\|\tilde u_0-u_0\|_\infty<\epsilon$
 and $\tilde u(x,t)$ exhibits nondegenerate ODE type blowup.
 \end{enumerate}
 For a higher dimensional case,
 Fermanian Kammerer-Merle-Zaag in \cite{Kammerer-Merle-Zaag}
 (see also \cite{Merle-Zaag})
 obtained the same stability result as (rI)
 under the condition $p\in(1,\frac{n+2}{n-2})$.
 The goal of this paper is
 to extend their result to a general case.
 \begin{thm}
 \label{thm_1}
 Let $n\in\N$ and $p>1$.
 The nondegenerate ODE type blowup is stable in the sense of {\rm(rI)}.
 \end{thm}
 In Section \ref{sec_2},
 we will give full statement of this theorem
 after preparing some notations (see Theorem \ref{thm_2})
 and some remarks on it.

 \section{Notations and Main result}
 \label{sec_2}
 Let $u(x,t)$ be a classical solution of $u_t=\Delta u+|u|^{p-1}u$
 and let $T$ be its blowup time.
 As in Introduction,
 we define a rescaled function $w(z,\tau)$ by
 \begin{align}
 \label{eq_2.1}
 w(z,\tau)
 =
 (T-t)^{\frac{1}{p-1}}
 u(z\sqrt{T-t},t)
 \qquad
 \text{with }
 T-t=e^{-\tau}.
 \end{align}
 The new function $w(z,\tau)$ satisfies
 \begin{align}
 \label{eq_2.2}
 w_\tau
 =
 \Delta_z
 w
 -
 \tfrac{z}{2}
 &\cdot
 \nabla_z
 w
 -
 \tfrac{w}{p-1}
 +
 |w|^{p-1}
 w
 \end{align}
 for
 $(z,\tau)\in\R^n\times(-\log T,\infty)$.
 For simplicity,
 we put
 \begin{align}
 \label{eq_2.3}
 \kappa
 =
 (p-1)^{-\frac{1}{p-1}}
 \end{align}
 The constant $\kappa$ gives a stationary solution of \eqref{eq_2.2},
 which corresponds to the ODE type blowup $u(x,t)=\kappa(T-t)^{-\frac{1}{p-1}}$
 in the original valuables $x,t$.
 It is convenient to introduce the weighted Lebesgue spaces defined by
 \begin{align*}
 L_\rho^2&(\R^n)
 =
 \{
 w(z)\in L_{\text{loc}}^1(\R^n);\
 \|w\|_{L_\rho^2(\R^n)}<\infty
 \}
 \quad
 \text{with }
 \\
 &
 \|w\|_{L_\rho^2(\R^n)}
 =
 \left(
 \int_{\R^n}|w(z)|^2\rho(z)dz
 \right)^\frac{1}{2}
 \quad
 \text{and }
 \quad
 \rho(z)
 =
 e^{-\frac{|z|^2}{4}}.
 \end{align*}
 This space $L_\rho^2(\R^n)$ has an inner product given by
 \[
 \langle w_1,w_2 \rangle_\rho
 =
 \int_{\R^n}
 w_1(z)
 w_2(z)
 \rho(z)
 dz
 \quad
 \text{for }
 w_1,w_2
 \in
 L_\rho^2(\R^n).
 \]
 A wighted Sobolev space is defined in the same way.
 \begin{align*}
 H_\rho^1&(\R^n)
 =
 \{
 w(z),\nabla_zw(z)\in L_\rho^2(\R^n)
 \}.
 \end{align*}
 The inner product on $H_\rho^1(\R^n)$ is defined by
 \begin{align*}
 \langle
 w_1,w_2
 \rangle_{H_\rho^1(\R^n)}
 =
 \langle w_1,w_2\rangle_\rho
 +
 \sum_{j=1}^n
 \langle \pa_{z_j}w_1,\pa_{z_j}w_2\rangle_\rho
 \qquad
 \text{for }
 w_1,w_2\in H_\rho^1(\R^n).
 \end{align*}
 To investigate the dynamics near $w=\kappa$,
 we consider a linearized equation for $W(z,\tau)=w(z,\tau)-\kappa$.
 \begin{align}
 \label{eq_2.4}
 W_\tau
 =
 \Delta_z
 W
 -
 \tfrac{z}{2}
 \cdot
 \nabla_z
 W
 +
 W
 +
 N,
 \end{align}
 here
 $N=|\kappa+W|^{p-1}(\kappa+W)-\kappa^p-p\kappa^{p-1}W$.
 For simplicity,
 we put
 \begin{align}
 \label{eq_2.5}
 A_z
 =
 \Delta_z
 -
 \tfrac{z}{2}
 \cdot
 \nabla_z.
 \end{align}
 Consider the eigenvalue problem
 \begin{align}
 \label{eq_2.6}
 -A_zH=\mu H
 \qquad
 \text{in }
 L_\rho^2(\R^n).
 \end{align}
 We collect spectrum properties of $A_z$.
 \begin{itemize}
 \item
 The $j$-th eigenvalu of \eqref{eq_2.6} is given by
 \[
 \mu_j=\tfrac{j}{2}
 \qquad
 (j=0,1,2,\cdots).
 \]
 \item
 We denote by $V_j$ the eigenspace associated with $\mu_j$ ($j=0,1,2,\cdots$).
 \item
 The eigenspace $V_0$ associated with $\mu_0=0$ is given by
 \begin{align*}
 V_0
 =
 \text{span}\,\{H_0(z)\}
 \qquad
 (\text{dim}\,V_0=1),
 \end{align*}
 here
 \begin{align*}
 H_0(z)
 =
 {\sf k}_0^n
 \qquad
 ({\sf k}_0=(4\pi)^{-\frac{1}{4}}).
 \end{align*}
 \item
 The eigenspace $V_1$ associated with $\mu_1=\frac{1}{2}$ is given by
 \begin{align*}
 V_1
 =
 \text{span}\,\{H_{1,1}(z),H_{1,2}(z),\cdots,H_{1,n}(z)\}
 \qquad
 (\text{dim}\,V_1=n),
 \end{align*}
 here
 \begin{align*}
 H_{1,J}(z)
 =
 \tfrac{{\sf k}_0^n}{\sqrt{2}}
 z_J.
 \end{align*}

 \item
 The eigenspace $V_2$ associated with $\mu_1=1$ is given by
 \begin{align*}
 V_2
 =
 \text{span}\,\{H_{2,IJ}(z);I\leq J\}
 \qquad
 (\text{dim}\,V_2=\tfrac{n(n+1)}{2}),
 \end{align*}
 here
 \begin{align*}
 H_{2,JJ}(z)
 &=
 \tfrac{{\sf k}_0^n}{2\sqrt{2}}(z_I^2-2),
 \\
 H_{2,IJ}(z)
 &=
 \tfrac{{\sf k}_0^n}{2}z_Iz_J
 \qquad \text{if } I<J.
 \end{align*}
 All eigenfunctions
 $H_{0}(z)$, $H_{1,J}(z)$ and $H_{2,IJ}(z)$
 are normalized as $\|\cdot\|_{L_\rho^2(\R^n)}=1$.
 \end{itemize}
 To state our theorem,
 we define
 \begin{align}
 \label{eq_2.7}
 c_p
 =
 \tfrac{\sqrt{2}p{\sf k}_0^n}{\kappa}.
 \end{align}
 \begin{thm}\label{thm_2}
 Let $n\in\N$ and $p>1$.
 Let $\varphi(x,t)$ be a classical solution of
 \begin{align*}
 \begin{cases}
 \varphi_t=\Delta \varphi+|\varphi|^{p-1}\varphi
 \qquad
 \text{\rm for }
 (x,t)\in\R^n\times(0,T),
 \\
 \varphi(x,0)
 =
 \varphi_0\in C(\R^n)\cap L^\infty(\R^n).
 \end{cases}
 \end{align*}
 Assume that
 $\varphi(x,t)$ satisfies
 \begin{enumerate}[\rm ({\rm a}1)]
 \setlength{\leftskip}{5mm}
 \item
 $\varphi(x,t)$ blows up in a finite time $t=T$,
 \item
 $\dis\sup_{t\in(0,T)}\|\varphi(x,t)\|_{L^\infty(|x|>r)}<\infty$ \
 for any fixed $r>0$,
 \item
 there exists $M>0$ such that
 $\varphi(x,t)>-M$ \
 for $x\in\R^n$ and $t\in(0,T)$,
 \item
 $\dis\sup_{t\in(0,T)}(T-t)^\frac{1}{p-1}\|\varphi(x,t)\|_{L_x^\infty(\R^n)}<\infty$,
 \item
 a rescaled function
 $\Phi(z,\tau)=e^{-\frac{\tau}{p-1}}\varphi(e^{-\frac{\tau}{2}}z,T-e^{-\tau})$
 satisfies
 $\dis\lim_{\tau\to\infty}\|\Phi(z,\tau)-\kappa\|_{L_\rho^2(\R^n)}=0$
 \quad
 {\rm(}$\kappa=(p-1)^{-\frac{1}{p-1}}${\rm)},
 \item
 $\dis
 \lim_{\tau\to\infty}
 \tau
 \|
 \Phi(z,\tau)
 -
 \kappa
 +
 \tfrac{1}{c_p\tau}
 \sum_{J=1}^n
 H_{2,JJ}(z)
 \|_{H_\rho^1(\R^n)}
 =0$.
 \end{enumerate}
 Then
 there exist $\delta_1\in(0,1)$ and $K_1>0$ depending only on $\varphi(x,t)$
 such that
 if $\|u_0-\varphi_0\|_{L^\infty(\R^n)}<\delta_1$,
 the corresponding solution $u(x,t)$ of \eqref{eq_1.1} satisfies
 \begin{enumerate}[{\rm (}\sf i{\rm)}]
 \setlength{\leftskip}{5mm}
 \item
 $u(x,t)$ blows up in a finite time $t=T(u)$,
 \item
 $\dis\sup_{t\in(0,T)}(T(u)-t)^\frac{1}{p-1}
 \|u(x,t)\|_{L_x^\infty(\R^n)}<K_1$,
 \item
 there exists a unique $\xi(u)\in\R^n$
 such that
 a rescaled function
 $w(z,\tau)=e^{-\frac{\tau}{p-1}}u(\xi(u)+e^{-\frac{\tau}{2}}z,$ $T(u)-e^{-\tau})$
 satisfies
 $\dis\lim_{\tau\to\infty}\|w(z,\tau)-\kappa\|_{L_\rho^2(\R^n)}=0$,
 \item
 $\dis
 \lim_{\tau\to\infty}
 \tau
 \|
 w(z,\tau)
 -
 \kappa
 +
 \tfrac{1}{c_p\tau}
 \sum_{J=1}^n
 H_{2,JJ}(z)
 \|_{H_\rho^1(\R^n)}
 =0$.
 \end{enumerate}
 \end{thm}

 \begin{rem}\label{Rem2.1}
 The blowup is classified into two cases in terms of its blowup rate.
 \begin{align*}
 \sup_{t\in(0,T)}
 (T-t)^\frac{1}{p-1}
 \|u(x,t)\|_{L_x^\infty(\R^n)}
 <
 \infty
 \qquad
 \text{\rm(type I)},
 \\
 \sup_{t\in(0,T)}
 (T-t)^\frac{1}{p-1}
 \|u(x,t)\|_{L_x^\infty(\R^n)}
 =
 \infty
 \qquad
 \text{\rm(type II)}.
 \end{align*}
 The most famous sufficient conditions for type I blowup
 go back to Giga-Kohn {\rm\cite{Giga-Kohn_1985, Giga-Kohn_1987}}.
 They proved that
 all blowup solutions are type I if
 \begin{enumerate}[{\rm(gi)}]
 \item
 $p\in(1,\frac{n+2}{n-2})$ and $u_0(x)\geq0$,
 or
 \item
 $p\in(1,\frac{3n+8}{3n-4})$.
 \end{enumerate}
 Their result was extended to the case $p\in(1,\frac{n+2}{n-2})$
 by Giga-Matsui-Sasayama {\rm\cite{Giga-Matsui-Sasayama}}
 {\rm(}see also Quittner {\rm\cite{Quittner}}{\rm)}.
 For the case $p\geq\frac{n+2}{n-2}$,
 there are several significant works
 that provide criteria for type I blowup.
 They are listed below.
 \begin{enumerate}[\rm(f1)]
 \item
 Let $p\in(\frac{n+2}{n-2},1+\frac{4}{n-4-2\sqrt{n-1}})$.
 If $u_0(x)\in L^\infty(\R^n)\cap H^1(\R^n)$ is radially symmetric,
 then any blowup is type I
 {\rm(Theorem 3.9 p. 1027 in Matano-Merle {\rm\cite{Matano-Merle_2009}})}.

 \item
 Let $p\in(\frac{n+2}{n-2},\infty)$.
 For any radially symmetric function $v(x)\geq0$,
 consider an initial data of the form $u_0(x)=\lambda v(x)$ $(\lambda>0)$.
 There exist finitely many values $0<\lambda_1<\lambda_2<\cdots<\lambda_k$
 {\rm(}depend on $v(x)${\rm)}
 such that
 the solution $u^{(\lambda)}(x,t)$ starting from $u_0(x)=\lambda v(x)$
 exhibits type I blowup for any
 $\lambda\in(\lambda_1,\infty)\setminus\{\lambda_2,\lambda_3,\cdots,\lambda_k\}$
 {\rm(Theorem 1.2 p. 721 in Matano-Merle {\rm\cite{Matano-Merle_2011}})}.

 \item
 \label{f3}
 Let $p=\frac{n+2}{n-2}$.
 If $u_0(x)$ is nonnegative, radially symmetric and $\frac{\pa u_0}{\pa r}\leq0$,
 then any blowup is type I
 {\rm(Section 6 p. 2972 - p. 2975 in
 Filippas-Herrero-Vel\'azquez \cite{Filippas-Herrero-Velazquez})}.

 \item
 Let $p=\frac{n+2}{n-2}$.
 The same result as {\rm (f\ref{f3})} holds true 
 without the assumption $\frac{\pa u_0}{\pa r}\leq0$
 {\rm(Theorem 1.7 p. 1498 in Matano-Merle {\rm\cite{Matano-Merle_2004}})}.

 \item
 Let $p=\frac{n+2}{n-2}$.
 The set of initial data $u_0(x)$
 such that the corresponding solution $u(x,t)$ blows up in a finite time
 with type I is open in $C(\R^n)\cap L^\infty(\R^n)$
 {\rm(Theorem 1.1 p. 66 in
 Collot-Merle-Rapha\"el \cite{Collot-Merle-Raphael_Stability})}.

 \item
 Let $p=\frac{n+2}{n-2}$, $n\geq7$ and put
 ${\sf Q}(x)=(1+\frac{|x|^2}{n(n-2)})^{-\frac{n-2}{2}}$.
 There exists $\eta\in(0,1)$ such that
 if the initial data $u_0(x)\in \dot H^1(\R^n)$ satisfies
 $\|u_0(x)-{\sf Q}(x)\|_{\dot H^1(\R^n)}<\eta$,
 then
 if the solution $u(x,t)$ blows up in a finite time,
 the blowup is type I
 {\rm( Theorem 1.1 p. 217 in Collot-Merle-Rapha\"el
 \cite{Collot-Merle-Raphael_Dynamics})}.

 \item
 Let $p=\frac{n+2}{n-2}$, $n\geq7$.
 If $u_0(x)$ is nonnegative,
 any blowup is type I.
 \end{enumerate}
 Out result provides a new criterion for type I blowup
 {\rm(}see Theorem {\rm\ref{thm_2}} {\rm({\sf ii})}{\rm)}.
 \end{rem}

 \begin{rem}\label{Rem2.2}
 Our proof is completely different from
 Herrero-Vel\'azquez {\rm\cite{Herrero-Velazquez_1Dgeneric}}
 and
 Fermanian Kammerer-Merle-Zaag {\rm\cite{Kammerer-Merle-Zaag}}.
 Our main idea of the proof
 is to trace the maximum point of the solution $u(x,t)$ for all time $t\in(0,T(u))$.
 Let $u(x,t)$ be a blowup solution and
 let $w(z,\tau)$ be its rescaled function defined by
 \eqref{eq_2.1}.
 We will see that there exists a unique $C^1$ function
 $\xi(\tau)$ on $\tau\in(-\log T(u),\infty)$ such that
 \begin{align}
 \label{eq_2.8}
 \langle w(z+\xi(\tau),\tau)&-\kappa,H_{1,J}(z)\rangle_{L_\rho^2(\R_z^n)}=0
 \\
 \nonumber
 &
 \text{\rm for } \tau\in(-\log T(u),\infty)
 \text{ \rm and all } J.
 \end{align}
 We expect that
 $v(z,\tau)=w(z+\xi(\tau),\tau)$ satisfies
 \begin{align*}
 \lim_{\tau\to\infty}
 v(z,\tau)
 =
 \kappa
 \qquad
 \text{in }
 L_\rho^2(\R_{z}^n).
 \end{align*}
 Put
 $V(z,\tau)=v(z,\tau)-\kappa$.
 Then
 $V(z,\tau)$ satisfies
 \begin{align}
 \label{eq_2.9}
 V_\tau
 =
 A_{z}V
 +
 V
 +
 (\tfrac{d\xi}{d\tau}-\tfrac{\xi}{2})
 \cdot
 \nabla_{z}V
 +
 N(V).
 \end{align}
 We can verify that
 the contribution of
 the last two terms\,{\rm:}
 $(\tfrac{d\xi}{d\tau}-\tfrac{\xi}{2})
 \cdot
 \nabla_{z}V$
 and
 $N(V)$
 on the right hand side are negligible in \eqref{eq_2.9}.
 Therefore
 $V(z,\tau)$ is governed by the linear equation\,{\rm:}
 $ V_\tau
 =
 A_{z}V
 +
 V$
 for $\tau\in(-\log T(u),\infty)$.
 We recall that $(A_z+1)$ has $(n+1)$ unstable modes
 given by
 $H_0(z),H_{1,1}(z),H_{1,2}(z),\cdots,H_{1,n}(z)$.
 From the choice of $\xi(\tau)$ {\rm(see \eqref{eq_2.8})},
 it is sufficient to investigate the behavior of $V(z,\tau)$ along $H_0(z)$.
 As a consequence of
 {\rm Proposition \ref{proposition_4.2}}
 and
 {\rm Proposition \ref{proposition_4.4}},
 we will see that
 $|\langle V(z,\tau),H_0(z)\rangle_{L_\rho^2(\R_z^n)}|\ll
 \|V(z,\tau)\|_{L_\rho^2(\R_z^n)}$.
 This fact concludes that
 $V(z,\tau)$ does not change its asymptotic form
 and
 proves the stability result stated in Theorem {\rm\ref{thm_2}}.
 \end{rem}

\section{Preliminary}
 \label{sec_3}
 For simplicity,
 we prepare some notations.
 \vspace{2mm}
 \begin{itemize}
 \item 
 ${\bf a}_1=(a_{1,1},a_{1,2},\cdots,a_{1,n})\in\R^n$,
 \item 
 ${\bf a}_{2,JJ}=(a_{2,11},a_{2,22},\cdots,a_{2,nn})\in\R^n$,
 \item 
 ${\bf a}_{2,IJ}=(\{a_{2,IJ}\}_{I<J})
 \in\R^{\frac{n^2-n}{2}}$,
 \item
 ${\bf a}_2=({\bf a}_{2,JJ},{\bf a}_{2,IJ})\in\R^{\frac{n^2+n}{2}}$.
 \end{itemize}
 \vspace{2mm}
 A standard Euclid norm is written as
 \begin{align*}
 |{\bf a}_1|
 &=
 \sqrt{
 a_{1,1}^2
 +
 \cdots
 +
 a_{1,n}^2},
 \\
 |{\bf a}_{2,JJ}|
 &=
 \sqrt{
 a_{2,11}^2
 +
 \cdots
 +
 a_{2,nn}^2},
 \\
 |{\bf a}_{2,IJ}|
 &=
 \sqrt{
 a_{2,12}^2
 +
 \cdots
 +
 a_{2,n-1n}^2},
 \\
 |{\bf a}_2|
 &=
 \sqrt{
 |{\bf a}_{2,JJ}|^2
 +
 |{\bf a}_{2,IJ}|^2}.
 \end{align*}
 In the same way,
 we put
 \begin{itemize}
 \item 
 ${\bf H}_1=(H_{1,1},H_{1,2},\cdots,H_{1,n})\in\R^n$,
 \item 
 ${\bf H}_{2,JJ}=(H_{2,11},H_{2,22},\cdots,H_{2,nn})\in\R^n$,
 \item 
 ${\bf H}_{2,IJ}=(\{H_{2,IJ}\}_{I<J})
 \in\R^{\frac{n^2-n}{2}}$,
 \item
 ${\bf H}_2=({\bf H}_{2,JJ},{\bf H}_{2,IJ})\in\R^{\frac{n^2+n}{2}}$.
 \end{itemize}
 Under these notations,
 we  write
 \begin{itemize}
 \item
 $\dis{\bf a}_1\cdot{\bf H}_1=\sum_{J=1}^na_{1,J}H_{1,J}$,
 \item
 $\dis{\bf a}_{2,JJ}\cdot{\bf H}_{2,JJ}=\sum_{J=1}^na_{2,JJ}H_{2,JJ}$,
 \item
 $\dis{\bf a}_{2,IJ}\cdot{\bf H}_{2,IJ}=\sum_{I<J}a_{2,IJ}H_{2,IJ}$,
 \item
 $\dis{\bf a}_2\cdot{\bf H}_2
 ={\bf a}_{2,JJ}\cdot{\bf H}_{2,J}+{\bf a}_{2,IJ}\cdot{\bf H}_{2,IJ}$.
 \end{itemize}

 \subsection{Functional inequalities}
 \label{sec_3.2}
 We collect fundamental estimates for solutions to
 \begin{align}
 \label{eqeqeq_3.2}
 \pa_\tau w=A_zw+f(z,\tau)
 \qquad
 (A_z=\Delta_z-\tfrac{z}{2}\cdot\nabla_z).
 \end{align}
 \begin{lem}[(6) p. 4218 \cite{Harada}]
 \label{lem_3.1}
 There exists $C>0$ such that
 \begin{equation*}
 \int_{\R^n}
 |z|^2f(z)^2
 \rho(z)
 dz
 <
 C
 \|f\|_{H_\rho^1(\R^n)}^2
 \qquad
 \text{\rm for }
 f\in H_\rho^1(\R^n).
 \end{equation*}
 \end{lem}

 \begin{lem}[Lemma 2.1 p. 4219 \cite{Harada}]
 \label{lem_3.2}
 For any $R>0$
 there exists $c_R>0$ such that
 \begin{equation*}
 \sup_{|z|<Re^\frac{\tau}{2}}
 |e^{A_z\tau}w_0|
 <
 c_R
 \|w_0\|_\rho
 \qquad
 \text{\rm for }
 \tau>1.
 \end{equation*}
 \end{lem}

 \begin{lem}[Lemma 2.3 p. 4220 \cite{Harada}]
 \label{lem_3.3}
 For any $\sigma>0$ and $p>0$,
 there exists $C>0$ such that
 if $\tau-2\log \tau<\tau_1<\tau$,
 then it holds that
 \begin{equation*}
 \sup_{|z|<\tau^\sigma}
 \int_{\tau_1}^\tau
 e^{A_z(\tau-s)}
 (|\xi|^p{\bf 1}_{|\xi|>2\tau^\sigma})
 ds
 <
 C\tau^{p+n-2}e^{-\frac{1}{16}\tau^{2\sigma}}.
 \end{equation*}
 \end{lem}

 \section{Key estimates for rescaled solutions}
 \begin{lem}
 \label{lemma_4.1}
 Assume $n\geq1$.
 Fix ${\sf m}\in\R\setminus\{0\}$,
 and put
 \begin{align*}
 B(\nu)
 =
 \{w\in L_\rho^2(\R^n);
 \|w-{\sf m}\sum_{J=1}^nH_{2,JJ}\|_\rho<\nu
 \},
 \quad
 \nu\in(0,1).
 \end{align*}
 There exist $\nu_1,\nu_2\in(0,1)$ and $\Psi(w)\in C^1(B(\nu_1);\R^n)$
 satisfying
 \begin{itemize}
 \item
 $\dis
 \langle
 w(z+\Psi(w)),H_{1,J}(z)
 \rangle_\rho=0$
 \quad $(J=1,2,\cdots,n)$,
 \item
 $\Psi(w)=0$ \quad when $\dis w={\sf m}\sum_{J=1}^nH_{2,JJ}$,
 \item
 for $(w,\xi)\in B_{\nu_1}\times\{|\xi|<\nu_2\}$,
 {\rm(i)} and {\rm(ii)} are equivalent\,$:$
 \begin{enumerate}[\rm(i) ]
 \item 
 $\langle w(z+\xi),H_{1,J}(z)\rangle_\rho=0$
 \quad
 {\rm for all} $J$,
 \item
 $\xi=\Psi(w)$,
 \end{enumerate}

 \item
 there exists $c_1>1$
 such that
 \begin{align*}
 |\Psi(w)|
 <
 c_1\|w-{\sf m}\sum_{J=1}^nH_{2,JJ}\|_\rho.
 \end{align*}
 \end{itemize}
 \end{lem}
 \begin{proof}
 We define $F:L_\rho^2(\R^n)\times\R^n\to\R^n$ by
 \begin{align*}
 F(w,\xi)=
 \begin{pmatrix}
 \langle w(z+\xi),H_{1,1}(z)\rangle_\rho
 \\
 \vdots
 \\
 \langle w(z+\xi),H_{1,n}(z)\rangle_\rho
 \end{pmatrix}.
 \end{align*}
 We note that
 $\pa_wF(w,\xi)\in\mathcal{L}(L_\rho^2(\R^n);\R^n)$ is given by
 \begin{align*}
 \langle\pa_wF(w,\xi),v\rangle
 =
 \begin{pmatrix}
 \langle v(z+\xi),H_{1,1}(z)\rangle_\rho
 \\
 \vdots
 \\
 \langle v(z+\xi),H_{1,n}(z)\rangle_\rho
 \end{pmatrix}
 \qquad
 \text{for }
 v\in L_\rho^2(\R^n).
 \end{align*}
 From this formula,
 it is clear that
 \begin{align}
 \label{EQU4.1}
 \pa_wF(w,\xi)\in C(L_\rho^2(\R^n)\times\R^n;\mathcal{L}(L_\rho^2(\R^n);\R^n)).
 \end{align}
 We next compute
 $\pa_{\xi}F(w,\xi)$.
 By a direct computation,
 we see that
 \begin{align*}
 \tfrac{\pa F_1}{\pa \xi_1}
 &=
 \tfrac{\pa}{\pa \xi_1}
 \int_{\R^n}
 w(z+\xi)
 H_{1,1}(z)
 e^{-\frac{|z|^2}{4}}
 dz
 \\
 &=
 \tfrac{\pa}{\pa \xi_1}
 \int_{\R^n}
 w(\eta)
 H_{1,1}(\eta-\xi)
 e^{-\frac{|\eta-\xi|^2}{4}}
 d\eta.
 \end{align*}
 Since
 $H_{1,1}(z)=\frac{{\sf k}_0^n}{\sqrt{2}}z_1$,
 we verify that
 \begin{align*}
 \tfrac{\pa}{\pa \xi_1}
 \{
 H_{1,1}(\eta-\xi)
 e^{-\frac{|\eta-\xi|^2}{4}}
 \}
 &=
 \tfrac{\pa}{\pa \xi_1}
 \{
 \tfrac{{\sf k}_0^n}{\sqrt{2}}
 (\eta_1-\xi_1)
 e^{-\frac{|\eta-\xi|^2}{4}}
 \}
 \\
 &=
 -
 \tfrac{{\sf k}_0^n}{\sqrt{2}}
 e^{-\frac{|\eta-\xi|^2}{4}}
 +
 \tfrac{{\sf k}_0^n}{\sqrt{2}}
 \tfrac{(\eta_1-\xi_1)^2}{2}
 e^{-\frac{|\eta-\xi|^2}{4}}.
 \end{align*}
 Hence
 we get
 \begin{align*}
 \tfrac{\pa F_1}{\pa \xi_1}
 &=
 \int_{\R^n}
 w(\eta)
 \tfrac{{\sf k}_0^n}{\sqrt{2}}
 \{
 -1
 +
 \tfrac{(\eta_1-\xi_1)^2}{2}
 \}
 e^{-\frac{|\eta-\xi|^2}{4}}
 d\eta
 \\
 &=
 \tfrac{{\sf k}_0^n}{\sqrt{2}}
 \int_{\R^n}
 w(z+\xi)
 (
 -1
 +
 \tfrac{z_1^2}{2}
 )
 \rho(z)
 dz
 \\
 &=
 \int_{\R^n}
 w(z+\xi)
 \tfrac{{\sf k}_0^n}{2\sqrt{2}}
 (
 z_1^2-2
 )
 \rho(z)
 dz
 \\
 &=
 \int_{\R^n}
 w(z+\xi)
 H_{2,11}(z)
 \rho(z)
 dz
 \\
 &=
 \langle
 w(z+\xi),
 H_{2,11}(z)
 \rangle_\rho.
 \end{align*}
 In the same way,
 we verify that
 \begin{align*}
 \tfrac{\pa}{\pa \xi_2}
 \{
 H_{1,1}(\eta-\xi)
 e^{-\frac{|\eta-\xi|^2}{4}}
 \}
 &=
 \tfrac{{\sf k}_0^n}{\sqrt{2}}
 \tfrac{(\eta_1-\xi_1)(\eta_2-\xi_2)}{2}
 e^{-\frac{|\eta-\xi|^2}{4}}
 \end{align*}
 and
 \begin{align*}
 \tfrac{\pa F_1}{\pa \xi_2}
 &=
 \int_{\R^n}
 w(\eta)
 \tfrac{{\sf k}_0^n}{\sqrt{2}}
 \tfrac{(\eta_1-\xi_1)(\eta_2-\xi_2)}{2}
 e^{-\frac{|\eta-\xi|^2}{4}}
 d\eta
 \\
 &=
 \tfrac{{\sf k}_0^n}{2\sqrt{2}}
 \int_{\R^n}
 w(z+\xi)
 z_1z_2
 \rho(z)
 dz
 \\
 &=
 \tfrac{1}{\sqrt{2}}
 \int_{\R^n}
 w(z+\xi)
 H_{2,12}(z)
 \rho(z)
 dz
 \\
 &=
 \tfrac{1}{\sqrt{2}}
 \langle
 w(z+\xi),
 H_{2,12}(z)
 \rangle_\rho.
 \end{align*}
 Therefore
 it follows that
 \begin{align*}
 &
 \begin{pmatrix}
 \tfrac{\pa F_1}{\pa \xi_1} & \tfrac{\pa F_1}{\pa \xi_2} & \cdots
 & \tfrac{\pa F_1}{\pa \xi_n}
 \\[2mm]
 \tfrac{\pa F_2}{\pa \xi_1} & \tfrac{\pa F_2}{\pa \xi_2} & \cdots
 & \tfrac{\pa F_2}{\pa \xi_n}
 \\[2mm]
 \vdots & \vdots & \vdots &\vdots
 \\[2mm]
 \tfrac{\pa F_n}{\pa \xi_1} & \tfrac{\pa F_n}{\pa \xi_n} & \cdots
 & \tfrac{\pa F_2}{\pa \xi_n}
 \end{pmatrix}
 \\
 &=
 \begin{pmatrix}
 \langle
 w(z+\xi),
 H_{2,11}
 \rangle_\rho
 &
 \langle
 w(z+\xi),
 \tfrac{H_{2,12}}{\sqrt{2}}
 \rangle_\rho
 & \cdots &
 \langle
 w(z+\xi),
 \tfrac{H_{2,1n}}{\sqrt{2}}
 \rangle_\rho
 \\[2mm]
 \nonumber
 \langle
 w(z+\xi),
 \tfrac{H_{2,21}}{\sqrt{2}}
 \rangle_\rho
 &
 \langle
 w(z+\xi),
 H_{2,22}
 \rangle_\rho
 & \cdots &
 \langle
 w(z+\xi),
 \tfrac{H_{2,2n}}{\sqrt{2}}
 \rangle_\rho
 \\[2mm]
 \vdots & \vdots & \vdots &\vdots
 \\[2mm]
 \langle
 w(z+\xi),
 \tfrac{H_{2,n1}}{\sqrt{2}}
 \rangle_\rho
 &
 \langle
 w(z+\xi),
 \tfrac{H_{2,n1}}{\sqrt{2}}
 \rangle_\rho
 & \cdots &
 \langle
 w(z+\xi),
 H_{2,nn}
 \rangle_\rho
 \end{pmatrix}.
 \end{align*}
 This implies that
 \begin{align}
 \label{EQU4.2}
 \pa_\xi F(w,\xi)
 \in
 C(L_\rho^2(\R^n)\times\R^n;\mathcal{L}(\R^n;\R^n))
 \end{align}
 and
 \begin{align}
 \label{EQU4.3}
 &
 \begin{pmatrix}
 \tfrac{\pa F_1}{\pa \xi_1} & \tfrac{\pa F_1}{\pa \xi_2} & \cdots
 & \tfrac{\pa F_1}{\pa \xi_n}
 \\[2mm]
 \tfrac{\pa F_2}{\pa \xi_1} & \tfrac{\pa F_2}{\pa \xi_2} & \cdots
 & \tfrac{\pa F_2}{\pa \xi_n}
 \\[2mm]
 \vdots & \vdots & \vdots &\vdots
 \\[2mm]
 \tfrac{\pa F_n}{\pa \xi_1} & \tfrac{\pa F_n}{\pa \xi_n} & \cdots
 & \tfrac{\pa F_2}{\pa \xi_n}
 \end{pmatrix}
 (w,\xi)
 =
 \begin{pmatrix}
 {\sf m} & 0 & \cdots & 0
 \\[2mm]
 0 & {\sf m} & \cdots & 0
 \\[2mm]
 \vdots & \vdots & \vdots &\vdots
 \\[2mm]
 0 & 0 & \cdots & {\sf m}
 \end{pmatrix}
 \end{align}
 when $\dis(w,\xi)=({\sf m}\sum_{J=1}^nH_{2,JJ},0)$.
 From the definition of $F(w,\xi)$,
 it is clear that
 \begin{align}
 \label{EQU4.4}
 F(w,\xi)=0
 \in\R^n
 \qquad
 \text{when }
 (w,\xi)=({\sf m}\sum_{J=1}^nH_{2,JJ},0).
 \end{align}
 Therefore from \eqref{EQU4.1} - \eqref{EQU4.4},
 we can apply an implicit function theorem,
 and obtain the desired result.
 \end{proof}

 \subsection{Stability estimates}
 Consider
 \begin{align}
 \label{eq_4.5}
 w_\tau
 =
 \Delta_zw
 -
 \tfrac{z}{2}
 \cdot
 \nabla_zw
 &-
 \tfrac{w}{p-1}
 +
 |w|^{p-1}
 w
 \\
 \nonumber
 &\quad
 \text{for }
 z\in\R^n
 \text{ and }
 \tau>\tau_0.
 \end{align}
 We call $w(z,\tau)$ a solution of \eqref{eq_4.5}
 if
 there exists $\tau_1>\tau_0$ such that
 $w(z,\tau)$ satisfies (w1) - (w2) and solve \eqref{eq_4.5}
 for $\tau\in(\tau_0,\tau_1)$.
 \begin{enumerate}[\rm ({w}1) ]
 \setlength{\leftskip}{5mm}
 \item
 $w(z,\tau)\in C([\tau_0,\tau_1);H_\rho^1(\R^n))\cap
 C^{2,1}(\R^n\times(\tau_0,\tau_1))$,
 \item
 $w(z,\tau)\in L^\infty(\R^n\times(\tau_0,\tau_1-h))$ \quad for any small $h>0$.
 \end{enumerate}
 Throughout Proposition \ref{proposition_4.2} - Proposition \ref{proposition_4.4},
 we assume
 \begin{align}
 \label{eq_4.6}
 w(z,\tau_0)
 >
 -e^{-\frac{3\tau_0}{4(p-1)}}
 \qquad
 \text{for }
 z\in\R^n.
 \end{align}
 If $w(z,\tau)$ is a solution of \eqref{eq_4.5} satisfying \eqref{eq_4.6},
 a comparison argument immediately shows
 \begin{align}
 \label{eq_4.7}
 w(z,\tau)
 >
 -e^{-\frac{3\tau}{4(p-1)}}
 \qquad
 \text{for }
 z\in\R^n
 \text{ and }
 \tau>\tau_0.
 \end{align}
 \begin{pro}
 \label{proposition_4.2}
 Assume $n\geq1$ and $p>1$.
 Let
 $w(z,\tau)$ be a solution of \eqref{eq_4.5}
 satisfying \eqref{eq_4.6}.
 There exists $\tau_*>0$ depending only on $p,n$
 such that if there exists $\tau_1\geq\max\{\tau_*,\tau_0\}$
 such that
 \begin{align}
 \label{eq_4.8}
 \langle w(\tau_1)-\kappa,H_0\rangle_\rho
 >
 \tfrac{4p\kappa^{p-1}}{H_0}
 e^{-\frac{3\tau_1}{4(p-1)}},
 \end{align}
 then
 there exists $\tau_\text{\sf b}\in(\tau_1,\infty)$
 such that
 $\dis\limsup_{\tau\to\tau_{\sf b}}\|w(z,\tau)\|_{L_z^\infty(\R^n)}=\infty$.
 \end{pro}
 \begin{proof}
 We put
 $W(z,\tau)
 =
 w(z,\tau)-\kappa$.
 The function $W(z,\tau)$ satisfies
 \begin{align}
 \label{eq_4.9}
 W_\tau
 =
 A_z
 W
 +
 W
 +
 N,
 \end{align}
 here $N$ is defined by
 \begin{align*}
 N(W)
 =
 |\kappa+W|^{p-1}
 (\kappa+W)
 -
 \kappa^p
 -
 p\kappa^{p-1}W.
 \end{align*}
 We take an inner product
 $\langle\cdot,H_0\rangle_\rho$ in \eqref{eq_4.9},
 we get
 \begin{align}
 \label{eq_4.10}
 \tfrac{d}{d\tau}
 \langle W,&\,H_0\rangle_\rho
 =
 \langle W,\,H_0\rangle_\rho
 +
 \langle N,H_0 \rangle_\rho.
 \end{align}
 We now define a continuous function $\tilde N(s):\R\to\R$ by
 \begin{align*}
 \tilde N(s)
 =
 \begin{cases}
 |\kappa+s|^{p-1}
 (\kappa+s)
 -
 \kappa^p
 -
 p\kappa^{p-1}s
 & \text{if } s>-\kappa,
 \\
 (p-1)\kappa^p-2p\kappa^{p-1}(\kappa+s)+(\kappa+s)^2
 & \text{if } s<-\kappa.
 \end{cases}
 \end{align*}
 Using \eqref{eq_4.7},
 we can compute the difference between $N(W)$ and $\tilde N(W)$.
 \begin{align}
 \nonumber
 N(W)&-\tilde N(W)
 =
 |\kappa+W|^{p-1}
 (\kappa+W)
 -
 \kappa^p
 -
 p\kappa^{p-1}W
 \\
 \nonumber
 &\quad
 -
 (p-1)\kappa^p+2p\kappa^{p-1}(\kappa+W)-(\kappa+W)^2
 \\
 \nonumber
 &=
 |\kappa+W|^{p-1}
 (\kappa+W)
 +
 p\kappa^{p-1}(\kappa+W)-(\kappa+W)^2
 \\
 \nonumber
 &=
 |w|^{p-1}
 w
 +
 p\kappa^{p-1}w
 -
 w^2
 \\
 \nonumber
 &>
 -
 e^{-\frac{3p\tau}{4(p-1)}}
 -
 p\kappa^{p-1}
 e^{-\frac{3\tau}{4(p-1)}}
 -
 e^{-\frac{3p\tau}{2(p-1)}}
 \\
 \label{eq_4.11}
 &>
 -2p\kappa^{p-1}
 e^{-\frac{3\tau}{4(p-1)}}
 \end{align}
 for
 $z\in\Omega=\{z\in\R^n;W(z,\tau)<-\kappa\}$ and large $\tau$.
 Therefore
 \eqref{eq_4.10} - \eqref{eq_4.11} imply
 \begin{align*}
 \tfrac{d}{d\tau}
 \langle W,H_0\rangle_\rho
 >
 \langle W,H_0 \rangle_\rho
 +
 \
 \langle \tilde{N}(W),H_0 \rangle_\rho
 -
 2p\kappa^{p-1}
 e^{-\frac{3\tau}{4(p-1)}}.
 \end{align*}
 We easily check that $\tilde N(s)$ is convex on $\R$.
 Put $\sigma=\int_{\R^n}\rho dz$.
 We apply the Jensen inequality to get
 \begin{align}
 \nonumber
 \tfrac{d}{d\tau}
 \langle W,H_0\rangle_\rho
 &>
 \langle W,H_0 \rangle_\rho
 +
 H_0
 \sigma
 \int_{\R^n}\tilde{N}(W)\tfrac{\rho dz}{\sigma}
 -
 \langle
 2p\kappa^{p-1}
 e^{-\frac{3\tau}{4(p-1)}},H_0\rangle_\rho
 \\
 \nonumber
 &>
 \langle W,H_0 \rangle_\rho
 +
 H_0
 \sigma
 \tilde N\left(\int_{\R^n}W\tfrac{\rho dz}{\sigma}\right)
 -
 \tfrac{2p\kappa^{p-1}}{H_0}
 e^{-\frac{3\tau}{4(p-1)}}
 \\
 \label{eq_4.12}
 &>
 \langle W,H_0 \rangle_\rho
 +
 H_0
 \sigma
 \tilde N\left(\tfrac{\langle W,H_0\rangle_\rho}{\sigma H_0}\right)
 -
 \tfrac{2p\kappa^{p-1}}{H_0}
 e^{-\frac{3\tau}{4(p-1)}}.
 \end{align}
 Since $\tilde N(s)>0$ for $s>0$,
 it holds from \eqref{eq_4.12} that
 \[
 \tfrac{d}{d\tau}
 \langle W(\tau),H_0 \rangle_\rho
 >
 \langle W(\tau),H_0 \rangle_\rho
 -
 \tfrac{2p\kappa^{p-1}}{H_0}
 e^{-\frac{3\tau}{4(p-1)}}
 \qquad
 \text{for }
 \tau>\tau_1.
 \]
 Integrating over $\tau\in(\tau_1,\tau)$,
 we get
 \begin{align*}
 \langle W(\tau)&,H_0 \rangle_\rho
 >
 e^{\tau-\tau_1}
 (
 \langle W(\tau_1),H_0 \rangle_\rho
 -
 \tfrac{2p\kappa^{p-1}}{H_0}
 \tfrac{e^{-\frac{3\tau_1}{4(p-1)}}}{1+\frac{3}{4(p-1)}}
 )
 \\
 &>
 e^{\tau-\tau_1}
 (
 \langle W(\tau_1),H_0 \rangle_\rho
 -
 \tfrac{2p\kappa^{p-1}}{H_0}
 e^{-\frac{3\tau_1}{4(p-1)}}
 )
 \qquad
 \text{for }
 \tau>\tau_1.
 \end{align*}
 Hence
 we deduce from \eqref{eq_4.8} that
 $\langle W(\tau),H_0 \rangle_\rho
 >\frac{2p\kappa^{p-1}}{H_0}e^{-\frac{3\tau_1}{4(p-1)}}e^{\tau-\tau_1}$
 for $\tau>\tau_1$.
 Since $\tilde N(s)>\frac{s^p}{2}$ for large $s>1$,
 from \eqref{eq_4.12},
 we conclude that
 there exists $\tau_{\sf b}>\tau_1$ such that
 \begin{align*}
 \lim_{\tau\to\tau_{\sf b}}
 \langle W(\tau),H_0 \rangle_\rho
 =
 \infty.
 \end{align*}
 The proof is finished. 
 \end{proof}

 We next consider
 \begin{align}
 \label{eq_4.13}
 w_\tau
 =
 \Delta_zw
 -
 \tfrac{z}{2}
 \cdot
 \nabla_zw
 -
 \tfrac{w}{p-1}
 +
 |w|^{p-1}
 w&
 +
 {\bf c}(\tau)
 \cdot
 \nabla_zw
 \\
 \nonumber
 &
 \text{for }
 z\in\R^n,\ \tau>\tau_0.
 \end{align}
 We here assume that
 \begin{itemize}
 \item $\tau_2>\tau_1>\tau_0$,
 \item ${\bf c}(\tau)\in C([\tau_0,\tau_2];\R^n)\cap L^\infty(\tau_0,\tau_2)$.
 \end{itemize}
 \begin{pro}\label{proposition_4.3}
 Assume $n\geq1$ and $p>1$.
 Let
 $w(z,\tau)$
 be a solution of \eqref{eq_4.13}
 satisfying
 \begin{align}
 \label{eq_4.14}
 \langle w(z,&\,\tau),H_{1,J}(z)\rangle_\rho=0
 \qquad \text{\rm for any }
 \tau\in(\tau_1,\tau_2)
 \text{ \rm and all } J.
 \end{align}
 We decompose $W(z,\tau)=w(z,\tau)-\kappa$ as
 \begin{align*}
 w(z,\tau)-\kappa
 &=
 b_0(\tau)
 H_0(z)
 +
 {\bf b}_2(\tau)
 \cdot
 {\bf H}_2(z)
 +
 w^\bot(z,\tau),
 \end{align*}
 here the symbol ${}^\bot$ implies
 \begin{align*}
 \langle w^\bot(z), H_0(z)\rangle_\rho
 &=
 \langle w^\bot(z), H_{1,J}(z)\rangle_\rho
 =
 \langle w^\bot(z), H_{2,IJ}(z)\rangle_\rho
 \\
 &=0
 \qquad
 \text{\rm for all } J,\ (I,J).
 \end{align*}
 Assume that
 \begin{enumerate}[\rm({k}1) ]
 \setlength{\leftskip}{5mm}
 \item[\rm({k}0) ] 
 $|b_0(\tau)|<\frac{4{\nu}^2}{pH_0\tau}$
 \quad {\rm for }
 $\tau\in[\tau_1,\tau_2]$,
 \item 
 $|\frac{1}{b_{2,JJ}(\tau_1)}+c_p\tau_1|<{\nu}c_p\tau_1$
 \quad
 {\rm for all} $J$,
 \item 
 $|b_{2,IJ}(\tau_1)|<{\nu}^2\tau_1^{-1}$
 \quad
 {\rm and all} $I<J$,
 \item 
 $\|w^\bot(\tau_1)\|_\rho<{\nu}^4\tau_1^{-1}$,
 \item 
 $\|\pa_{z_J}w^\bot(\tau_1)\|_\rho<{\nu}^4\tau_1^{-1}$
 \quad
 {\rm for all} $J$,
 \item 
 $-e^{-\frac{3\tau_1}{4(p-1)}}<w(z,\tau_1)<\kappa+\tau_1^{-\frac{1}{2}}$
 \quad
 {\rm for}
 $z\in\R^n$,
 \item 
 $w(z,\tau_1)<\kappa$
 \quad {\rm for} $|z|>2\sqrt{n}$,
 \item 
 $\dis\sup_{|z|<8R_1}|w(z,\tau_1)-\kappa|<
 \sup_{|z|<8R_1}
 \tfrac{4|\sum_{J=1}^nH_{2,JJ}(z)|}{c_p\tau_1}$,
 \item 
 $\dis\sup_{|z|<16\sqrt{n}}|w^\bot(z,\tau_1)|<{\nu}^4\tau_1^{-1}$.
 \end{enumerate}
 There exists ${\nu}_1\in(0,1)$
 such that
 for any $\nu\in(0,\nu_1)$
 there exist $R_1=R_1(\nu)>1$
 and
 $s_1=s_1(\nu,R_1)>1$ such that
 if
 $\tau_1>s_1$,
 then
 it holds that
 \begin{enumerate}[\rm({K}1)]
 \setlength{\leftskip}{5mm}
 \item 
 $|\frac{1}{b_{2,JJ}(\tau)}+c_p\tau|<{\nu}c_p\tau$
 \quad {\rm for} $\tau\in(\tau_1,\tau_2)$
 {\rm and all} $J$,
 \item 
 $|b_{2,IJ}(\tau)|<{\nu}^2\tau^{-1}$
 \quad {\rm for} $\tau\in(\tau_1,\tau_2)$
 {\rm and all} $I<J$,
 \item 
 $\|w^\bot(\tau)\|_\rho<2{\nu}^4\tau^{-1}$
 \quad {\rm for} $\tau\in(\tau_1,\tau_2)$,
 \item 
 $\|\pa_{z_J}w^\bot(\tau)\|_\rho<2{\nu}^4\tau^{-1}$
 \quad {\rm for} $\tau\in(\tau_1,\tau_2)$
 {\rm and all} $J$,
 \item 
 $-e^{-\frac{3\tau}{4(p-1)}}<w(z,\tau)<\kappa+\tau^{-\frac{1}{2}}$
 \quad {\rm for}
 $(z,\tau)\in\R^n\times[\tau_1,\tau_2]$,
 \item 
 $w(z,\tau)<\kappa$
 \quad
 {\rm for} $|z|>2\sqrt{n}$
 {\rm and} $\tau\in[\tau_1,\tau_2+\delta_1]$
 {\rm for some} $\delta_1\in(0,1)$
 \\
 {\rm(}$\delta_1$ may depend on $w(z,\tau)${\rm)},
 \item 
 $\dis\sup_{|z|<4R_1}|w(z,\tau)-\kappa|<\tau^{-\frac{3}{4}}$
 \quad {\rm for} $\tau\in(\tau_1,\tau_2)$,
 \item 
 $\dis\sup_{|z|<8\sqrt{n}}|w^\bot(z,\tau)|<{\nu}^2\tau^{-1}$
 \quad {\rm for} $\tau\in(\tau_1,\tau_2)$.
 \end{enumerate}
 \end{pro}
 \begin{proof}
 Let $\nu\in(0,\nu_1)$.
 Assume that
 there exists
 $\tilde\tau_2\in(\tau_1,\tau_2)$
 such that
 \begin{enumerate}[\rm({l}1)]
 \setlength{\leftskip}{5mm}
 \item 
 $|\frac{1}{b_{2,JJ}(\tau)}+c_p\tau|<2{\nu}c_p\tau$
 \quad {\rm for} $\tau\in(\tau_1,\tilde\tau_2)$
 {\rm and all} $J$,
 \item 
 $|b_{2,IJ}(\tau)|<2{\nu}^2\tau^{-1}$
 \quad {\rm for} $\tau\in(\tau_1,\tilde\tau_2)$
 {\rm and all} $I<J$,
 \item 
 $\|w^\bot(\tau)\|_\rho<2{\nu}^4\tau^{-1}$
 \quad {\rm for} $\tau\in(\tau_1,\tilde\tau_2)$,
 \item 
 $\|\pa_{z_J}w^\bot(\tau)\|_\rho<2{\nu}^4\tau^{-1}$
 \quad {\rm for} $\tau\in(\tau_1,\tilde\tau_2)$
 {\rm and all} $J$,
 \item 
 $-e^{-\frac{3\tau}{4(p-1)}}<w(z,\tau)<\kappa+2\tau^{-\frac{1}{2}}$
 \quad {\rm for}
 $(z,\tau)\in\R^n\times[\tau_1,\tilde\tau_2]$,
 \item 
 $w(z,\tau)<\kappa$
 \quad
 {\rm for} $|z|>2\sqrt{n}$
 {\rm and} $\tau\in[\tau_1,\tilde\tau_2]$.
 \end{enumerate}
 We now show
 \begin{enumerate}[\rm({L}1)]
 \setlength{\leftskip}{5mm}
 \item 
 $|\frac{1}{b_{2,JJ}(\tau)}+c_p\tau|<{\nu}c_p\tau$
 \quad {\rm for} $\tau\in(\tau_1,\tilde\tau_2)$
 {\rm and all} $J$,
 \item 
 $|b_{2,IJ}(\tau)|<{\nu}^2\tau^{-1}$
 \quad {\rm for} $\tau\in(\tau_1,\tilde\tau_2)$
 {\rm and all} $I<J$,
 \item 
 $\|w^\bot(\tau)\|_\rho<\frac{3}{2}{\nu}^4\tau^{-1}$
 \quad {\rm for} $\tau\in(\tau_1,\tilde\tau_2)$,
 \item 
 $\|\pa_{z_J}w^\bot(\tau)\|_\rho<2{\nu}^4\tau^{-1}$
 \quad {\rm for} $\tau\in(\tau_1,\tilde\tau_2)$
 {\rm and all} $J$,
 \item 
 $-e^{-\frac{3\tau}{4(p-1)}}<w(z,\tau)<\kappa+\tau^{-\frac{1}{2}}$
 \quad {\rm for} $(z,\tau)\in\R^n\times(\tau_1,\tilde\tau_2)$,
 \item 
 there exists $\delta_1\in(0,1)$ such that
 $w(z,\tau)<\kappa$
 {\rm for} $|z|>2\sqrt{n}$
 \text{\rm and} $\tau\in(\tau_1,\tilde\tau_2+\delta_1)$,
 \item 
 $\dis\sup_{|z|<4R_1}|w(z,\tau)-\kappa|<\tfrac{1}{2}\tau^{-\frac{3}{4}}$
 \quad {\rm for} $(z,\tau)\in\R^n\times(\tau_1,\tilde\tau_2)$,
 \item 
 $\dis\sup_{|z|<8\sqrt{n}}|w^\bot(z,\tau)|<\tfrac{{\nu}^2}{2}\tau^{-1}$
 \quad {\rm for} $(z,\tau)\in\R^n\times(\tau_1,\tilde\tau_2)$.
 \end{enumerate}
 If (L1) - (L8) are obtained from the assumptions (l1) - (l6),
 the proof of Proposition \ref{proposition_4.3} is completed.
 \begin{itemize}
 \item 
 Throughout this proof,
 we denote by $C$ a generic positive constant
 which may vary from line to line.
 The constant $C$ does not depend on $\nu_1,\nu,R_1,s_1,\tau_1,\tau_2$,
 \item 
 Furthermore
 we assume that
 $\nu_1\in(0,1)$ is small enough,
 and
 $R_1,s_1$ are large enough
 in this proof.
 \end{itemize}
 Put $W(z,\tau)=w(z,\tau)-\kappa$.
 The function $W(z,\tau)$ satisfies
 (see \eqref{eq_4.13})
 \begin{align}
 \label{eq_4.15}
 W_\tau
 =
 \Delta_zW
 -
 \tfrac{z}{2}
 \cdot
 \nabla_zW
 +
 W
 +
 N
 +
 {\bf c}(\tau)
 \cdot
 \nabla_zW,
 \end{align}
 here
 \begin{align*}
 N
 =
 |\kappa+W|^{p-1}(\kappa+W)-\kappa^p-p\kappa^{p-1}W.
 \end{align*}
 To derive estimates of ${\bf c}(\tau)$,
 we take an inner product $\langle\cdot,H_{1,J}\rangle_\rho$ in \eqref{eq_4.15}.
 From \eqref{eq_4.14},
 we get
 \begin{align}
 \label{eq_4.16}
 {\bf c}(\tau)\cdot
 \langle\nabla_zW,H_{1,J}\rangle_\rho
 =
 -
 \langle N,H_{1,J}\rangle_\rho.
 \end{align}
 We note that $|W(z,\tau)|<2\kappa$ for $\tau\in(\tau_1,\tilde\tau_2)$
 (see (l5)).
 Hence
 Lemma \ref{lemma_A.3} implies
 $|N|<CW^2$ for $\tau\in(\tau_1,\tilde\tau_2)$.
 From Lemma \ref{lem_3.1},
 the right hand side of \eqref{eq_4.16}
 is estimated as
 \begin{align}
 \nonumber
 |\langle N,\,H_{1,J} \rangle_\rho|
 &<
 C
 \langle
 W^2,|H_{1,J}|
 \rangle_\rho
 <
 C
 \langle
 W^2,|z|
 \rangle_\rho
 \\
 \label{eq_4.17}
 &<
 C
 \|W\|_\rho
 \|W\|_{H_\rho^1(\R^n)}.
 \end{align}
 Furthermore
 from (k0) and (l1) - (l4),
 we observe that
 \begin{align}
 \label{eq_4.18}
 \|W(z,\tau)\|_{H_\rho^1(\R^n)}
 <\tfrac{2n\|{\bf H}_{2,J}\|_{H_\rho^1(\R^n)}}{c_p\tau}
 \qquad
 \text{for }
 \tau\in(\tau_1,\tilde\tau_2).
 \end{align}
 Hence
 from \eqref{eq_4.17} - \eqref{eq_4.18},
 we obtain
 \begin{align}
 \label{eq_4.19}
 |\langle N,H_{1,J} \rangle_\rho|
 <
 C\tau^{-2}
 \qquad
 \text{for }
 \tau\in(\tau_1,\tilde\tau_2).
 \end{align}
 We now compute the left hand side of \eqref{eq_4.16}.
 A direct computation shows
 \begin{align*}
 \langle{\pa_{z_I}}W,H_{1,J}\rangle_\rho
 &=
 \int_{\R^n}
 \pa_{z_I}W(z,\tau)
 \cdot
 H_{1,J}(z)
 e^{-\frac{|z|^2}{4}}
 dz
 \\
 &=
 -
 \int_{\R^n}
 W(z,\tau)
 (
 \pa_{z_I}
 H_{1,J}(z)
 -
 \tfrac{z_I}{2}
 H_{1,J}(z)
 )
 e^{-\frac{|z|^2}{4}}
 dz
 \\
 &=
 \langle W,
 -\tfrac{{\sf k}_0^n}{\sqrt{2}}\delta_{IJ}
 +\tfrac{z_I}{2}\tfrac{{\sf k}_0^n}{\sqrt{2}}z_J\rangle_\rho
 \\
 &=
 \begin{cases}
 \langle W,H_{2,JJ}\rangle_\rho
 & \text{if } I=J,
 \\
 \tfrac{1}{\sqrt{2}}
 \langle W,H_{2,IJ}\rangle_\rho
 & \text{if } I<J.
 \end{cases}
 \end{align*}
 Hence
 it holds that
 \begin{align}
 \nonumber
 {\bf c}(\tau)\cdot
 \langle\nabla_zW,H_{1,1}\rangle_\rho
 &=
 c_1(\tau)
 b_{2,11}(\tau)
 +
 \tfrac{1}{\sqrt{2}}
 \sum_{I=2}^n
 c_I(\tau)b_{2,I1}(\tau),
 \\
 \label{eq_4.20}
 \vdots
 \hspace{20mm}
 &=
 \hspace{20mm}
 \vdots
 \\
 \nonumber
 {\bf c}(\tau)\cdot
 \langle\nabla_zW,H_{1,n}\rangle_\rho
 &=
 c_n(\tau)
 b_{2,nn}(\tau)
 +
 \tfrac{1}{\sqrt{2}}
 \sum_{I=1}^{n-1}
 c_I(\tau)b_{2,In}(\tau).
 \end{align}
 Relations \eqref{eq_4.20} and \eqref{eq_4.16} imply
 \begin{align*}
 \begin{pmatrix}
 b_{2,11} & \frac{b_{2,21}}{\sqrt{2}} & \frac{b_{2,31}}{\sqrt{2}} &
 \cdots & \frac{b_{2,n1}}{\sqrt{2}}
 \\
 \frac{b_{2,12}}{\sqrt{2}} & b_{2,22} & \frac{b_{2,32}}{\sqrt{2}} &
 \cdots & \frac{b_{2,n2}}{\sqrt{2}}
 \\
 \vdots & \vdots & \vdots & \cdots & \vdots
 \\
 \frac{b_{2,1n}}{\sqrt{2}} & \frac{b_{2,2n}}{\sqrt{2}} & \frac{b_{2,3n}}{\sqrt{2}} &
 \cdots & b_{2,nn}
 \end{pmatrix}
 {\bf c}(\tau)
 =
 -
 \begin{pmatrix}
 \langle N,H_{1,1}\rangle_\rho
 \\
 \langle N,H_{1,2}\rangle_\rho
 \\
 \vdots
 \\
 \langle N,H_{1,n}\rangle_\rho
 \end{pmatrix}.
 \end{align*}
 Hence from (l1) - (l2) and \eqref{eq_4.19},
 there exists $C>0$ such that
 \begin{align}
 \label{eq_4.21}
 |{\bf c}(\tau)|
 <
 C\tau^{-1}
 \qquad
 \text{for }
 \tau\in(\tau_1,\tilde\tau_2).
 \end{align}
 We now confirm (L1) - (L2).
 From \eqref{eq_4.15},
 we get
 \begin{align}
 \nonumber
 \tfrac{d}{d\tau}
 &b_{2,IJ}
 =
 \langle N,H_{2IJ} \rangle_\rho
 +
 {\bf c}(\tau)
 \cdot
 \langle \nabla_zW,H_{2IJ} \rangle_\rho
 \\
 \label{eq_4.22}
 &=
 \langle N,H_{2IJ} \rangle_\rho
 -
 \sum_{i=1}^n
 c_i(\tau)
 \langle W,\pa_{z_i}H_{2IJ}-\tfrac{z_i}{2}H_{2IJ}) \rangle_\rho.
 \end{align}
 Since
 $\pa_{z_i}H_{2IJ}\in V_1$
 and
 $\tfrac{z_i}{2}H_{2IJ}\in V_3$
 (see Section \ref{sec_2} for the definition of $V_k$),
 we verify that
 \begin{align}
 \nonumber
 |
 \sum_{i=1}^n
 &c_i(\tau)
 \underbrace{
 \langle W,\pa_{z_i}H_{2IJ}-\tfrac{z_i}{2}H_{2IJ}) \rangle_\rho
 }_{=\langle W,-\tfrac{z_i}{2}H_{2IJ}) \rangle_\rho}
 |
 \\
 \label{eq_4.23}
 &<
 C
 |{\bf c}(\tau)|
 \|W^\bot\|_\rho
 <
 C{\nu}^4
 \tau^{-2}
 \qquad
 \text{for }
 \tau\in(\tau_1,\tilde\tau_2).
 \end{align}
 To treat nonlinear terms in \eqref{eq_4.22},
 we here provide pointwise estimates for $W(z,\tau)$ in $|z|<4R_1$.
 The constant $R_1$ is the same as in (k7).
 Put
 \begin{align}
 \label{eq_4.24}
 W_{\bf a}(z,\tau)
 =
 W(z+e^{\frac{\tau-\tau_1}{2}}{\bf a},\tau)
 \qquad
 (|{\bf a}|<4R_1).
 \end{align}
 This function $W_{\bf a}(z,\tau)$ satisfies
 \begin{align*}
 \pa_\tau W_{\bf a}
 &=
 A_zW_{\bf a}
 +
 \underbrace{
 W_{\bf a}
 +
 \tfrac{N(W_{\bf a})}{W_{\bf a}}
 W_{\bf a}
 }_{=P(z,\tau)W_{\bf a}}
 +
 {\bf c}(\tau)\cdot\nabla_zW_{\bf a}
 \\
 &=
 A_zW_{\bf a}
 +
 P(z,\tau)
 W_{\bf a}
 +
 {\bf c}(\tau)\cdot\nabla_zW_{\bf a}.
 \end{align*}
 Since $|W_{\bf a}(z,\tau)|<2\kappa$ for $\tau\in(\tau_1,\tilde\tau_2)$ (see (l5)),
 we note that $|P(z,\tau)|<C$ for $\tau\in(\tau_1,\tilde\tau_2)$.
 A local boundary estimate for parabolic equations
 implies that there exists $C>0$ independent of ${\bf a}\in\R^n$ such that
 \begin{align}
 \nonumber
 &\sup_{\tau\in(\tau_1,\tau_1+1)}
 \sup_{|z|<1}
 |W_{\bf a}(z,\tau)|
 \\
 \label{eq_4.25}
 &<
 C
 (
 \sup_{\tau\in(\tau_1,\tau_1+1)}
 \|W_{\bf a}(z,\tau)\|_{L_z^2(|z|<2)}
 +
 \sup_{|z|<2}
 |W_{\bf a}(z,\tau_1)|
 ).
 \end{align}
 From the definition of $W_{\bf a}(z,\tau)$ (see \eqref{eq_4.24}) and (k7),
 we see that
 \begin{align}
 \nonumber
 &
 \sup_{|z|<2}
 |W_{\bf a}(z,\tau_1)|
 =
 \sup_{|z|<2}
 |W(z+{\bf a},\tau_1)|
 \\
 &<
 \sup_{|z|<8R_1}
 |W(z,\tau_1)|
 \label{eq_4.26}
 <
 \sup_{|z|<8R_1}
 \tfrac{4|\sum_{J=1}^nH_{2,JJ}(z)|}{c_p\tau_1}
 <
 \tfrac{CR_1^2}{\tau_1}.
 \end{align} 
 Furthermore
 we observe that
 \begin{align}
 \nonumber
 \|&W_{\bf a}(z,\tau)\|_{L_z^2(|z|<2)}^2
 =
 \int_{|z|<2}
 W_{\bf a}(z,\tau)^2
 dz
 \\
 \nonumber
 &=
 \int_{|z|<2}
 W(z+e^\frac{\tau-\tau_1}{2}{\bf a},\tau)^2
 e^{\frac{|z+e^\frac{\tau-\tau_1}{2}{\bf a}|^2}{4}}
 e^{-\frac{|z+e^\frac{\tau-\tau_1}{2}{\bf a}|^2}{4}}
 dz 
 \\
 \nonumber
 &<
 (
 \sup_{|z|<2}
 e^{\frac{|z+e^\frac{\tau-\tau_1}{2}{\bf a}|^2}{4}}
 )
 \int_{|z|<2}
 W(z+e^\frac{\tau-\tau_1}{2}{\bf a},\tau)^2
 e^{-\frac{|z+e^\frac{\tau-\tau_1}{2}{\bf a}|^2}{4}}
 dz 
 \\
 \nonumber
 &<
 e^{\frac{{e^{\tau-\tau_1}|{\bf a}|^2}}{2}}
 \|
 W(z,\tau)
 \|_\rho^2
 \\
 \label{eq_4.27}
 &<
 C
 e^{8eR_1^2}
 \tau^{-1}
 \qquad
 \text{for }
 \tau\in(\tau_1,\tau_1+1).
 \end{align}
 Hence if
 $s_1,R_1$ satisfy
 \begin{enumerate}[(n1)]
 \item 
 $R_1^2s_1^{-\frac{1}{8}}<1$
 and
 \item 
 $e^{8eR_1^2}s_1^{-\frac{1}{8}}<1$,
 \end{enumerate}
 then
 it follow from \eqref{eq_4.25} - \eqref{eq_4.27}
 that
 \begin{align*}
 \sup_{\tau\in(\tau_1,\tau_1+1)}
 \sup_{|z|<2}
 |W_{\bf a}(z,\tau)|
 <
 Cs_1^{\frac{1}{8}}
 \tau^{-1}
 <
 C
 \tau^{-\frac{7}{8}}
 <
 \tfrac{1}{4}
 \tau^{-\frac{3}{4}}.
 \end{align*}
 We here used $\tau>s_1\gg1$.
 Therefore
 from the definition of $W_{\bf a}(z,\tau)$,
 we have
 \begin{align}
 \label{eq_4.28}
 \sup_{\tau\in(\tau_1,\tau_1+1)}
 \sup_{|z|<4R_1}
 |W(z,\tau)|
 <
 \tfrac{1}{4}
 \tau^{-\frac{3}{4}}.
 \end{align}
 In the same way,
 we can obtain estimates for $\tau\in(\tau_1+1,\tilde\tau_2)$.
 In fact,
 for $\sigma\in(\tau_1+1,\tilde\tau_2)$,
 we put
 $W_{\bf a}(z,\tau)=W(z+e^\frac{\tau-\sigma+1}{2}{\bf a},\tau)$
 ($|{\bf a}|<4R_1$)
 instead of \eqref{eq_4.24}.
 We can verify that
 \begin{align*}
 \nonumber
 \sup_{\tau\in(\sigma-\frac{1}{2},\sigma)}
 \sup_{|z|<1}
 |W_{\bf a}(z,\tau)|
 &<
 C
 \sup_{\tau\in(\sigma-1,\sigma)}
 \|W_{\bf a}(z,\tau)\|_{L_z^2(|z|<2)}
 \\
 &<
 C
 e^{8eR_1^2}
 \tau^{-1}
 <
 \tfrac{1}{4}
 \tau^{-\frac{3}{4}}.
 \end{align*}
 In the last inequality,
 we used (n2).
 This estimate implies
 \begin{align}
 \label{eq_4.29}
 \sup_{|z|<4R_1}
 |W(z,\tau)|
 <
 \tfrac{1}{4}
 \tau^{-\frac{3}{4}}
 \qquad
 \text{for }
 \tau\in(\tau_1+1,\tilde\tau_2).
 \end{align}
 For both cases (see \eqref{eq_4.28} - \eqref{eq_4.29}),
 if (n1) - (n2) are true,
 it holds that
 \begin{align}
 \label{eq_4.30}
 \sup_{|z|<4R_1}|w(z,\tau)-\kappa|
 <
 \tfrac{1}{4}
 \tau^{-\frac{3}{4}}
 \qquad
 \text{for }
 \tau\in(\tau_1,\tilde\tau_2).
 \end{align}
 This implies (L7).
 From now,
 we always assume (n1) - (n2).
 From Lemma \ref{lemma_A.4} \eqref{equation_A.9} - \eqref{equation_A.10},
 it holds that
 \begin{align}
 \nonumber
 |
 \langle N,&\,H_{2,JJ}\rangle_\rho
 -
 c_p
 b_{2,JJ}^2
 -
 \tfrac{pH_0}{\kappa}
 b_0b_{2,JJ}
 |
 \\
 \nonumber
 &<
 C
 (
 \underbrace{
 |{\bf b}_{2,IJ}|^2
 +
 |{\bf b}_2|\cdot\|W^\bot\|_\rho
 +
 \|W^\bot\|_{H_\rho^1}^2
 }_{<C{\nu}^4\tau^{-2}}
 )
 \\
 \nonumber
 &\quad
 +
 |\langle \tfrac{pW^2}{2\kappa}{\bf 1}_{|z|>R_1},H_{2,JJ}\rangle_\rho|
 +
 |\langle (N-\tfrac{pW^2}{2\kappa}){\bf 1}_{|z|<R_1},H_{2,JJ}\rangle_\rho|
 \\
 \label{eq_4.31}
 &\quad
 +
 \underbrace{
 |\langle N{\bf 1}_{|z|>R_1},H_{2,JJ}\rangle_\rho|
 }_{<C\langle W^2{\bf 1}_{|z|>R_1},|H_{2,JJ}|\rangle_\rho}
 \qquad
 \text{for }
 \tau\in(\tau_1,\tilde\tau_2)
 \end{align}
 and for $I<J$
 \begin{align}
 \nonumber
 |
 \langle N&,H_{2,IJ}\rangle_\rho
 -
 c_p
 (b_{2,II}+b_{2,JJ})
 b_{2,IJ}
 -
 \tfrac{pH_0}{\kappa}
 b_0b_{2,IJ}
 |
 \\
 \nonumber
 &<
 C
 (
 \underbrace{
 |{\bf b}_{2,IJ}|^2
 +
 |{\bf b}_2|\cdot\|W^\bot\|_\rho
 +
 \|W^\bot\|_{H_\rho^1}^2
 }_{<C{\nu}^4\tau^{-2}}
 )
 \\
 \nonumber
 &\quad
 +
 |\langle \tfrac{pW^2}{2\kappa}{\bf 1}_{|z|>R_1},H_{2,IJ}\rangle_\rho|
 +
 |\langle (N-\tfrac{pW^2}{2\kappa}){\bf 1}_{|z|<R_1},H_{2,IJ}\rangle_\rho|
 \\
 \label{eq_4.32}
 &\quad
 +
 \underbrace{
 |\langle N{\bf 1}_{|z|>R_1},H_{2,IJ}\rangle_\rho|
 }_{<C\langle W^2{\bf 1}_{|z|>R_1},|H_{2,IJ}|\rangle_\rho}
 \qquad
 \text{for }
 \tau\in(\tau_1,\tilde\tau_2).
 \end{align}
 Since
 $W^2<2(b_0^2H_0^2+|{\bf b}_2|^2\cdot|{\bf H}_2|^2)+2|W^\bot|^2$,
 we get from Lemma \ref{lem_3.1} that
 \begin{align}
 \nonumber
 |\langle &W^2{\bf 1}_{|z|>R_1},H_{2,IJ} \rangle_\rho|
 \\
 \nonumber
 &<
 \tfrac{C}{\tau^2}
 \langle (1+|z|^2){\bf 1}_{|z|>R_1},|z|^2 \rangle_\rho
 +
 C\langle (W^\bot)^2{\bf 1}_{|z|>R_1},|z|^2 \rangle_\rho
 \\
 \nonumber
 &<
 \tfrac{C}{\tau^2}
 \langle (1+|z|^2){\bf 1}_{|z|>R_1},|z|^2 \rangle_\rho
 +
 C
 \|W^\bot\|_{H_\rho^1(\R^n)}^2
 \\
 \nonumber
 &<
 \tfrac{C}{\tau^2}R_1^{n+2}e^{-\frac{R_1^2}{4}}
 +
 \tfrac{C{\nu}^8}{\tau^2}
 \\
 \label{eq_4.33}
 &<
 \tfrac{C}{\tau^2}
 (
 e^{-\frac{R_1^2}{8}}
 +
 {\nu}^8
 )
 \qquad
 \text{for }
 \tau\in(\tau_1,\tilde\tau_2).
 \end{align}
 Furthermore
 we get from \eqref{eq_4.30} that
 \begin{align}
 \nonumber
 |\langle (N&-\tfrac{pW^2}{2\kappa}){\bf 1}_{|z|<R_1},H_{2,IJ}\rangle_\rho|
 <
 C\langle |W|^3{\bf 1}_{|z|<R_1},|z|^2\rangle_\rho
 \\
 \nonumber
 &<
 C\tau^{-\frac{1}{2}}
 \langle |W|^2{\bf 1}_{|z|<R_1},|z|^2\rangle_\rho
 <
 C\tau^{-\frac{1}{2}}
 \|W\|_{H_\rho^1(\R^n)}^2
 \\
 \label{eq_4.34}
 &<
 C\tau^{-\frac{5}{2}}
 \qquad
 \text{for }
 \tau\in(\tau_1,\tilde\tau_2).
 \end{align}
 Applying \eqref{eq_4.33} - \eqref{eq_4.34} in \eqref{eq_4.31} - \eqref{eq_4.32},
 we obtain
 \begin{align}
 \label{eq_4.35}
 |
 \langle N,H_{2,JJ}\rangle_\rho
 &-
 c_p
 b_{2,JJ}^2
 -
 \tfrac{pH_0}{\kappa}
 b_0b_{2,JJ}
 |
 \\
 \nonumber
 &<
 \tfrac{C}{\tau^2}
 (
 {\nu}^4
 +
 e^{-\frac{R_1^2}{8}}
 )
 \qquad
 \text{for }
 \tau\in(\tau_1,\tilde\tau_2)
 \end{align}
 and for $I<J$
 \begin{align}
 \label{eq_4.36}
 |
 \langle N,H_{2,IJ}\rangle_\rho
 &-
 c_p
 (b_{2,II}+b_{2,JJ})
 b_{2,IJ}
 -
 \tfrac{pH_0}{\kappa}
 b_0b_{2,IJ}
 |
 \\
 \nonumber
 &<
 \tfrac{C}{\tau^2}
 (
 {\nu}^4
 +
 e^{-\frac{R_1^2}{8}}
 )
 \qquad
 \text{for }
 \tau\in(\tau_1,\tilde\tau_2).
 \end{align}
 Here we note from (l1) that
 $|b_{2,II}+b_{2,JJ}+\frac{2}{c_p\tau}|<C{\nu}\tau^{-1}$.
 Therefore
 from (k0) and (l1) - (l2),
 we can write \eqref{eq_4.35} - \eqref{eq_4.36} as
 \begin{align}
 \nonumber
 |
 \langle N,H_{2,JJ}\rangle_\rho
 -
 c_p
 b_{2,JJ}^2
 |
 &<
 \tfrac{C}{\tau^2}
 (
 {\nu}^4
 +
 e^{-\frac{R_1^2}{8}}
 )
 +
 \tfrac{pH_0}{\kappa}
 \underbrace{|b_0b_{2,JJ}|}_{<C{\nu}^2\tau^{-2}}
 \\
 \label{eq_4.37}
 &<
 \tfrac{C}{\tau^2}
 (
 {\nu}^2
 +
 e^{-\frac{R_1^2}{8}}
 )
 \qquad
 \text{for }
 \tau\in(\tau_1,\tilde\tau_2)
 \end{align}
 and for $I<J$
 \begin{align}
 \nonumber
 |
 \langle N,H_{2,IJ}\rangle_\rho
 +
 \tfrac{2}{\tau}
 b_{2,IJ}
 |
 &<
 \tfrac{C}{\tau^2}
 (
 {\nu}
 +
 e^{-\frac{R_1^2}{8}}
 )
 +
 c_p
 \underbrace{
 |b_{2,II}+b_{2,JJ}+\tfrac{2}{c_p\tau}|
 |b_{2,IJ}|
 }_{C{\nu}^3\tau^{-2}}
 \\
 \label{eq_4.38}
 &<
 \tfrac{C}{\tau^2}
 (
 {\nu}^3
 +
 e^{-\frac{R_1^2}{8}}
 ) 
 \qquad
 \text{for }
 \tau\in(\tau_1,\tilde\tau_2).
 \end{align}
 We now assume
 \begin{enumerate}[(n1)]
 \setcounter{enumi}{2}
 \item $e^{-\frac{R_1^2}{8}}<{\nu}^3$.
 \end{enumerate}
 As a consequence,
 we get from \eqref{eq_4.22} - \eqref{eq_4.23}
 and \eqref{eq_4.37}- \eqref{eq_4.38} that
 \begin{align}
 \label{eq_4.39}
 |
 \tfrac{d}{d\tau}b_{2,JJ}
 -
 c_p
 b_{2,JJ}^2
 |
 &<
 C{\nu}^2
 \tau^{-2}
 \qquad
 \text{for }
 \tau\in(\tau_1,\tilde\tau_2),
 \\
 \label{eq_4.40}
 |
 \tfrac{d}{d\tau}b_{2,IJ}
 +
 \tfrac{2}{\tau}
 b_{2,IJ}
 |
 &<
 C{\nu}^3
 \tau^{-2}
 \qquad
 \text{for }
 \tau\in(\tau_1,\tilde\tau_2)
 \text{ and }
 I<J.
 \end{align}
 We note from (l1) that
 \begin{align*}
 -\tfrac{5}{4c_p\tau}
 <
 b_{2,JJ}(\tau)
 <
 -\tfrac{3}{4c_p\tau}
 \qquad
 \text{for }
 \tau\in(\tau_1,\tilde\tau_2).
 \end{align*}
 Hence
 it follows from \eqref{eq_4.39} that
 \begin{align*}
 (c_p-C{\nu}^2)b_{2,JJ}^2
 <
 \tfrac{d}{d\tau}
 b_{2,JJ}
 <
 (c_p+C{\nu}^2)b_{2,JJ}^2
 \qquad
 \text{for }
 \tau\in(\tau_1,\tilde\tau_2).
 \end{align*}
 Integrating both sides,
 we get
 \begin{align*}
 (c_p-C{\nu}^2)
 (\tau&-\tau_1)
 <
 \tfrac{-1}{b_{2,JJ}(\tau)}
 +
 \tfrac{1}{b_{2,JJ}(\tau_1)}
 \\
 &<
 (c_p+C{\nu}^2)
 (\tau-\tau_1)
 \qquad
 \text{for }
 \tau\in(\tau_1,\tilde\tau_2).
 \end{align*}
 Combining (k1),
 we obtain
 for $\tau\in(\tau_1,\tilde\tau_2)$.
 \begin{align}
 \nonumber
 |\tfrac{1}{b_{2,JJ}(\tau)}+c_p\tau|
 &<
 |\tfrac{1}{b_{2,JJ}(\tau_1)}+c_p\tau_1|
 +
 C{\nu}^2
 (\tau-\tau_1)
 \\
 \nonumber
 &<
 {\nu}
 c_p
 \tau_1
 +
 C{\nu}^2
 (\tau-\tau_1)
 \\
 \nonumber
 &<
 {\nu}
 c_p
 \tau
 -
 {\nu}
 c_p
 (\tau-\tau_1)
 +
 C{\nu}^2
 (\tau-\tau_1)
 \\
 \label{eq_4.41}
 &<
 {\nu}
 c_p
 \tau
 -
 {\nu}
 c_p
 (
 1
 -
 \tfrac{C{\nu}}{c_p}
 )
 (\tau-\tau_1).
 \end{align}
 Furthermore
 integrating \eqref{eq_4.40},
 and from (k2),
 we have
 \begin{align}
 \nonumber
 |b_{2,IJ}(\tau)|
 &<
 (\tfrac{\tau_1}{\tau})^2
 |b_{2,IJ}(\tau_1)|
 +
 \int_{\tau_1}^\tau
 (\tfrac{s}{\tau})^2
 \tfrac{C{\nu}^3}{s^2}
 ds
 \\
 \nonumber
 &<
 (\tfrac{\tau_1}{\tau})^2
 \tfrac{{\nu}^2}{\tau_1}
 +
 \tfrac{C{\nu}^3(\tau-\tau_1)}{\tau^2}
 =
 \tfrac{{\nu}^2\tau_1}{\tau^2}
 +
 \tfrac{C{\nu}^3(\tau-\tau_1)}{\tau^2}
 \\
 \nonumber
 &=
 \tfrac{{\nu}^2}{\tau}
 -
 \tfrac{{\nu}^2(\tau-\tau_1)}{\tau^2}
 +
 \tfrac{C{\nu}^3(\tau-\tau_1)}{\tau^2}
 \\
 \label{eq_4.42}
 &=
 \tfrac{{\nu}^2}{\tau}
 -
 \tfrac{{\nu}^2(\tau-\tau_1)}{\tau^2}
 (1-C{\nu})
 \qquad
 \text{for }
 \tau\in(\tau_1,\tilde\tau_2).
 \end{align}
 Estimates \eqref{eq_4.41} and \eqref{eq_4.42} show (L1) - (L2).
 From \eqref{eq_4.15},
 we can derive the equation for $W^\bot$.
 \begin{align}
 \label{eq_4.43}
 W_\tau^\bot
 &=
 A_zW^\bot
 +
 W^\bot
 +
 N
 -
 (N,H_0)H_0
 -
 (N,{\bf H}_2)
 \cdot
 {\bf H}_2
 \\
 \nonumber
 &\quad
 +
 {\bf c}(\tau)
 \cdot
 \nabla_zW^\bot
 -
 \langle
 ({\bf c}(\tau)
 \cdot
 \nabla_zW),
 {\bf H}_2
 \rangle_\rho
 \cdot
 {\bf H}_2.
 \end{align}
 We multiply \eqref{eq_4.43} by $W^\bot$.
 Then we get
 \begin{align}
 \label{eq_4.44}
 \tfrac{1}{2}
 \tfrac{d}{d\tau}
 \|W^\bot\|_\rho^2
 &=
 -
 \|
 \nabla_zW^\bot
 \|_\rho^2
 +
 \|
 W^\bot
 \|_\rho^2
 +
 \langle
 N,W^\bot
 \rangle_\rho
 \\
 \nonumber
 &\quad
 +
 {\bf c}(\tau)
 \cdot
 \langle
 \nabla_zW^\bot,
 W^\bot
 \rangle_\rho.
 \end{align}
 We compute the third term $\langle N,W^\bot\rangle_\rho$
 on the right hand side of \eqref{eq_4.44}.
 \begin{align}
 \nonumber
 &\|N\|_\rho
 <
 C\|W^2\|_\rho
 \\
 \nonumber
 &<
 C\|\underbrace{(b_0H_0+{\bf b}_2\cdot{\bf H}_2)W
 }_{<C\tau^{-1}(1+|z|^2)|W|}\|_\rho
 +
 C\|W^\bot W\|_\rho
 \\
 \label{eq_4.45}
 &<
 \tfrac{C}{\tau}
 \|(1+|z|^2)W\|_\rho
 +
 C\|W^\bot W\|_\rho
 \qquad
 \text{for }
 \tau\in(\tau_1,\tilde\tau_2).
 \end{align}
 As for the first term on the right hand side of \eqref{eq_4.45},
 we see that
 \begin{align}
 \nonumber
 \tfrac{1}{\tau}
 \|&(1+|z|^2)W\|_\rho
 <
 \tfrac{1}{\tau}
 \|(1+|z|^2)(b_0H_0+{\bf b}_2\cdot{\bf H}_2)\|_\rho
 \\
 \nonumber
 &\quad
 +
 \tfrac{1}{\tau}
 \|(1+|z|^2)W^\bot\|_\rho
 \\
 \nonumber
 &<
 \tfrac{C}{\tau^2}
 +
 \tfrac{1}{\tau}
 \|(1+|z|^2)W^\bot{\bf 1}_{|z|<\tau^\frac{1}{4}}\|_\rho
 +
 \tfrac{1}{\tau}
 \|(1+|z|^2)W^\bot{\bf 1}_{|z|>\tau^\frac{1}{4}}\|_\rho
 \\
 \nonumber
 &<
 \tfrac{C}{\tau^2}
 +
 \tfrac{1}{\sqrt\tau}
 \|W^\bot{\bf 1}_{|z|<\tau^\frac{1}{4}}\|_\rho
 +
 \tfrac{1}{\tau}
 \|\underbrace{(1+|z|^2)W}_{<2\kappa(1+|z|^2)}{\bf 1}_{|z|>\tau^\frac{1}{4}}\|_\rho
 \\
 \nonumber
 &\quad
 +
 \tfrac{1}{\tau}
 \|(1+|z|^2)(b_0H_0+{\bf b}_2\cdot{\bf H}_2)
 {\bf 1}_{|z|>\tau^\frac{1}{4}}\|_\rho
 \\
 \label{eq_4.46}
 &<
 \tfrac{C}{\tau^2}
 +
 \tfrac{\|W^\bot\|_\rho}{\sqrt\tau}
 \qquad
 \text{for }
 \tau\in(\tau_1,\tilde\tau_2).
 \end{align}
 We estimate the second term on the right hand side of \eqref{eq_4.45},
 by using \eqref{eq_4.30} and $|W(z,\tau)|<2\kappa$ (see (l5)).
 \begin{align}
 \nonumber
 \|W^\bot W\|_\rho
 &<
 \|W^\bot W{\bf 1}_{|z|<R_1}\|_\rho
 +
 \|W^\bot W{\bf 1}_{|z|>R_1}\|_\rho
 \\
 \nonumber
 &<
 \tfrac{1}{\sqrt\tau}
 \|W^\bot{\bf 1}_{|z|<R_1}\|_\rho
 +
 2\kappa\|W^\bot{\bf 1}_{|z|>R_1}\|_\rho
 \\
 \nonumber
 &<
 \tfrac{\|W^\bot\|_\rho}{\sqrt\tau}
 +
 \tfrac{2\kappa}{R_1}
 \||z|W^\bot{\bf 1}_{|z|>R_1}\|_\rho
 \\
 \label{eq_4.47}
 &<
 \tfrac{\|W^\bot\|_\rho}{\sqrt\tau}
 +
 \tfrac{2\kappa C}{R_1}
 \|W^\bot\|_{H_\rho^1(\R^n)}
 \qquad
 \text{for }
 \tau\in(\tau_1,\tilde\tau_2).
 \end{align}
 In the last inequality,
 we used Lemma \ref{lem_3.1}.
 Hence
 \eqref{eq_4.45} - \eqref{eq_4.47}
 imply
 \begin{align}
 \label{eq_4.48}
 \|N\|_\rho
 <
 \tfrac{C}{\tau^2}
 +
 (
 \tfrac{C}{\sqrt\tau}
 +
 \tfrac{C}{R_1}
 )
 \|W^\bot\|_{H_\rho^1(\R^n)}
 \qquad
 \text{for }
 \tau\in(\tau_1,\tilde\tau_2).
 \end{align}
 The constant $C$ does not depend on $R_1$.
 Therefore
 substituting \eqref{eq_4.21} and \eqref{eq_4.48}
 to \eqref{eq_4.44},
 we get for $\tau\in(\tau_1,\tilde\tau_2)$
 \begin{align}
 \nonumber
 \tfrac{1}{2}
 &\tfrac{d}{d\tau}
 \|W^\bot\|_\rho^2
 <
 -
 \|
 \nabla_zW^\bot
 \|_\rho^2
 +
 \|
 W^\bot
 \|_\rho
 +
 \tfrac{C\|W^\bot\|_\rho}{\tau^2}
 +
 (
 \tfrac{C}{\sqrt\tau}
 \\
 \nonumber
 &\quad
 +
 \tfrac{C}{R_1}
 )
 \|W^\bot\|_{H_\rho^1(\R^n)}
 \|W^\bot\|_\rho
 +
 \tfrac{C}{\tau}
 \|\nabla_zW^\bot\|_\rho
 \|W^\bot\|_\rho
 \\
 \label{eq_4.49}
 &<
 (-1+\tfrac{C}{\sqrt\tau}+\tfrac{C}{R_1})
 \|
 \nabla_zW^\bot
 \|_\rho^2
 +
 (\tfrac{33}{32}+\tfrac{C}{\sqrt\tau}+\tfrac{C}{R_1})
 \|
 W^\bot
 \|_\rho
 +
 \tfrac{C}{\tau^4}.
 \end{align}
 If $s_1,R_1$ are large enough,
 the right hand side of \eqref{eq_4.49} is bounded by
 \begin{align}
 \label{eq_4.50}
 \tfrac{1}{2}
 \tfrac{d}{d\tau}
 \|W^\bot\|_\rho^2
 <
 -
 \tfrac{1}{4}
 \|
 W^\bot
 \|_\rho
 +
 \tfrac{C}{\tau^4}
 \qquad
 \text{for }
 \tau\in(\tau_1,\tilde\tau_2).
 \end{align}
 Integrating \eqref{eq_4.50},
 we get from (k3) that
 \begin{align*}
 \|W^\bot&(\tau)\|_\rho^2
 <
 e^{-\frac{\tau-\tau_1}{2}}
 \|W^\bot(\tau_1)\|_\rho^2
 +
 e^{-\frac{\tau}{2}}
 \int_{\tau_1}^\tau
 e^\frac{s}{2}
 \tfrac{C}{s^4}
 ds
 \\
 &<
 e^{-\frac{\tau-\tau_1}{2}}
 {\nu}^8
 \tau_1^{-2}
 +
 e^{-\frac{\tau}{2}}
 [
 e^\frac{s}{2}
 \tfrac{C}{s^4}
 ]_{s=\tau_1}^{s=\tau}
 \qquad
 \text{for }
 \tau\in(\tau_1,\tilde\tau_2).
 \end{align*}
 Since $e^{-\frac{\tau-\tau_1}{2}}\tau_1^{-2}<\tau^{-2}$ if $\tau_1>4$,
 it follows that
 \begin{align}
 \label{eq_4.51}
 \|W^\bot&(\tau)\|_\rho^2
 <
 {\nu}^8
 \tau^{-2}
 +
 C\tau^{-4}
 \qquad
 \text{for }
 \tau\in(\tau_1,\tilde\tau_2).
 \end{align}
 We next derive estimates for $\|\nabla_zW^\bot\|_\rho$.
 Put
 \begin{align*}
 V
 =
 \pa_{z_i}W^\bot.
 \end{align*}
 From \eqref{eq_4.43},
 we verify that
 $V$ satisfies
 \begin{align}
 \label{eq_4.52}
 V_\tau
 &=
 A_zV
 +
 \tfrac{1}{2}
 V
 +
 \pa_{z_i}N
 -
 (N,{\bf H}_2)
 \cdot
 \pa_{z_i}{\bf H}_2
 \\
 \nonumber
 &\quad
 +
 {\bf c}(\tau)
 \cdot
 \nabla_zV
 -
 \langle
 ({\bf c}(\tau)
 \cdot
 \nabla_zW),
 {\bf H}_2
 \rangle_\rho
 \cdot
 \pa_{z_i}
 {\bf H}_2.
 \end{align}
 Multiplying \eqref{eq_4.52} by $V$,
 we get
 \begin{align}
 \label{eq_4.53}
 \tfrac{1}{2}
 \tfrac{d}{d\tau}
 \|V\|_\rho^2
 &=
 -
 \|
 \nabla_zV
 \|_\rho^2
 +
 \tfrac{1}{2}
 \|
 V
 \|_\rho^2
 +
 \langle
 \pa_{z_i}N,V
 \rangle_\rho
 \\
 \nonumber
 &\quad
 +
 {\bf c}(\tau)
 \cdot
 \langle
 \nabla_zV,
 V
 \rangle_\rho.
 \end{align}
 From \eqref{eq_4.48} and (l3) - (l4),
 we can estimate the last two terms in \eqref{eq_4.53}.
 \begin{align}
 \nonumber
 |\langle
 &
 \pa_{z_i}N,V
 \rangle_\rho|
 =
 |\langle
 N,(\pa_{z_i}V-\tfrac{z_i}{2}V)
 \rangle_\rho|
 \\
 \nonumber
 &<
 C
 \|N\|_\rho
 \cdot
 C\|V\|_{H_\rho^1(\R^n)}
 \\
 \nonumber
 &<
 C
 \{
 \tfrac{1}{\tau^2}
 +
 (
 \tfrac{1}{\sqrt\tau}
 +
 \tfrac{1}{R_1}
 )
 \|W^\bot\|_{H_\rho^1(\R^n)}
 \}
 \|V\|_{H_\rho^1(\R^n)}
 \\
 \label{eq_4.54}
 &<
 C(
 \tfrac{1}{\tau^2}
 +
 \tfrac{{\nu}^4}{\tau\sqrt\tau}
 +
 \tfrac{{\nu}^4}{R_1\tau}
 )
 \|V\|_{H_\rho^1(\R^n)}
 \qquad
 \text{for }
 \tau\in(\tau_1,\tilde\tau_2)
 \end{align}
 and
 \begin{align}
 \nonumber
 |{\bf c}(\tau)
 \cdot
 \langle
 \nabla_zV,
 V
 \rangle_\rho|
 &<
 |{\bf c}(\tau)|
 \|\nabla_zV\|_\rho
 \|V\|_\rho
 \\
 \label{eq_4.55}
 &<
 \tfrac{C}{\tau}
 \|\nabla_zV\|_\rho
 \|V\|_\rho
 \qquad
 \text{for }
 \tau\in(\tau_1,\tilde\tau_2).
 \end{align}
 Substituting \eqref{eq_4.54} - \eqref{eq_4.55} to \eqref{eq_4.53}
 and repeating the above computation,
 we obtain
 if $R_1,s_1$ are large enough
 \begin{align}
 \nonumber
 \|V\|_\rho^2
 &<
 e^{-\frac{\tau-\tau_1}{2}}
 \|V(\tau_1)\|_\rho
 +
 \tfrac{C}{\tau^4}
 +
 \tfrac{C{\nu}^8}{\tau^3}
 +
 \tfrac{C{\nu}^8}{R_1^2\tau^2}
 \\
 \nonumber
 &<
 \tfrac{{\nu}^8}{\tau^2}
 +
 \tfrac{C}{\tau^4}
 +
 \tfrac{C{\nu}^8}{\tau^3}
 +
 \tfrac{C{\nu}^8}{R_1^2\tau^2}
 \\
 \label{eq_4.56}
 &<
 \tfrac{9{\nu}^8}{8\tau^2}
 +
 \tfrac{C}{\tau^4}
 \qquad
 \text{for }
 \tau\in(\tau_1,\tilde\tau_2).
 \end{align}
 Summarizing \eqref{eq_4.51} and \eqref{eq_4.56},
 we deduce that
 \begin{align}
 \label{eq_4.57}
 \|W^\bot\|_\rho^2
 &<
 \tfrac{{\nu}^8}{\tau^2}
 +
 \tfrac{C}{\tau^4}
 \qquad
 \text{for }
 \tau\in(\tau_1,\tilde\tau_2),
 \\
 \label{eq_4.58}
 \|\pa_{z_J}W^\bot\|_\rho
 &<
 \tfrac{9{\nu}^8}{8\tau^2}
 +
 \tfrac{C}{\tau^4}
 \qquad
 \text{for }
 \tau\in(\tau_1,\tilde\tau_2).
 \end{align}
 We here assume
 \begin{enumerate}[(n1)]
 \setcounter{enumi}{3}
 \item $s_1^{-1}<{\nu}^5$.
 \end{enumerate}
 If (n4) is satisfied,
 (L3) - (L4) follow from \eqref{eq_4.57} - \eqref{eq_4.58}.
 We finally prove (L5) - (L6) and (L8).
 We recall that $W^\bot(z,\tau)$ satisfies (see \eqref{eq_4.43})
 \begin{align}
 \label{eq_4.59}
 W_\tau^\bot
 &=
 A_zW^\bot
 +
 W^\bot
 +
 N
 +
 \underbrace{
 \langle-N,H_0\rangle_\rho
 H_0
 +
 \langle-N,{\bf H}_2\rangle_\rho
 \cdot
 {\bf H}_2
 }_{=F_1(z,\tau)}
 \\
 \nonumber
 &\quad
 +
 {\bf c}(\tau)
 \cdot
 \nabla_zW^\bot
 +
 \underbrace{
 \langle
 (-{\bf c}(\tau)
 \cdot
 \nabla_zW),
 {\bf H}_2
 \rangle_\rho
 \cdot
 {\bf H}_2
 }_{=F_2(z,\tau)}.
 \end{align}
 The nonlinear term $N$ is expressed as
 \begin{align*}
 N
 &=
 \tfrac{pW^2}{2\kappa}
 +
 (N-\tfrac{pW^2}{2\kappa})
 \\
 &=
 \tfrac{pWW^\bot}{2\kappa}
 +
 \tfrac{p(b_0H_0+{\bf b}_2\cdot{\bf H}_2)W}{2\kappa}
 +
 (N-\tfrac{pW^2}{2\kappa})
 \\
 &=
 \tfrac{pWW^\bot}{2\kappa}
 +
 \tfrac{p(b_0H_0+{\bf b}_2\cdot{\bf H}_2)W^\bot}{2\kappa}
 +
 \tfrac{p(b_0H_0+{\bf b}_2\cdot{\bf H}_2)^2}{2\kappa}
 +
 (N-\tfrac{pW^2}{2\kappa})
 \\
 &=
 \underbrace{
 (
 \tfrac{pW}{2\kappa}
 +
 \tfrac{p(b_0H_0+{\bf b}_2\cdot{\bf H}_2)}{2\kappa}
 )
 }_{=P(z,\tau)}
 W^\bot
 +
 \underbrace{
 \tfrac{p(b_0H_0+{\bf b}_2\cdot{\bf H}_2)^2}{2\kappa}
 }_{=F_3(z,\tau)}
 +
 \underbrace{
 (N-\tfrac{pW^2}{2\kappa})
 }_{=F_4(z,\tau)}.
 \end{align*}
 We rewrite \eqref{eq_4.59} as
 \begin{align}
 \label{eq_4.60}
 W_\tau^\bot
 &=
 A_zW^\bot
 +
 W^\bot
 +
 {\bf c}(\tau)
 \cdot
 \nabla_zW^\bot
 +
 P(z,\tau)
 W^\bot
 \\
 \nonumber
 &\quad
 +
 F_1+F_2+F_3+F_4.
 \end{align}
 From Lemma \ref{lem_3.1} and the bound of $\|W(\tau)\|_{H_\rho^1(\R^n)}$
 (see \eqref{eq_4.18}),
 we see that
 \begin{align*}
 |F_1(z,\tau)|
 &<
 C|\langle W^2,(1+|z|^2)\rangle_\rho|
 (1+|z|^2)
 \\
 &<
 C\|W\|_{H_\rho^1(\R^n)}^2
 (1+|z|^2)
 \\
 &<
 \tfrac{C}{\tau^2}
 (1+|z|^2)
 \qquad
 \text{for }
 \tau\in(\tau_1,\tilde\tau_2),
 \\
 |F_2(z,\tau)|
 &=
 |\langle
 ({\bf c}(\tau)
 \cdot
 \nabla_zW),
 {\bf H}_2
 \rangle_\rho
 \cdot
 {\bf H}_2
 |
 \\
 &=
 |\langle
 W,
 {\bf c}(\tau)\cdot(\nabla_z{\bf H}_2-\tfrac{z}{2}{\bf H}_2)
 \rangle_\rho
 \cdot
 {\bf H}_2
 |
 \\
 &=
 |\langle
 W,
 {\bf c}(\tau)\cdot\tfrac{z}{2}{\bf H}_2
 \rangle_\rho
 \cdot
 {\bf H}_2
 |
 \\
 &<
 |{\bf c}(\tau)|
 \|W^\bot\|_\rho
 (1+|z|^2)
 \\
 &<
 \tfrac{{\nu}^4}{\tau^2}
 (1+|z|^2)
 \qquad
 \text{for }
 \tau\in(\tau_1,\tilde\tau_2),
 \\
 |F_3(z,\tau)|
 &<
 C(|b_0|^2+|{\bf b}_2|^2)(1+|z|^4)
 \\
 &<
 \tfrac{C}{\tau^2}
 (1+|z|^4)
 \qquad
 \text{for }
 \tau\in(\tau_1,\tilde\tau_2).
 \end{align*}
 From \eqref{eq_4.30},
 we note that
 \begin{align*}
 |F_4(z,\tau)|
 &<
 |N-\tfrac{pW^2}{2\kappa}|
 <
 C|W|^3
 <
 C\tau^{-\frac{9}{4}}
 \qquad
 \text{for }
 |z|<16\sqrt{n}
 \end{align*}
 and
 $\tau\in(\tau_1,\tilde\tau_2)$.
 Furthermore
 we note that $|P(z,\tau)|<C$
 for $|z|<16\sqrt{n}$ and $\tau\in(\tau_1,\tilde\tau_2)$.
 Applying a local parabolic regularity estimate in \eqref{eq_4.60},
 we get from (k8) and (l3) that
 \begin{align*}
 \sup_{\tau\in(\tau_1,\tau_1+1)}
 &\sup_{|z|<8\sqrt{n}}
 |W^\bot(z,\tau)|
 <
 C
 \sup_{\tau\in(\tau_1,\tau_1+1)}
 (
 \|W^\bot(\tau)\|_\rho
 \\
 &\quad
 +
 \sup_{|z|<16\sqrt{n}}
 \sum_{i=1}^4
 |F_i(z,\tau)|
 )
 +
 C
 \sup_{|z|<16\sqrt{n}}
 |W^\bot(z,\tau_1)|
 \\
 &
 <
 C({\nu}^4\tau_1^{-1}+\tau_1^{-2}).
 \end{align*}
 Furthermore
 we obtain
 \begin{align*}
 &\sup_{\sigma\in(\tau-\frac{1}{2},\tau)}
 \sup_{|z|<8\sqrt{n}}
 |W^\bot(z,\sigma)|
 \\
 &<
 C
 (
 \sup_{\sigma\in(\tau-1,\tau)}
 \|W^\bot(\sigma)\|_\rho
 +
 \sup_{\sigma\in(\tau-1,\tau)}
 \sup_{|z|<16\sqrt{n}}
 \sum_{i=1}^4
 |F_i(z,\sigma)|
 )
 \\
 &<
 C({\nu}^4\tau^{-1}+\tau^{-2})
 \qquad
 \text{for }
 \tau\in(\tau_1+1,\tilde\tau_2).
 \end{align*}
 Both estimates imply
 \begin{align}
 \label{eq_4.61}
 \sup_{|z|<8\sqrt{n}}
 |W^\bot(z,\tau)|
 <
 \tfrac{C{\nu}^4}{\tau}
 +
 \tfrac{C}{\tau^2}
 \qquad
 \text{for }
 \tau\in(\tau_1,\tilde\tau_2).
 \end{align}
 If ${\nu}_1$ is small enough and
 \begin{enumerate}[(n1)]
 \setcounter{enumi}{4}
 \item $\tfrac{C}{s_1}<\frac{{\nu}^2}{8}$
 \quad ($C$ is a constant in \eqref{eq_4.61}),
 \end{enumerate}
 we obtain
 \begin{align}
 \label{eq_4.62}
 \sup_{|z|<8\sqrt{n}}
 |W^\bot(z,\tau)|
 <
 \tfrac{{\nu}^2}{4\tau}
 \qquad
 \text{for }
 \tau\in(\tau_1,\tilde\tau_2).
 \end{align}
 This estimate shows (L8).
 We now prove (L6) by using \eqref{eq_4.62}.
 From (l1) - (l2) and \eqref{eq_4.62},
 we have
 \begin{align*}
 &|
 w(z,\,\tau)
 -
 \kappa
 -
 b_0
 H_0
 +
 \tfrac{1}{c_p\tau}
 \sum_{J=1}^n
 H_{2,JJ}(z)
 |
 \\
 &<
 |
 \sum_{J=1}^n
 (\underbrace{b_{2,JJ}+\tfrac{1}{c_p\tau}}_{<C{\nu}\tau^{-1}})H_{2,JJ}(z)
 |
 +
 \underbrace{|{\bf b}_{2,IJ}|}_{<{\nu}^{2}\tau^{-1}}
 |{\bf H}_{2,IJ}(z)|
 +
 \underbrace{
 |w^\bot(z,\tau)|
 }_{=|W^\bot(z,\tau)|}
 \\
 &<
 \tfrac{C{\nu}}{\tau}
 \qquad
 \text{for }
 |z|<4\sqrt{n}
 \text{ and }
 \tau\in(\tau_1,\tilde\tau_2).
 \end{align*}
 We here note that (see Section \ref{sec_2} for the definition of $H_{2,JJ}$ and
 \eqref{eq_2.7} for the definition of $c_p$)
 \begin{align}
 \label{eq_4.63}
 \tfrac{1}{c_p\tau}\sum_{J=1}^nH_{2,JJ}(z)
 =
 \tfrac{\kappa}{4p\tau}(|z|^2-2n).
 \end{align}
 Since $|b_0(\tau)|<C{\nu}^2\tau^{-1}$ (see (k0)),
 it follows that
 \begin{align}
 \nonumber
 w(z,\tau)&
 |_{|z|=2\sqrt{n}}
 <
 \kappa
 +
 b_0H_0
 +
 \tfrac{1}{c_p\tau}
 \sum_{J=1}^n
 H_{2,JJ}(z)
 |_{|z|=2\sqrt{n}}
 +
 \tfrac{C{\nu}}{\tau}
 \\
 \nonumber
 &<
 \kappa
 +
 \tfrac{C{\nu}^2}{\tau}
 -
 \tfrac{\kappa}{4p\tau}(|z|^2-2n)
 |_{|z|=2\sqrt{n}}
 +
 \tfrac{C{\nu}}{\tau}
 \\
 \label{eq_4.64}
 &<
 \kappa
 -
 \tfrac{\kappa n}{2p\tau}
 +
 \tfrac{C{\nu}}{\tau}
 \qquad
 \text{for }
 \tau\in(\tau_1,\tilde\tau_2).
 \end{align}
 Therefore
 if $C\nu_1<\frac{\kappa n}{8p}$,
 we get
 \begin{align}
 w(z,\tau)|_{|z|=2\sqrt{n}}
 \label{eq_4.65}
 <
 \kappa
 -
 \tfrac{\kappa n}{4p\tau}
 \qquad
 \text{for }
 \tau\in(\tau_1,\tilde\tau_2).
 \end{align}
 We apply a comparison argument
 together with \eqref{eq_4.65} and (k6) to obtain
 \begin{align}
 \label{eq_4.66}
 w(z,\tau)
 <\kappa
 \qquad
 \text{for }
 |z|>2\sqrt{n},\
 \tau\in(\tau_1,\tilde\tau_2).
 \end{align}
 This shows (L6).
 From (k0), (l1) - (l2) and \eqref{eq_4.62},
 it holds that
 \begin{align}
 \nonumber
 \sup_{|z|<4\sqrt{n}}
 |w(z,\tau)-\kappa|
 &<
 \sup_{|z|<4\sqrt{n}}
 |b_0H_0+{\bf b}_2\cdot{\bf H}_2|
 +
 \sup_{|z|<4\sqrt{n}} 
 |w^\bot(z,\tau)|
 \\
 \label{eq_4.67}
 &<
 \tfrac{C}{\tau}
 \qquad
 \text{for }
 \tau\in(\tau_1,\tilde\tau_2).
 \end{align}
 Therefore
 \eqref{eq_4.66} - \eqref{eq_4.67} show the right hand inequality in (L5).
 The left hand inequality in (L5)
 is immediately obtained by a comparison argument.
 In fact,
 since we are now assuming (k5),
 it is sufficient to check that
 $w_1(z,\tau)=e^{-\frac{3\tau}{4(p-1)}}$ gives a subsolution of \eqref{eq_4.13}.
 The proof of (L1) - (L8) is completed.
 We finally list (n1) - (n5), which is conditions for $R_1,s_1$.
 \begin{enumerate}[(n1)]
 \item 
 $R_1^2s_1^{-\frac{1}{8}}<1$,

 \item 
 $e^{8eR_1^2}s_1^{-\frac{1}{8}}<1$,

 \item 
 $e^{-\frac{R_1^2}{8}}<{\nu}^3$,

 \item 
 $s_1^{-1}<{\nu}^5$,

 \item 
 $\tfrac{C}{s_1}<\frac{{\nu}^2}{8}$
 \quad ($C$ is a constant in \eqref{eq_4.61}).
 \end{enumerate}
 For given $\nu\in(0,\nu_1)$,
 we can choose $R_1,s_1$ satisfying (n1) - (n5).
 For example,
 it is enough to take $R_1=\frac{1}{{\nu}}$ and $s_1=e^\frac{128e}{{\nu^2}}$.
 \end{proof}

 We now go back to \eqref{eq_4.5}.
 We do not assume the orthogonal condition here:
 \[
 \langle w(z,\tau),H_{1,J}\rangle_\rho=0.
 \]
 Furthermore
 instead of (k0),
 we assume
 \begin{enumerate}[({i}1)]
 \setlength{\leftskip}{5mm}
 \item 
 $\tau_2>\tau_0$,
 \item 
 $b_0(\tau_2)
 =
 -\tfrac{\nu_0^2}{pH_0\tau_2}$,
 \item
 $
 b_{1,J}(\tau_2)
 =
 0$
 \quad
 for all $J$.
 \end{enumerate}
 \begin{pro}\label{proposition_4.4}
 Assume $n\geq1$ and $p>1$.
 Let
 $w(z,\tau)$
 be a solution of \eqref{eq_4.5}
 satisfying
 {\rm(i1) - (i3)}.
 Assume that
 \begin{enumerate}[\rm({j}1) ]
 \setlength{\leftskip}{5mm}
 \item 
 $|b_{2,JJ}(\tau_2)+\frac{1}{c_p\tau_2}|<\frac{2{\nu}}{c_p\tau_2}$
 \quad
 {\rm for all} $J$,
 \item 
 $|b_{2,IJ}(\tau_2)|<{\nu}^2\tau_2^{-1}$
 \quad
 {\rm for all} $I<J$,
 \item 
 $\|w^\bot(\tau_2)\|_\rho<2{\nu}^4\tau_2^{-1}$,
 \item 
 $\|\pa_{z_J}w^\bot(\tau_2)\|_\rho<2{\nu}^4\tau_2^{-1}$
 \quad
 {\rm for all} $J$,
 \item 
 $-e^{-\frac{3\tau_2}{4(p-1)}}<w(z,\tau_2)<\kappa+\tau_2^{-\frac{1}{2}}$
 \quad
 {\rm for} $z\in\R^n$
 \item 
 $w(z,\tau_2)<\kappa$
 \quad {\rm for} $|z|>2\sqrt{n}$.
 \item 
 $\dis\sup_{|z|<4R_2}|w(z,\tau_2)-\kappa|<\tau_2^{-\frac{3}{4}}$,
 \item 
 $\dis\sup_{|z|<8\sqrt{n}}|w^\bot(z,\tau_2)|<{\nu}^2\tau_2^{-1}$.
 \end{enumerate}
 There exists ${\nu}_2\in(0,1)$
 such that
 for any $\nu\in(0,\nu_2)$
 there exist $R_2=R_2(\nu)>1$
 and
 $s_2=s_2(\nu,R_2)>1$ such that
 if
 $\tau_2>s_2$,
 there exists $\tau_3\in(\tau_2,\infty)$ such that
 \begin{align*}
 w(z,\tau_3)
 <
 \kappa
 \qquad
 \text{\rm for }
 z\in\R^n.
 \end{align*}
 \end{pro}
 \begin{proof}
 Let $\nu_2>0$ be a small constant.
 For $\nu\in(0,\nu_2)$,
 we put
 (see the end of the proof of Proposition \ref{proposition_4.3})
 \begin{itemize}
 \item $R_2=\tfrac{1}{\nu}$,
 \item $s_2=e^\frac{128e}{{\nu^2}}$.
 \end{itemize}
 Define $\tau_3>\tau_2$ by
 \begin{align}
 \label{eq_4.68}
 \tau_3
 =
 \tau_2
 +
 |\log(L_1{\nu}^2)|
 \qquad
 (L_1=\tfrac{1}{\kappa n}).
 \end{align}
 Since $\tau_2>s_2=e^\frac{128e}{{\nu^2}}\gg{\nu}^{-1}$,
 it holds that
 \begin{align}
 \label{eq_4.69}
 1<\tfrac{\tau_3}{\tau_2}<\tfrac{3}{2}.
 \end{align}
 We first claim that
 \begin{enumerate}[(m1) ]
 \setlength{\leftskip}{5mm}
 \item 
 $-\tfrac{4\kappa n}{p}
 \tfrac{1}{H_0\tau}<b_0(\tau)<0$
 \quad
 \text{for }
 $\tau\in(\tau_2,\tau_3)$,
 \item 
 $|b_{1,J}(\tau)|
 <
 {\nu}^2
 \tau^{-1}$
 \quad
 for
 $\tau\in(\tau_2,\tau_3)$
 and all $J$,
 \item 
 $|b_{2,JJ}(\tau)+\tfrac{1}{c_p\tau}|
 <
 \frac{4{\nu}}{c_p\tau_2}$
 \quad
 for
 $\tau\in(\tau_2,\tau_3)$
 and all $J$,
 \item 
 $|b_{2,IJ}(\tau)|
 <
 4{\nu}^2\tau^{-1}$
 \quad
 for
 $\tau\in(\tau_2,\tau_3)$
 {\rm and all} $I<J$,
 \item 
 $\|w^\bot(\tau)\|_\rho
 <
 8
 {\nu}^4
 \tau^{-1}$
 \quad
 for
 $\tau\in(\tau_2,\tau_3)$,
 \item 
 $\|\pa_{z_J}w^\bot(\tau)\|_\rho
 <
 8
 {\nu}^4
 \tau^{-1}$
 \quad
 for
 $\tau\in(\tau_2,\tau_3)$
 and all $J$
 \item 
 $-e^{-\frac{3\tau}{4(p-1)}}<w(z,\tau)<\kappa+2\tau^{-\frac{1}{2}}$
 \quad {\rm for}
 $(z,\tau)\in\R^n\times[\tau_2,\tau_3]$,
 \item 
 $w(z,\tau)<\kappa$
 \quad {\rm for} $|z|>2\sqrt{n}$
 \quad
 {\rm for}
 $\tau\in[\tau_2,\tau_3]$.
 \end{enumerate}
 Assume that there exists $\tilde \tau_3\in(\tau_2,\tau_3)$
 such that
 \begin{enumerate}[(h1) ]
 \setlength{\leftskip}{5mm}
 \item 
 $-\tfrac{4\kappa n}{p}\tfrac{1}{H_0\tau}<b_0(\tau)<0$
 \text{for }
 $\tau\in(\tau_2,\tilde\tau_3)$,
 \item 
 $|b_{1,J}(\tau)|
 <
 2{\nu}^2
 \tau^{-1}$
 \quad
 for
 $\tau\in(\tau_2,\tilde\tau_3)$
 and all $J$,
 \item 
 $|b_{2,JJ}(\tau)+\tfrac{1}{c_p\tau}|
 <
 \frac{8{\nu}}{c_p\tau}$
 \quad
 for
 $\tau\in(\tau_2,\tilde\tau_3)$
 {\rm and all} $J$,
 \item 
 $|b_{2,IJ}(\tau)|
 <
 4{\nu}^2\tau^{-1}$
 \quad
 for
 $\tau\in(\tau_2,\tilde\tau_3)$
 {\rm and all} $I<J$,
 \item 
 $\|w^\bot(\tau)\|_\rho
 <
 8{\nu}^4
 \tau^{-1}$
 \quad
 for
 $\tau\in(\tau_2,\tilde\tau_3)$,
 \item 
 $\|\pa_{z_J}w^\bot(\tau)\|_\rho
 <
 8{\nu}^4
 \tau^{-1}$
 \quad
 for
 $\tau\in(\tau_2,\tilde\tau_3)$
 and all $J$,
 \item 
 $-e^{-\frac{3\tau}{4(p-1)}}<w(z,\tau)<\kappa+2\tau^{-\frac{1}{2}}$
 \quad {\rm for}
 $(z,\tau)\in\R^n\times[\tau_2,\tilde\tau_3]$,
 \item 
 $w(z,\tau)<\kappa$
 \quad
 {\rm for} $|z|>2\sqrt{n}$
 {\rm and} $\tau\in[\tau_2,\tilde\tau_3]$.
 \end{enumerate}
 We now prove
 \begin{enumerate}[(H1) ]
 \setlength{\leftskip}{5mm}
 \item 
 $-
 \tfrac{3\kappa n}{p}
 \tfrac{1}{H_0\tau}<b_0(\tau)<0$
 \text{for }
 $\tau\in(\tau_2,\tilde\tau_3)$,
 \item 
 $|b_{1,J}(\tau)|
 <
 \frac{1}{2}
 {\nu}^2
 \tau^{-1}$
 \quad
 for
 $\tau\in(\tau_2,\tilde\tau_3)$
 and all $J$,
 \item 
 $|b_{2,JJ}(\tau)+\tfrac{1}{c_p\tau}|
 <
 \frac{4{\nu}}{c_p\tau}$
 \quad
 for
 $\tau\in(\tau_2,\tilde\tau_3)$
 {\rm and all} $J$,
 \item 
 $|b_{2,IJ}(\tau)|
 <
 2
 {\nu}^2\tau^{-1}$
 \quad
 for
 $\tau\in(\tau_2,\tilde\tau_3)$
 {\rm and all} $I<J$,
 \item 
 $\|w^\bot(\tau)\|_\rho
 <
 4{\nu}^4
 \tau^{-1}$
 \quad
 for
 $\tau\in(\tau_2,\tilde\tau_3)$,
 \item 
 $\|\pa_{z_J}w^\bot(\tau)\|_\rho
 <
 4{\nu}^4
 \tau^{-1}$
 \quad
 for
 $\tau\in(\tau_2,\tilde\tau_3)$
 and all $J$,
 \item 
 $-e^{-\frac{3\tau}{4(p-1)}}<w(z,\tau)<\kappa+\tau^{-\frac{1}{2}}$
 \quad {\rm for}
 $(z,\tau)\in\R^n\times[\tau_2,\tilde\tau_3]$,
 \item 
 $w(z,\tau)<\kappa$
 \quad
 {\rm for} $|z|>2\sqrt{n}$
 {\rm and} $\tau\in[\tau_2,\tilde\tau_3+\delta_1]$
 {\rm for some} $\delta_1\in(0,1)$
 \\
 {\rm(}$\delta_1$ may depend on $w(z,\tau)${\rm)}.
 \end{enumerate}
 In order to show the claim (m1) - (m8),
 it is sufficient to derive (H1) - (H8) from the assumptions (h1) - (h8).
 We repeat the argument in Proposition \ref{proposition_4.3}.
 We first confirm (H3) - (H4).
 The proof for (H3) - (H4) is simpler than that of (L1) -(L2)
 in Proposition \ref{proposition_4.3}.
 We recall that $W(z,\tau)=w(z,\tau)-\kappa$ satisfies
 \begin{align}
 \label{eq_4.70}
 W_\tau=A_zW+W+N.
 \end{align}
 From Lemma \ref{lemma_A.3} and (h1) - (h6),
 we can verify that
 \begin{align*}
 |
 \langle N,H_{2,IJ}\rangle_\rho
 |
 &<
 C\langle W^2,|z|^2\rangle_\rho 
 <
 C\|W\|_{H_\rho^1(\R^n)}^2
 \\
 &<
 \tfrac{C}{\tau^2}
 \qquad
 \text{for } \tau\in(\tau_2,\tilde\tau_3).
 \end{align*}
 This implies
 (see \eqref{eq_4.39} - \eqref{eq_4.40})
 \begin{align*}
 |\tfrac{d}{d\tau}b_{2,JJ}|
 &<
 C
 \tau^{-2}
 \qquad
 \text{for }
 \tau\in(\tau_2,\tilde\tau_3),
 \\
 |\tfrac{d}{d\tau}b_{2,IJ}|
 &<
 C\tau^{-2}
 \qquad
 \text{for }
 \tau\in(\tau_2,\tilde\tau_3)
 \text{ and }
 I<J.
 \end{align*}
 Integrating both sides,
 we get from (j1) - (j2) that
 for $\tau\in(\tau_2,\tilde\tau_3)$
 \begin{align*}
 |b_{2,JJ}(\tau)-\tfrac{1}{c_p\tau}|
 &<
 |\tfrac{1}{c_p\tau}-\tfrac{1}{c_p\tau_2}|
 +
 |b_{2,JJ}(\tau_2)-\tfrac{1}{c_p\tau_2}|
 +
 \tfrac{C(\tau-\tau_2)}{\tau_2^2}
 \\
 &<
 \tfrac{\tau-\tau_2}{c_p\tau\tau_2}
 +
 \tfrac{2{\nu}}{c_p\tau_2}
 +
 \tfrac{C(\tau-\tau_2)}{\tau_2^2},
 \\
 |b_{2,IJ}(\tau)|
 &<
 |b_{2,IJ}(\tau_2)|
 +
 \tfrac{C(\tau-\tau_2)}{\tau_2^2}
 \\
 &<
 \tfrac{{\nu}^2}{\tau_2}
 +
 \tfrac{C(\tau-\tau_2)}{\tau_2^2}
 \qquad
 (I<J).
 \end{align*}
 From \eqref{eq_4.68} and $\tau_2>s_2=e^\frac{128e}{{\nu^2}}\gg{\nu}^{-1}$,
 it holds that
 \begin{align*}
 |\tau-\tau_2|
 <
 |\tau_3-\tau_2|
 <
 |\log(L_1{\nu^2})|
 <
 \log\tau_2
 \qquad
 \text{for }
 \tau\in(\tau_2,\tilde\tau_3).
 \end{align*}
 Hence
 combining \eqref{eq_4.69},
 we obtain
 \begin{align}
 \label{eq_4.71}
 |b_{2,JJ}(\tau)-\tfrac{1}{c_p\tau}|
 &<
 \tfrac{4{\nu}}{c_p\tau}
 \qquad
 \text{for }
 \tau\in(\tau_2,\tilde\tau_3),
 \\
 \label{eq_4.72}
 |b_{2,IJ}(\tau)|
 &<
 \tfrac{2{\nu}^2}{\tau}
 \qquad
 \text{for }
 \tau\in(\tau_2,\tilde\tau_3)
 \text{ and }
 I<J.
 \end{align}
 We next compute $\|W^\bot(z,\tau)\|_\rho$ and $\|\nabla_zW^\bot(z,\tau)\|_\rho$.
 To derive a pointwise estimate for $|w(z,\tau)-\kappa|$ (see \eqref{eq_4.30}),
 we follow the argument \eqref{eq_4.24} - \eqref{eq_4.30}.
 Since we now assume (j7) instead of (k7),
 it is sufficient to replace \eqref{eq_4.26} by
 \begin{align*}
 \sup_{|z|<2}
 |W_{\bf a}(z,\tau_2)|
 <
 \tau_2^{-\frac{3}{4}}
 \qquad
 (|{\bf a}|<2R_2).
 \end{align*}
 Then
 we can verify that
 \begin{align}
 \label{eq_4.73}
 \sup_{|z|<2R_2}|w(z,\tau)-\kappa|
 <
 C
 \tau^{-\frac{3}{4}}
 \qquad
 \text{for }
 \tau\in(\tau_2,\tilde\tau_3).
 \end{align}
 Once \eqref{eq_4.73} is derived,
 we can treat the nonlinear term in the same way as \eqref{eq_4.45} - \eqref{eq_4.48}.
 In fact,
 we obtain
 \begin{align}
 \label{eq_4.74}
 \|N\|_\rho
 <
 \tfrac{C}{\tau^2}
 +
 (
 \tfrac{C}{\sqrt\tau}
 +
 \tfrac{C}{R_2}
 )
 \|W^\bot\|_{H_\rho^1(\R^n)}
 \qquad
 \text{for }
 \tau\in(\tau_2,\tilde\tau_3).
 \end{align}
 Since
 the argument \eqref{eq_4.49} - \eqref{eq_4.58} is still valid,
 we can derive
 \begin{align}
 \label{eq_4.75}
 \|W^\bot\|_\rho^2
 &<
 \tfrac{{\nu}^8}{\tau^2}
 +
 \tfrac{C}{\tau^4}
 \qquad
 \text{for }
 \tau\in(\tau_2,\tilde\tau_3),
 \\
 \label{eq_4.76}
 \|\pa_{z_J}W^\bot\|_\rho
 &<
 \tfrac{9{\nu}^8}{8\tau^2}
 +
 \tfrac{C}{\tau^4}
 \qquad
 \text{for }
 \tau\in(\tau_2,\tilde\tau_3).
 \end{align}
 Therefore
 \eqref{eq_4.71} - \eqref{eq_4.72} and \eqref{eq_4.75} - \eqref{eq_4.76}
 guarantee (H3) - (H6).
 We next show (H1).
 We take an inner product $\langle\cdot,H_0\rangle_\rho$ in \eqref{eq_4.70}.
 \begin{align}
 \label{eq_4.77}
 \tfrac{d}{d\tau}
 \langle W,H_0\rangle_\rho
 =
 \langle W,H_0\rangle_\rho
 +
 \langle N,H_0\rangle_\rho.
 \end{align}
 From Lemma \ref{lemma_A.3} and (h1) - (h5),
 we get
 \begin{align}
 \label{eq_4.78}
 |\langle N,H_0\rangle_\rho|
 &<
 C\tau^{-2}
 \qquad
 \text{for }
 \tau\in(\tau_2,\tilde\tau_3).
 \end{align}
 Hence
 integrating \eqref{eq_4.77}
 together with \eqref{eq_4.78} and (i2),
 we get
 \begin{align}
 \label{eq_4.79}
 |b_0(\tau)+e^{\tau-\tau_2}\tfrac{{\nu}^2}{pH_0\tau_2}|
 &<
 Ce^{\tau-\tau_2}
 \tau_2^{-2}
 \qquad
 \text{for }
 \tau\in(\tau_2,\tilde\tau_3).
 \end{align}
 Since $\tau_2^{-1}<s_2^{-1}=e^{-\frac{128e}{{\nu^2}}}\ll{\nu}^2$,
 \eqref{eq_4.79} implies
 \begin{align}
 \label{eq_4.80}
 -\tfrac{5{\nu}^2}{4pH_0\tau_2}
 e^{\tau-\tau_2}
 <
 b_0(\tau)
 <
 -\tfrac{3{\nu}^2}{4pH_0\tau_2}
 e^{\tau-\tau_2}
 \qquad
 \text{for }
 \tau\in(\tau_2,\tilde\tau_3).
 \end{align}
 From the definition of $\tau_3$ (see \eqref{eq_4.68}),
 we deduce from \eqref{eq_4.80} that
 \begin{align*}
 0
 >
 b_0(\tau)
 &>
 -\tfrac{5{\nu}^2}{4pH_0\tau_2}
 e^{\tau_3-\tau_2}
 =
 -\tfrac{5}{4pH_0L_1\tau_2}
 =
 -\tfrac{5\kappa n}{4pH_0\tau_2}
 \\
 &>
 -\tfrac{5\kappa n}{2pH_0\tau}
 \qquad
 \text{for }
 \tau\in(\tau_2,\tilde\tau_3).
 \end{align*}
 This shows (H1).
 To derive (H2),
 we take an inner product $\langle\cdot,H_{1,J}\rangle_\rho$ in \eqref{eq_4.70}.
 \begin{align}
 \label{eq_4.81}
 \tfrac{d}{d\tau}
 \langle W,H_{1,J}\rangle_\rho
 =
 \tfrac{1}{2}
 \langle W,H_{1,J}\rangle_\rho
 +
 \langle N,H_{1,J}\rangle_\rho.
 \end{align}
 From Lemma \ref{lemma_A.3} and (h1) - (h6),
 we get
 \begin{align}
 \label{eq_4.82}
 |\langle N,H_{1,J}\rangle_\rho|
 &<
 C\tau^{-2}
 \qquad
 \text{for }
 \tau\in(\tau_2,\tilde\tau_3).
 \end{align}
 Hence
 integrating \eqref{eq_4.81}
 together with \eqref{eq_4.82} and (i3),
 we get
 \begin{align}
 \nonumber
 |b_{1,J}(\tau)|
 &<
 \tfrac{C}{\tau_2^2}
 e^{\frac{\tau-\tau_2}{2}}
 <
 \tfrac{C}{\tau_2^2}
 e^{\frac{\tau_3-\tau_2}{2}}
 =
 \tfrac{C}{\tau_2^2}
 \tfrac{1}{\sqrt{L_1{\nu}^2}}
 \\
 \label{eq_4.83}
 &<
 \tfrac{C}{{\nu}\tau_2^2}
 <
 \tfrac{C}{{\nu}\tau^2}
 \qquad
 \text{for }
 \tau\in(\tau_2,\tilde\tau_3).
 \end{align}
 In the last inequality,
 we used \eqref{eq_4.69}.
 Since $\tau_2^{-1}<s_2^{-1}=e^{-\frac{128e}{{\nu^2}}}\ll{\nu}^2$,
 it follows that
 $|b_{1,J}(\tau)|<\frac{{\nu}^2}{2\tau}$ for $\tau\in(\tau_2,\tilde\tau_3)$.
 This shows (H2).
 To obtain a pointwise estimate for $W^\bot(z,\tau)$ (see \eqref{eq_4.62}),
 we repeat \eqref{eq_4.59} - \eqref{eq_4.62}.
 By using (j8) instead of (k8),
 we can derive
 \begin{align}
 \label{eq_4.84}
 \sup_{|z|<4\sqrt{n}}
 |W^\bot(z,\tau)|
 <
 \tfrac{C{\nu}^2}{\tau}
 <
 \tfrac{{\nu}}{\tau}
 \qquad
 \text{for }
 \tau\in(\tau_2,\tilde\tau_3).
 \end{align}
 Once \eqref{eq_4.84} is derived,
 by using the fact that $b_0(\tau)<0$,
 we can confirm (H7) - (H8) in the same way as \eqref{eq_4.64} - \eqref{eq_4.67}.
 Therefore
 (m1) - (m8) are proved.
 We expand $w(z,\tau)$ as
 $w(z,\tau)
 =
 \kappa
 +
 b_0(\tau)H_0
 +
 {\bf b}_{1}(\tau)
 \cdot
 {\bf H}_{1}
 +
 {\bf b}_2(\tau)
 \cdot
 {\bf H}_2
 +
 w^\bot(z,\tau)$.
 From
 \eqref{eq_4.80}, (m2) - (m4) and \eqref{eq_4.84}
 we see that
 \begin{align*}
 w&(z,\tau_3)
 =
 \kappa
 +
 b_0(\tau_3)H_0
 +
 {\bf b}_{1}(\tau_3)
 \cdot
 {\bf H}_{1}
 +
 {\bf b}_{2,JJ}(\tau_3)
 \cdot
 {\bf H}_{2,JJ}
 \\
 &\quad
 +
 {\bf b}_{2,IJ}(\tau_3)
 \cdot
 {\bf H}_{2,IJ}
 +
 w^\bot(z,\tau_3)
 \\
 &<
 \kappa
 -
 \tfrac{3{\nu}^2}{4pH_0\tau_2}
 e^{\tau_3-\tau_2}
 H_0
 +
 {\bf b}_{1}(\tau_3)
 \cdot
 {\bf H}_{1}
 -
 \tfrac{1}{c_p\tau_3}
 \sum_{J=1}^n
 H_{2,JJ}
 \\
 &\quad
 +
 \sum_{J=1}^n
 (b_{2,JJ}(\tau_3)+\tfrac{1}{c_p\tau_3})
 H_{2,JJ}
 +
 {\bf b}_{2,IJ}(\tau_3)
 \cdot
 {\bf H}_{2,IJ}
 +
 w^\bot(z,\tau_3)
 \\
 &<
 \kappa
 -
 \tfrac{3\kappa n}{4p\tau_2}
 -
 \tfrac{1}{c_p\tau_3}
 \sum_{J=1}^n
 H_{2,JJ}
 +
 C{\nu}\tau_3^{-1}
 \qquad
 \text{for }
 |z|<4\sqrt{n}.
 \end{align*}
 From \eqref{eq_4.63},
 we have
 \begin{align*}
 w(z,\tau_3)
 &<
 \kappa
 -
 \tfrac{3\kappa n}{4p\tau_2}
 -
 \tfrac{\kappa}{4p\tau_3}(|z|^2-2n)
 +
 \tfrac{C{\nu}}{\tau_3}
 \\
 &<
 \kappa
 -
 \tfrac{3\kappa n}{4p\tau_2}
 +
 \tfrac{2\kappa n}{4p\tau_3}
 +
 \tfrac{C{\nu}}{\tau_3}
 \\
 &=
 \kappa
 -
 \tfrac{\kappa n}{4p\tau_2}
 +
 \tfrac{C{\nu}}{\tau_3}
 \qquad
 \text{for }
 |z|<4\sqrt{n}.
 \end{align*}
 Combining (m8),
 we obtain conclusion.
 \end{proof}

 \section{Proof of Theorem \ref{thm_2}}
 Let $\varphi(x,t)$ be a solution of
 $\varphi_t=\Delta_x\varphi+|\varphi|^{p-1}\varphi$
 satisfying (a1) - (a6) stated in Theorem \ref{thm_2},
 and  put
 \begin{align*}
 \Phi(z,\tau)
 =
 e^{-\frac{\tau}{p-1}}
 \varphi(e^{-\frac{\tau}{2}}z,T-e^{-\tau}),
 \end{align*}
 here $T$ is the blowup time of $\varphi(x,t)$.
 From Proposition \ref{proposition_B.1},
 we know that
 there exists $\sigma_1\in(-\log T,\infty)$
 such that
 \begin{enumerate}[(s1) ]
 \setlength{\leftskip}{5mm}
 \item 
 $\dis\lim_{\tau\to\infty}
 \tau\|\Phi(z,\tau)-\kappa+\tfrac{1}{c_p\tau}\sum_{J=1}^nH_{2,JJ}
 \|_{H_\rho^1(\R^n)}=0$,

 \item 
 $-e^{-\frac{7\tau}{8(p-1)}}<\Phi(z,\tau)
 <
 \kappa+\tfrac{2n\kappa}{p\tau}
 $
 \quad
 \text{for }
 $(z,\tau)\in\R^n\times(\sigma_1,\infty)$,

 \item 
 there exists $c_1\in(0,1)$ such that
 \begin{align*}
 \sup_{|z|>\sqrt{3n}}
 \Phi(z,\tau)
 &<
 \kappa
 -
 \tfrac{c_1}{\tau}
 \qquad
 \text{for }
 \tau\in(\sigma_1,\infty),
 \end{align*}

 \item 
 there exists $c_2>1$ such that
 \[
 \|\Phi^\bot(z,\tau)\|_{L_\rho^2(\R^n)}<c_2\tau^{-2}
 \qquad
 \text{for }
 \tau\in(\sigma_1,\infty),
 \]

 \item 
 $\dis\sup_{|z|<\tau^\frac{1}{16}}
 |\Phi^\bot(z,\tau)|<\tau^{-\frac{3}{2}}$
 \quad
 \text{for }
 $\tau\in(\sigma_1,\infty)$.
 \end{enumerate}
 Let $u(x,t)$ be a solution of
 $u_t=\Delta_xu+|u|^{p-1}u$ with $u|_{t=0}=u_0$ and
 let $T(u)$ be the blowup time of $u(x,t)$.
 We put
 \begin{align*}
 w(z,\tau)
 =
 e^{-\frac{\tau}{p-1}}
 u(e^{-\frac{\tau}{2}}z,T(u)-e^{-\tau}).
 \end{align*}
 We first prove the following fact.
 \begin{itemize}
 \item 
 (claim 1)
 \quad
 For any $\epsilon\in(0,1)$,
 there exists $\sigma_1\in(1,\infty)$
 such that
 for any $\sigma>\sigma_1$ there exists $\delta_1=\delta_1(\sigma)$
 such that
 if $\|u_0-\varphi_0\|_{L_x^\infty(\R^n)}<\delta_1$,
 then (t1) - (t5) hold true.
 \begin{enumerate}[(t1) ]
 \item 
 $\dis
 \|w(z,\tau)-\kappa+\tfrac{1}{c_p\tau}\sum_{J=1}^nH_{2,JJ
 }\|_\rho
 <\epsilon\tau^{-1}$
 \quad
 for $\tau\in(\sigma-1,\sigma+1)$,

 \item 
 $-2e^{-\frac{7\tau}{8(p-1)}}<w(z,\tau)
 <\kappa+\tfrac{4n\kappa}{p\tau}$
 \quad
 for $z\in\R^n$
 and $\tau\in(\sigma-1,\sigma+1)$,

 \item 
 $\dis\sup_{|z|>\sqrt{3n}}
 w(z,\tau)<\kappa-\tfrac{c_1}{2\tau}$
 \quad
 for $\tau\in(\sigma-1,\sigma+1)$,

 \item 
 $\|w^\bot(z,\tau)\|_\rho<2c_2\tau^{-2}$
 \quad
 for $\tau\in(\sigma-1,\sigma+1)$,

 \item 
 $\dis\sup_{|z|<\tau^\frac{1}{16}}
 |w^\bot(z,\tau)|<2\tau^{-\frac{3}{2}}$
 \quad
 for $\tau\in(\sigma-1,\sigma+1)$.
 \end{enumerate}
 \end{itemize}
 From Lemma \ref{lemma_C},
 for any $\epsilon\in(0,1)$,
 there exists $h_1\in(0,1)$ such that if $\|u_0-\varphi_0\|_{L_x^\infty(\R^n)}<h_1$,
 then
 \begin{align}
 \label{eq_5.1}
 |T(u)-T|<\epsilon.
 \end{align}
 We recall that $T$ is the blowup time of $\varphi(x,t)$.
 Furthermore
 for any fixed $t_1\in(0,T)$ and any $\epsilon\in(0,1)$,
 there exists $h_2\in(0,1)$ ($h_2$ depends on $t_1$) such that
 if $\|u_0-\varphi_0\|_{L_x^\infty(\R^n)}<h_2$,
 then
 \begin{align}
 \label{eq_5.2}
 \sup_{(x,t)\in\R^n\times(0,t_1)}
 |u(x,t)-\varphi(x,t)|
 <
 \epsilon.
 \end{align}
 From the definition of $w(z,\tau)$ and $\Phi(z,\tau)$,
 it holds that
 \begin{align}
 \nonumber
 &
 w(z,\tau)-\Phi(z,\tau)
 \\
 \nonumber
 &=
 e^{-\frac{\tau}{p-1}}
 u(e^{-\frac{\tau}{2}}z,T(u)-e^{-\tau})
 -
 e^{-\frac{\tau}{p-1}}
 \varphi(e^{-\frac{\tau}{2}}z,T-e^{-\tau})
 \\
 \label{eq_5.3}
 &=
 e^{-\frac{\tau}{p-1}}
 u(e^{-\frac{\tau}{2}}z,T(u)-e^{-\tau})
 -
 e^{-\frac{\tau}{p-1}}
 \varphi(e^{-\frac{\tau}{2}}z,T(u)-e^{-\tau})
 \\
 \nonumber
 &\quad
 +
 e^{-\frac{\tau}{p-1}}
 \varphi(e^{-\frac{\tau}{2}}z,T(u)-e^{-\tau})
 -
 e^{-\frac{\tau}{p-1}}
 \varphi(e^{-\frac{\tau}{2}}z,T-e^{-\tau}).
 \end{align}
 The function $e^{-\frac{\tau}{p-1}}\varphi(e^{-\frac{\tau}{2}}z,T(u)-e^{-\tau})$
 in \eqref{eq_5.3} is defined on
 $\{\tau>-\log T;T(u)-e^{-\tau}\in(0,T)\}$.
 For simplicity,
 we write
 \begin{align}
 \label{eq_5.4}
 t(u)
 =
 T(u)-e^{-\tau},
 \qquad
 t
 =
 T-e^{-\tau}.
 \end{align}
 We rewrite \eqref{eq_5.3} as
 \begin{align}
 \label{eq_5.5}
 w(z,\tau)&-\Phi(z,\tau)
 =
 e^{-\frac{\tau}{p-1}}
 \{
 u(e^{-\frac{\tau}{2}}z,t(u))
 -
 \varphi(e^{-\frac{\tau}{2}}z,t(u))
 \}
 \\
 \nonumber
 &\quad
 +
 e^{-\frac{\tau}{p-1}}
 \{
 \varphi(e^{-\frac{\tau}{2}}z,t(u))
 -
 \varphi(e^{-\frac{\tau}{2}}z,t)
 \}.
 \end{align}
 The last term in \eqref{eq_5.5} is computed as
 \begin{align}
 \nonumber
 \sup_{z\in\R^n}
 &|
 \varphi(e^{-\frac{\tau}{2}}z,t(u))-\varphi(e^{-\frac{\tau}{2}}z,t)
 |
 \\
 \nonumber
 &<
 |t(u)-t|
 \sup_{x\in\R^n}
 \int_0^1
 |\varphi_t(x,\theta t(u)+(1-\theta)t)|
 d\theta
 \\
 \label{eq_5.6}
 &<
 C
 |T(u)-T|
 \sup_{\theta\in(0,1)}
 \sup_{x\in\R^n}
 |\varphi_t(x,\theta t(u)+(1-\theta)t)|.
 \end{align}
 Since
 $\sup_{x\in\R^n}|\varphi_t(x,t)|<C(T-t)^{-\frac{p}{p-1}}$
 (see (2.11) in \cite{Giga-Kohn_1987} p. 7),
 the right hand side of \eqref{eq_5.6} is estimated by
 \begin{align}
 \nonumber
 \sup_{x\in\R^n}
 |\varphi_t(x,\underbrace{\theta t+(1-\theta)t(u)}_{=t+(1-\theta)(t(u)-t)})|
 &<
 C
 \{
 T-t-(1-\theta)(t(u)-t)
 \}^{-\frac{p}{p-1}}
 \\
 \label{eq_5.7}
 &<
 C
 \{T-t-(1-\theta)(T(u)-T)\}^{-\frac{p}{p-1}}.
 \end{align}
 We here used a relation $t(u)-t=T(u)-T$ (see \eqref{eq_5.4}).
 From \eqref{eq_5.1},
 for any fixed $\tau_1\in(-\log T,\infty)$,
 there exists $\delta_1\in(0,1)$ such that
 if $\|u_0-\varphi_0\|_{L_x^\infty(\R^n)}<\delta_1$,
 then
 \begin{align}
 \label{eq_5.8}
 |T(u)-T|
 <
 \tfrac{e^{-\tau}}{4}
 =
 \tfrac{T-t}{4}
 \qquad
 \text{for }
 \tau\in(-\log T,\tau_1).
 \end{align}
 Hence
 if $\|u_0-\varphi_0\|_{L_x^\infty(\R^n)}<\delta_1$,
 we get from \eqref{eq_5.7} - \eqref{eq_5.8} that
 \begin{align}
 \label{eq_5.9}
 \sup_{x\in\R^n}
 |\varphi_t(x,\theta t+(1-\theta)t(u))|
 <
 C
 e^{\frac{p\tau}{p-1}}
 \qquad
 \text{for }
 \tau\in(-\log T,\tau_1).
 \end{align}
 Therefore
 if $\|u_0-\varphi_0\|_{L_x^\infty(\R^n)}<\delta_1$,
 we deduce from \eqref{eq_5.6} and \eqref{eq_5.9} that
 \begin{align}
 \label{eq_5.10}
 \sup_{z\in\R^n}
 |
 \varphi(e^{-\frac{\tau}{2}}z,t(u))-\varphi(e^{-\frac{\tau}{2}}z,t)
 |
 <
 C
 e^{\frac{p\tau}{p-1}}
 |T(u)-T|
 \end{align}
 for
 $\tau\in(-\log T,\tau_1)$.
 Combining \eqref{eq_5.2}, \eqref{eq_5.5} and \eqref{eq_5.10},
 we conclude that
 for any fixed $\tau_1\in(-\log T,\infty)$ and any $\epsilon\in(0,1)$,
 there exists $\delta_1\in(0,1)$ such that
 if $\|u_0-\varphi_0\|_{L_x^\infty(\R^n)}<\delta_1$,
 we have
 \begin{align}
 \nonumber
 \sup_{z\in\R^n}
 &
 |w(z,\tau)-\Phi(z,\tau)|
 \\
 \nonumber
 &<
 C
 e^{-\frac{\tau}{p-1}}
 \underbrace{
 \sup_{x\in\R^n}
 |u(x,t(u))-\varphi(x,t(u))|
 }_{<\epsilon}
 \\
 \label{eq_5.11}
 &\quad
 +
 Ce^{\tau_1}
 \underbrace{
 |T(u)-T|
 }_{<\epsilon}
 \qquad
 \text{for }
 \tau\in(-\log T,\tau_1).
 \end{align}
 Furthermore
 since
 \begin{align*}
 w^\bot&(z,\tau)
 -
 \Phi^\bot(z,\tau)
 =
 w(z,\tau)
 -
 \Phi(z,\tau)
 \\
 &\quad
 +
 \langle w(z,\tau)-\Phi(z,\tau),H_0\rangle_\rho
 H_0
 \\
 &\quad
 +
 \langle w(z,\tau)-\Phi(z,\tau),{\bf H}_1\rangle_\rho
 \cdot
 {\bf H}_1(z)
 \\
 &\quad
 +
 \langle w(z,\tau)-\Phi(z,\tau),{\bf H}_2\rangle_\rho
 \cdot
 {\bf H}_2(z),
 \end{align*}
 it holds that
 \begin{align}
 \nonumber
 \sup_{|z|<\tau^\frac{3}{16}}
 &|w^\bot(z,\tau)
 -
 \Phi^\bot(z,\tau)|
 \\
 \nonumber
 &<
 C
 (
 \sup_{z\in\R^n}
 |w(z,\tau)
 -
 \Phi(z,\tau)|
 )
 \sup_{|z|<\tau^\frac{3}{16}}
 (1+|z|^2)
 \\
 \label{eq_5.12}
 &<
 C
 \tau^\frac{3}{8}
 \sup_{z\in\R^n}
 |w(z,\tau)
 -
 \Phi(z,\tau)|
 \end{align}
 and
 \begin{align}
 \label{eq_5.13}
 \|w^\bot(z,\tau)-\Phi^\bot(z,\tau)\|_\rho
 <
 C
 \sup_{z\in\R^n}
 |w(z,\tau)
 -
 \Phi(z,\tau)|.
 \end{align}
 Therefore
 from \eqref{eq_5.11} - \eqref{eq_5.13} and (s1) - (s5),
 for any $\epsilon\in(0,1)$,
 there exists $\sigma_1\in(-\log T,\infty)$
 such that for any $\sigma\in(\sigma_1,\infty)$
 there exists $\delta_1=\delta_1(\sigma)\in(0,1)$
 such that
 if
 $\|u_0-\varphi_0\|_{L_x^\infty(\R^n)}<\delta_1$,
 then it holds that
 \begin{align}
 \nonumber
 &
 \|
 w(z,\tau)-\kappa+\tfrac{1}{c_p\tau}\sum_{J=1}^nH_{2,JJ}
 \|_\rho
 \\
 \nonumber
 &<
 \|
 \Phi(z,\tau)-\kappa+\tfrac{1}{c_p\tau}\sum_{J=1}^nH_{2,JJ}
 \|_\rho
 +
 \|
 w(z,\tau)-\Phi(z,\tau)
 \|_\rho
 \\
 \label{eq_5.14}
 &<
 \tfrac{\epsilon}{2\tau}
 \quad
 \text{for }
 \tau\in(\sigma-1,\sigma+1),
 \end{align}
 \begin{align}
 \label{eq_5.15}
 \begin{cases}
 w(z,\tau)
 >
 \Phi(z,\tau)
 -
 |w(z,\tau)-\Phi(z,\tau)|
 \\
 \hspace{16.25mm}
 >
 \kappa-2e^{-\frac{7\tau}{8(p-1)}}
 \quad
 \text{for }
 \tau\in(\sigma-1,\sigma+1),
 \\
 w(z,\tau)
 <
 \Phi(z,\tau)
 +
 |w(z,\tau)-\Phi(z,\tau)|
 \\
 \hspace{16.25mm}
 <
 \kappa+\tfrac{4n\kappa}{p\tau}
 \quad
 \text{for }
 \tau\in(\sigma-1,\sigma+1),
 \end{cases}
 \end{align}
 \begin{align}
 \nonumber
 \sup_{|z|>\sqrt{3n}}
 |w(z,\tau)|
 &<
 \sup_{|z|>\sqrt{3n}}
 |\Phi(z,\tau)|
 +
 \sup_{|z|>\sqrt{3n}}
 |w(z,\tau)-\Phi(z,\tau)|
 \\
 \label{eq_5.16}
 &<
 \kappa
 -
 \tfrac{c_1}{2\tau}
 \quad
 \text{for }
 \tau\in(\sigma-1,\sigma+1),
 \\
 \nonumber
 \|w^\bot(z,\tau)\|_{L_\rho(\R^n)}
 &<
 \|\Phi^\bot(z,\tau)\|_{L_\rho(\R^n)}
 +
 \|w^\bot(z,\tau)-\Phi^\bot(z,\tau)\|_{L_\rho(\R^n)}
 \\
 \label{eq_5.17}
 &<
 2c_2\tau^{-2}
 \quad
 \text{for }
 \tau\in(\sigma-1,\sigma+1),
 \\
 \nonumber
 \sup_{|z|<\tau^\frac{1}{16}}
 |w^\bot(z,\tau)|
 &<
 \sup_{|z|<\tau^\frac{1}{16}}
 |\Phi^\bot(z,\tau)|
 +
 \sup_{|z|<\tau^\frac{1}{16}} 
 |w^\bot(z,\tau)-\Phi^\bot(z,\tau)|
 \\
 \label{eq_5.18}
 &<
 2\tau^{-\frac{3}{2}}
 \quad
 \text{for }
 \tau\in(\sigma-1,\sigma+1).
 \end{align}
 Estimates \eqref{eq_5.14} - \eqref{eq_5.18}
 prove (t1) - (t5).
 We rewrite \eqref{eq_5.14} as
 \begin{align}
 \label{eq_5.19}
 \|
 \tau\{w(z,\tau)-\kappa\}+\tfrac{1}{c_p}&\sum_{J=1}^nH_{2,JJ}
 \|_{L_\rho^2(\R^n)}
 <
 \tfrac{\epsilon}{2}
 \end{align}
 for $\tau\in(\sigma-1,\sigma+1)$.
 Hence
 from \eqref{eq_5.19},
 we can apply Lemma \ref{lemma_4.1} to a function
 $q(z)=\tau\{w(z,\tau)-\kappa\}$ for any fixed $\tau\in(\sigma-1,\sigma+1)$.
 There exists a function
 $\xi(\tau)\in C([\sigma-1,\sigma+1];\R^n)\cap C^1((\sigma-1,\sigma+1);\R^n)$
 satisfying
 \begin{align}
 \label{eq_5.20}
 \underbrace{
 \langle w(z+\xi(\tau),\tau)-\kappa,H_{1,J}\rangle_\rho
 }_{=\langle w(z+\xi(\tau),\tau),H_{1,J}\rangle_\rho}
 =
 0
 \quad
 &\text{for }
 \tau\in[\sigma-1,\sigma+1],
 \\
 \label{eq_5.21}
 |\xi(\tau)|<c_1\epsilon
 \quad
 &\text{for }
 \tau\in[\sigma-1,\sigma+1].
 \end{align}
 The constant $c_1$ does not depend on $\epsilon,\sigma$.
 We put
 \begin{align}
 \label{eq_5.22}
 v(z,\tau)
 =
 w(z+\xi(\tau),\tau).
 \end{align}
 We next confirm the following.
 \begin{itemize}
 \item 
 (claim 2)
 \quad
 Let $v(z,\tau)$ be defined in \eqref{eq_5.22}.
 There exists $C>0$ independent of $\epsilon,\sigma$
 such that
 \begin{enumerate}[(v1) ]
 \item 
 $\dis\|v(z,\tau)-\kappa+\tfrac{1}{c_p\tau}\sum_{J=1}^nH_{2,JJ}\|_\rho
 <C\epsilon\tau^{-1}$
 \quad
 for $\tau\in(\sigma-1,\sigma+1)$,

 \item 
 $-2e^{-\frac{7\tau}{8(p-1)}}<v(z,\tau)
 <\kappa+\tfrac{4n\kappa}{p\tau}$
 \quad
 for $\tau\in(\sigma-1,\sigma+1)$,

 \item 
 $\dis\sup_{|z|>2\sqrt{n}}
 v(z,\tau)<\kappa$
 \quad
 for $\tau\in(\sigma-1,\sigma+1)$,

 \item 
 $\dis\|v^\bot(z,\tau)\|_{L_\rho^2(\R^n)}
 <\tau^{-\frac{11}{8}}$
 \quad
 for $\tau\in(\sigma-1,\sigma+1)$,

 \item 
 $\dis\sup_{|z|<\frac{1}{2}\tau^\frac{1}{16}}
 |v^\bot(z,\tau)|<4\tau^{-\frac{3}{2}}$
 \quad
 for $\tau\in(\sigma-1,\sigma+1)$,

 \item 
 $\|\nabla_zv^\bot(z,\tau)\|_{L_\rho^2(\R^n)}
 <
 \tau^{-\frac{5}{4}}$
 \quad
 for $\tau\in(\sigma,\sigma+1)$.
 \end{enumerate}
 \end{itemize}
 From the definition of $v(z,\tau)$,
 (v2) is obvious (see (t2)).
 Furthermore from (t3) and $|\xi(\tau)|<c_1\epsilon$ (see \eqref{eq_5.21}),
 (v3) follows.
 We decompose $w(z,\tau)$ as
 \begin{align*}
 w(z,\tau)
 -
 \kappa
 =
 b_0(\tau)H_0(z)
 +
 {\bf b}_1(\tau)\cdot{\bf H}_1(z)
 +
 {\bf b}_2(\tau)\cdot{\bf H}_2(z)
 +
 w^\bot(z,\tau).
 \end{align*}
 From the definition of $v(z,\tau)$,
 we see that
 \begin{align}
 \nonumber
 v&(z,\tau)
 =
 w(z+\xi,\tau)
 \\
 \label{eq_5.23}
 &=
 \underbrace{
 \kappa
 +
 b_0H_0
 +
 {\bf b}_1\cdot{\bf H}_1(z+\xi)
 +
 {\bf b}_2\cdot{\bf H}_2(z+\xi)
 }_{\in\text{span}\{H_0,H_{1,J},H_{2,IJ}\}}
 +
 w^\bot(z+\xi,\tau).
 \end{align}
 Hence
 it follows that
 \begin{align}
 \nonumber
 v^\bot&(z,\tau)
 =
 w^\bot(z+\xi,\tau)
 -
 \langle w^\bot(z+\xi,\tau),H_0\rangle_\rho
 H_0
 \\
 \label{eq_5.24}
 &\quad
 -
 \langle w^\bot(z+\xi,\tau),{\bf H}_1\rangle_\rho
 \cdot{\bf H}_1
 -
 \langle w^\bot(z+\xi,\tau),{\bf H}_2\rangle_\rho
 \cdot{\bf H}_2.
 \end{align}
 For $H\in\{H_0,H_{1,J},H_{2,IJ}\}$,
 we easily verify from (t4) that
 \begin{align}
 \label{eq_5.25}
 |\langle w^\bot(z+\xi,\tau),H_0\rangle_\rho|
 &<
 C
 \|w^\bot(z,\tau)\|_\rho
 \\
 \nonumber
 &<
 C\tau^{-2}
 \qquad
 \text{for }
 \tau\in(\sigma-1,\sigma+1).
 \end{align}
 Therefore
 from \eqref{eq_5.24} - \eqref{eq_5.25} and (t5),
 we obtain (v5).
 Furthermore
 we can derive (v4) from \eqref{eq_5.24}, (v2) and (v5).
 We now prove (v1).
 From the definition of $H_0$, $H_{1,J}$ and $H_{2,IJ}$ (see Section \ref{sec_2}),
 we see that
 \begin{align*}
 H_0(z+\xi)
 &=
 H_0(z)
 =
 H_0,
 \\
 H_{1,J}(z+\xi)
 &=
 \tfrac{{\sf k}_0^n}{\sqrt{2}}(z_J+\xi_J)
 =
 H_{1,J}(z)
 +
 \tfrac{\xi_J}{\sqrt{2}}H_0,
 \\
 H_{2,JJ}(z+\xi)
 &=
 \tfrac{{\sf k}_0^n}{2\sqrt{2}}\{(z_J+\xi_J)^2-2\},
 \\
 &=
 H_{2,JJ}(z)
 +
 \xi_J
 H_{1,J}(z)
 +
 \tfrac{\xi_J^2}{2\sqrt{2}}
 H_0,
 \\
 H_{2,IJ}(z+\xi)
 &=
 \tfrac{{\sf k}_0^n}{2}
 (z_I+\xi_I)(z_J+\xi_J)
 \\
 &=
 H_{2,IJ}(z)
 +
 \tfrac{\xi_J}{\sqrt{2}}
 H_{1,I}(z)
 +
 \tfrac{\xi_I}{\sqrt{2}}
 H_{1,J}(z)
 +
 \tfrac{\xi_I\xi_J}{2}
 H_0.
 \end{align*}
 Since $|\xi(\tau)|<c_1\epsilon$ (see \eqref{eq_5.21}),
 combining these relations, \eqref{eq_5.23} and (v4),
 we can derive (v1).
 Hence
 (v1) and (v4) imply
 \begin{align}
 \label{eq_5.26}
 \|v(z,\tau)-\kappa\|_\rho
 <
 \tfrac{2n}{c_p\tau}
 \qquad
 \text{for }
 \tau\in
 (\sigma-1,\sigma+1).
 \end{align}
 To obtain (v6),
 we apply a standard energy inequality.
 From the definition of $v(z,\tau)$ (see \eqref{eq_5.22}),
 we verify that $v(z,\tau)$ solves
 \begin{align}
 \label{eq_5.27}
 v_\tau
 =
 A_zv
 -
 \tfrac{v}{p-1}
 +
 |v|^{p-1}v
 +
 (
 \underbrace{
 \tfrac{d\xi}{d\tau}
 -
 \tfrac{\xi}{2}
 }_{={\bf c}(\tau)}
 )
 \cdot
 \nabla_zv.
 \end{align}
 Furthermore we put $q(z,\tau)=v(z,\tau)-\kappa$.
 A function $q(z,\tau)$ satisfies
 \begin{align}
 \label{eq_5.28}
 q_\tau
 =
 A_zq
 +
 q
 +
 N(q)
 +
 {\bf c}(\tau)
 \cdot
 \nabla_zq,
 \end{align}
 here $N(q)=|\kappa+q|^{p-1}(\kappa+q)-\kappa^p-p\kappa^{p-1}q$.
 Taking an inner product $\langle\cdot,H_1\rangle$ in \eqref{eq_5.28},
 we get from \eqref{eq_5.20} that
 \begin{align}
 \label{eq_5.29}
 {\bf c}(\tau)
 \cdot
 \langle\nabla_zq,H_{1,J}\rangle_\rho
 =
 -\langle & N(q),H_{1,J} \rangle_\rho
 \\
 \nonumber
 &
 \text{for }
 \tau\in(\sigma-1,\sigma+1).
 \end{align}
 Since the expression of \eqref{eq_5.29} is the same as \eqref{eq_4.16},
 we can borrow estimates given in \eqref{eq_4.16} - \eqref{eq_4.21}.
 In fact,
 since $v(z,\tau)$ satisfies (v1),
 we obtain
 (see \eqref{eq_4.16} - \eqref{eq_4.21})
 \begin{align}
 \label{eq_5.30}
 |{\bf c}(\tau)|
 <
 C\tau
 |\langle N(q),{\bf H}_1\rangle_\rho|
 \qquad
 \text{for }
 \tau\in(\sigma-1,\sigma+1).
 \end{align}
 To get estimates of $|\langle N,{\bf H}_1\rangle_\rho|$ in \eqref{eq_4.17},
 we used the bound of $\|W\|_{H_\rho^1(\R^n)}$.
 However,
 we have not yet obtained the bound of $\|q\|_{H_\rho^1(\R^n)}$
 in this step.
 We now estimate $|\langle N(q),{\bf H}_1\rangle_\rho|$
 from (v1) - (v5).
 Since $|v(z,\tau)|$ is bounded by a universal constant (see (v2)),
 Lemma \ref{lemma_A.3} implies
 $|N(q)|<Cq^2$.
 We decompose $q(z,\tau)$ as
 $q(z,\tau)=\langle q(z,\tau),H_0\rangle_\rho H_0
 +\langle q(z,\tau),{\bf H}_2\rangle_\rho\cdot{\bf H}_2+q^\bot(z,\tau)$,
 and use (v2) and (v5) to get
 \begin{align*}
 &|\langle N(q),H_{1,J}\rangle_\rho|
 <
 C
 |\langle q^2,H_{1,J}\rangle_\rho|
 \\
 &<
 C
 \underbrace{
 \|(\langle q(\tau),H_0\rangle_\rho H_0
 +\langle q(\tau),{\bf H}_2\rangle_\rho\cdot{\bf H}_2)^2
 \|_\rho
 }_{<C\tau^{-2}}
 \|H_{1,J}\|_\rho
 \\
 &\quad
 +
 C
 |\langle(q^\bot)^2,H_{1,J}\rangle_\rho|
 \\
 &<
 \tfrac{C}{\tau^2}
 +
 C
 \underbrace{
 |
 \langle
 (q^\bot)^2{\bf 1}_{|z|<\tau^\frac{1}{32}},
 H_{1,J}\rangle_\rho|}_{<C\tau^{-3} \ (\text{see } (v5))}
 +
 C
 |\langle
 (q^\bot)^2{\bf 1}_{|z|>\tau^\frac{1}{32}},
 H_{1,J}\rangle_\rho
 |.
 \end{align*}
 Using
 $q(z,\tau)=\langle q(z,\tau),H_0\rangle_\rho H_0
 +\langle q(z,\tau),{\bf H}_2\rangle_\rho\cdot{\bf H}_2+q^\bot(z,\tau)$
 again in the last term,
 we get from (v2) that
 \begin{align*}
 |\langle
 (q^\bot)^2{\bf 1}_{|z|>\tau^\frac{1}{32}},
 &\,H_{1,J}\rangle_\rho|
 <
 C\langle q(z,\tau),H_0\rangle_\rho^2
 +
 C|\langle q(z,\tau),{\bf H}_2\rangle_\rho|^2
 \\
 &\quad
 +
 C
 |\langle
 {\bf 1}_{|z|>\tau^\frac{1}{32}},
 H_{1,J}\rangle_\rho|
 <
 C\tau^{-2}.
 \end{align*}
 Therefore
 it follows that
 \begin{align*}
 |\langle N(q),H_{1,J}\rangle_\rho|
 <
 \tfrac{C}{\tau^2}
 \qquad
 \text{for }
 \tau\in(\sigma-1,\sigma+1).
 \end{align*}
 Applying this estimate in \eqref{eq_5.30},
 we obtain
 \begin{align}
 \label{eq_5.31}
 |{\bf c}(\tau)|
 <
 C\tau^{-1}
 \qquad
 \text{for }
 \tau\in(\sigma-1,\sigma+1).
 \end{align}
 From \eqref{eq_5.28},
 $q^\bot(z,\tau)$ satisfies
 \begin{align}
 \label{eq_5.32}
 q_\tau^\bot
 =
 A_zq^\bot
 +
 q^\bot
 +
 N(q)^\bot
 +
 {\bf c}(\tau)
 \cdot
 \nabla_zq^\bot.
 \end{align}
 Multiplying \eqref{eq_5.32}
 by $q^\bot$,
 we get
 \begin{align}
 \nonumber
 \tfrac{1}{2}
 \tfrac{d}{d\tau}
 \|q^\bot\|_\rho^2
 &<
 -
 \|\nabla_zq^\bot\|_\rho^2
 +
 \|q^\bot\|_\rho^2
 +
 \|N(q)\|_\rho^2
 \\
 \nonumber
 &\quad
 +
 |{\bf c}(\tau)|
 \|\nabla_zq^\bot\|_\rho
 \|q^\bot\|_\rho
 \\
 \label{eq_5.33}
 &<
 -
 \tfrac{1}{8}
 \|\nabla_zq^\bot\|_\rho^2
 -
 \tfrac{1}{16}
 \|q^\bot\|_\rho^2
 +
 \|N(q)\|_\rho^2.
 \end{align}
 The last term on the right hand side of \eqref{eq_5.33}
 is bounded by
 $\|N(q)\|_\rho^2<C\|q^2\|_\rho^2$
 and
 \begin{align*}
 \|q^2\|_\rho^2
 &=
 \langle q^4{\bf 1}_{|z|<\tau^\frac{1}{32}},1\rangle_\rho
 +
 \langle q^4{\bf 1}_{|z|>\tau^\frac{1}{32}},1\rangle_\rho
 \\
 \nonumber
 &<
 C
 \underbrace{
 (|b_0|+|{\bf b}_1|+|{\bf b}_2|)^4
 }_{C\tau^{-4}}
 +
 \underbrace{
 \langle (q^\bot)^4{\bf 1}_{|z|<\tau^\frac{1}{32}},1\rangle_\rho
 }_{\tau^{-3}\|q^\bot\|_\rho^2 \ (\text{see } (v5))}
 \\
 &\quad
 +
 C
 \langle {\bf 1}_{|z|>\tau^\frac{1}{32}},1\rangle_\rho.
 \end{align*}
 Hence
 we deduce that
 \begin{align}
 \label{eq_5.34}
 \|N(q)\|_\rho^2
 &<
 C\tau^{-4}
 +
 C\tau^{-3}
 \|q^\bot\|_\rho^2.
 \end{align}
 Substituting \eqref{eq_5.34} to \eqref{eq_5.33},
 we obtain
 \begin{align}
 \label{eq_5.35}
 \tfrac{1}{2}
 \tfrac{d}{d\tau}
 \|q^\bot\|_\rho^2
 &<
 -
 \tfrac{1}{8}
 \|\nabla_zq^\bot\|_\rho^2
 -
 \tfrac{1}{32}
 \|q^\bot\|_\rho^2
 +
 C\tau^{-4}.
 \end{align}
 Let $s\in[\sigma,\sigma+1]$.
 Integrating \eqref{eq_5.35} over $\tau\in(s-1,s)$,
 we get
 \begin{align}
 \nonumber
 \tfrac{1}{8}
 \int_{s-1}^s
 \|\nabla_zq^\bot(\tau)\|_\rho^2
 d\tau
 &<
 \tfrac{1}{2}
 \|q^\bot(s-1)\|_\rho^2
 +
 \int_{s-1}^s
 \tfrac{C}{\tau^4}d\tau
 \\
 \label{eq_5.36}
 &<
 \tfrac{1}{2}
 \|q^\bot(s-1)\|_\rho^2
 +
 \tfrac{C}{s^4}.
 \end{align}
 We next multiply \eqref{eq_5.32} by $(\tau-s+1)\cdot q^\bot_\tau(z,\tau)$.
 Then
 we have
 \begin{align}
 \nonumber
 (\tau-&\,s+1)
 \|q_\tau^\bot(\tau)\|_\rho^2
 +
 \tfrac{d}{d\tau}
 (
 \tfrac{\tau-s+1}{2}
 \|\nabla_zq^\bot(\tau)\|_\rho^2
 )
 \\
 \label{eq_5.37}
 &<
 \tfrac{1}{2}
 \|\nabla_zq^\bot(\tau)\|_\rho^2
 +
 \tfrac{d}{d\tau}
 (
 \tfrac{\tau-s+1}{2}
 \|q^\bot(\tau)\|_\rho^2
 )
 -
 \tfrac{1}{2}
 \|q^\bot(\tau)\|_\rho^2
 \\
 \nonumber
 &\quad
 +
 (\tau-s+1)
 \|N(\tau)\|_\rho
 \|q_\tau^\bot(\tau)\|_\rho
 \\
 \nonumber
 &\quad
 +
 (\tau-s+1)
 |{\bf c}(\tau)|
 \|\nabla_zq^\bot(\tau)\|_\rho
 \|q_\tau^\bot(\tau)\|_\rho.
 \end{align}
 We apply \eqref{eq_5.31} and \eqref{eq_5.34} in \eqref{eq_5.37}
 to get
 \begin{align}
 \nonumber
 \tfrac{(\tau-s+1)}{4}
 &
 \|q_\tau^\bot(\tau)\|_\rho^2
 +
 \tfrac{d}{d\tau}
 (
 \tfrac{\tau-s+1}{2}
 \|\nabla_zq^\bot(\tau)\|_\rho^2
 )
 \\
 \label{eq_5.38}
 &<
 \|\nabla_zq^\bot(\tau)\|_\rho^2
 +
 \tfrac{d}{d\tau}
 (
 \tfrac{\tau-s+1}{2}
 \|q^\bot(\tau)\|_\rho^2
 )
 -
 \tfrac{1}{4}
 \|q^\bot(\tau)\|_\rho^2
 \\
 \nonumber
 &\quad
 +
 \tfrac{C}{\tau^4}
 \qquad
 \text{for }
 \tau\in(\sigma-1,\sigma+1).
 \end{align}
 Integrating \eqref{eq_5.37} over $\tau\in(s-1,s)$,
 we obtain
 \begin{align}
 \tfrac{1}{2}
 \|\nabla_zq^\bot(s)\|_\rho^2
 \label{eq_5.39}
 &<
 \int_{s-1}^s
 \|\nabla_zq^\bot(s)\|_\rho^2
 ds
 +
 \tfrac{1}{2}
 \|q^\bot(s)\|_\rho^2
 +
 \tfrac{C}{s^4}.
 \end{align}
 Estimates \eqref{eq_5.39}, \eqref{eq_5.36}
 and (v4) shows (v6).
 Our argument until now is summarized as follows.
 \begin{enumerate}[step 1]
 \item
 Let $\varphi(x,t)$ be a solution of 
 $\varphi_t=\Delta_x\varphi+|\varphi|^{p-1}\varphi$
 satisfying (a1) - (a6).
 \item
 Let $u(x,t)$ be a solution of
 $u_t=\Delta_xu+|u|^{p-1}u$ with $u|_{t=0}=u_0$,
 and put
 $w(z,\tau)=e^{-\frac{\tau}{p-1}}u(ze^{-\frac{\tau}{2}},\tau)$.
 \item
 \label{step3}
 For any $\epsilon\in(0,1)$,
 there exists $\sigma_1\in(1,\infty)$
 such that
 for any $\sigma>\sigma_1$ there exists $\delta_1=\delta_1(\sigma)$
 such that
 if $\|u_0-\varphi_0\|_{L_x^\infty(\R^n)}<\delta_1$,
 then
 there exists a function
 $\xi(\tau)\in C([\sigma-1,\sigma+1];\R^n)\cap C^1((\sigma-1,\sigma+1);\R^n)$
 such that
 \begin{align*}
 \langle w(z+\xi(\tau),\tau),H_{1,J}\rangle_\rho
 =0
 \qquad
 \text{for }
 \tau\in[\sigma-1,\sigma+1]
 \end{align*}
 and 
 $v(z,\tau)=w(z+\xi(\tau),\tau)$ satisfies (v1) - (v6).
 The constant $C$ in (v1) is independent of $\epsilon,\sigma$ and $w(z,\tau)$.
 \end{enumerate}
 Our goal is to continue this procedure until $\tau=\infty$.
 Let $\nu_1,\nu_2\in(0,1)$
 be constants
 given in Proposition \ref{proposition_4.3}, Proposition \ref{proposition_4.4},
 respectively,
 and
 put
 $\nu_0=\frac{1}{2}\min\{\nu_1,\nu_2\}$.
 Furthermore
 let
 $\epsilon\in(0,1)$ be a constant given in step 3.
 We now fix
 $\epsilon=\epsilon_0=\nu_0^5$.
 As a consequence of (v1) - (v6),
 there exist $\sigma=\sigma(\epsilon_0)$
 and $\delta=\delta(\sigma)$ such that
 if $\|u_0-\varphi_0\|_{L_x^\infty(\R^n)}<\delta$,
 then
 ($\tilde{\sf k}$0) - ($\tilde{\sf k}$8) hold.
 \begin{enumerate}[($\tilde{\sf k}1$) ]
 \setlength{\leftskip}{5mm}
 \item[($\tilde{\sf k}0$) ] 
 \label{tildek}
 $|\langle v(z,\tau)-\kappa,H_0\rangle_\rho|
 <
 \nu_0^4\tau^{-1}$
 \quad
 for $\tau\in(\sigma,\sigma+1)$
 \item 
 $|\frac{1}{\langle v(z,\tau),H_{2,JJ}\rangle_\rho}+c_p\tau|
 <\nu_0c_p\tau$
 \quad
 for $\tau\in(\sigma,\sigma+1)$
 and all $J$,

 \item 
 $|\langle v(z,\tau),H_{2,IJ}\rangle_\rho|<\nu_0^2\tau^{-1}$
 \quad
 for $\tau\in(\sigma,\sigma+1)$
 and all $I<J$,

 \item 
 $\|v^\bot(\tau)\|_\rho<\nu_0^4\tau^{-1}$
 \quad
 for $\tau\in(\sigma,\sigma+1)$,

 \item 
 $\|\pa_{z_J}v^\bot(\tau)\|_\rho<\nu_0^4\tau^{-1}$
 \quad
 for $\tau\in(\sigma,\sigma+1)$
 and all $J$,

 \item 
 $-2e^{-\frac{7\tau}{8(p-1)}}<v(z,\tau)<\kappa+\tau^{-\frac{1}{2}}$
 \quad
 for
 $z\in\R^n$
 and $\tau\in(\sigma,\sigma+1)$,

 \item 
 $v(z,\tau)<\kappa$
 \quad 
 for
 $|z|>2\sqrt{n}$
 and
 $\tau\in(\sigma,\sigma+1)$,

 \item 
 $\dis\sup_{|z|<8R_0}|v(z,\tau)-\kappa|<
 \sup_{|z|<8R_0}\tfrac{4|\sum_{J=1}H_{2,JJ}(z)|}{c_p\tau}$
 \quad
 for $\tau\in(\sigma,\sigma+1)$,

 \item 
 $\dis\sup_{|z|<16\sqrt{n}}
 |v^\bot(z,\tau)|<4\tau^{-\frac{3}{2}}$
 \quad
 for $\tau\in(\sigma,\sigma+1)$,
 \end{enumerate}
 where
 $R_0=\max\{R_1(\nu_0),R_2(\nu_0)\}$
 ($R_1,R_2$ are constants given in
 Proposition \ref{proposition_4.3}, Proposition \ref{proposition_4.3},
 respectively).
 Furthermore
 from the definition of $v(z,\tau)$,
 we recall that
 \begin{align*}
 \langle v(z,\tau),H_{1,J}\rangle_\rho
 =0
 \qquad
 \text{for }
 \tau\in[\sigma,\sigma+1].
 \end{align*}
 Let $\nu\in(0,1)$ be a constant given in Lemma \ref{lemma_4.1}.
 We can assume $\nu_0\ll\nu$.
 We finally prove claim 3.
 \begin{itemize}
 \item
 (claim 3)
 \quad
 There exists a unique function $\xi(\tau)\in C^1((\sigma,\infty),\R^n)$
 such that
 \begin{align}
 \label{eq_5.40}
 \langle w(z+\xi(\tau),\tau),H_{1,J}\rangle_\rho=0
 \qquad
 \text{for }
 \tau\in(\sigma,\infty).
 \end{align}
 \end{itemize}
 We define $\tau_\text{a}>\sigma$ by
 \begin{align}
 \label{eq_5.41}
 \tau_\text{a}
 &=
 \sup\{
 s>\sigma;
 \xi(\tau) \text{ defined in (claim 3) can be extended}
 \\
 \nonumber
 &\hspace{35mm}
 \text{for } \tau\in[\sigma,s]
 \}.
 \end{align}
 From the construction of $v(z,\tau)$ (see step 3 p. \pageref{step3}),
 we note that $\tau_\text{a}\geq\sigma+1$.
 Assume that $\tau_\text{a}<\infty$.
 We begin
 with the estimate for $\langle w(z+\xi(\tau),\tau)-\kappa,H_0\rangle_\rho$.
 We claim that
 \begin{align}
 \label{eq_5.42}
 -\tfrac{\nu_0^2}{pH_0\tau}
 <
 \langle w(z&+\xi(\tau),\tau)-\kappa,H_0\rangle_\rho
 \\
 \nonumber
 &<
 \tfrac{4p\kappa^{p-1}}{H_0}
 e^{-\frac{3\tau}{4(p-1)}}
 \qquad
 \text{for }
 \tau\in(\sigma,\tau_\text{a}). 
 \end{align}
 We first prove the right hand inequality in \eqref{eq_5.42}.
 From ($\tilde{\sf k}_5$),
 a comparison argument shows
 \begin{align}
 \label{eq_5.43}
 w(z,\tau)>-2e^{-\frac{7\tau}{8(p-1)}}
 \qquad
 \text{for }
 \tau>\sigma.
 \end{align}
 Assume that
 there exists $\sigma_1\in[\sigma,\tau_\text{a}]$
 such that
 \begin{align}
 \label{eq_5.44}
 \langle w(z+\xi(\sigma_1),\sigma_1)-\kappa,H_0\rangle_\rho
 \geq
 \tfrac{4p\kappa^{p-1}}{H_0}e^{-\frac{3\sigma_1}{4(p-1)}}.
 \end{align}
 We put
 \begin{align*}
 {\sf w}_1(z,\tau)
 &=
 w(z+e^\frac{\tau-\sigma_1}{2}\xi(\sigma_1),\tau).
 \end{align*}
 A function ${\sf w}_1(z,\tau)$ satisfies
 $\pa_\tau{\sf w}_1=A_z{\sf w}_1-\frac{{\sf w}_1}{p-1}+|{\sf w}_1|^{p-1}{\sf w}_1$.
 From the definition of ${\sf w}_1(z,\tau)$,
 it is clear that ${\sf w}_1(z,\sigma_1)=w(z+\xi(\sigma_1),\sigma_1)$.
 Hence it follows from \eqref{eq_5.44} that
 \begin{align*}
 \langle {\sf w}_1(z,\sigma_1)-\kappa,H_0\rangle_\rho
 \geq
 \tfrac{4p\kappa^{p-1}}{H_0}e^{-\frac{3\sigma_1}{4(p-1)}}.
 \end{align*}
 Furthermore
 \eqref{eq_5.43} is equivalent to
 ${\sf w}_1(z,\tau)>-2e^{-\frac{7\tau}{8(p-1)}}$
 for $(z,\tau)\in\R^n\times(\sigma,\infty)$.
 Therefore
 Proposition \ref{proposition_4.2} implies that
 ${\sf w}_1(z,\tau)$ blows up in a finite time $\tau=\tau_\text{b}<\infty$.
 This contradicts
 the fact that ${\sf w}_1(z,\tau)$ is defined on $\tau\in(\sigma,\infty)$.
 Therefore
 the right hand inequality in \eqref{eq_5.42} is proved.
 We next prove the left hand inequality in \eqref{eq_5.42}.
 Assume that
 there exists $\sigma_1\in[\sigma,\tau_\text{a}]$
 such that
 \begin{align}
 \label{eq_5.45}
 \begin{cases}
 \langle w(z+\xi(\tau),\tau)-\kappa,H_0\rangle_\rho
 >
 -\tfrac{\nu_0^2}{pH_0\tau}
 \quad
 \text{for }
 \tau\in(\sigma,\sigma_1),
 \\
 \langle w(z+\xi(\sigma_1),\sigma_1)-\kappa,H_0\rangle_\rho
 =
 -\tfrac{\nu_0^2}{pH_0\tau}.
 \end{cases}
 \end{align}
 We put
 \begin{align*}
 {\sf v}(z,\tau)
 =
 w(z+\xi(\tau),\tau).
 \end{align*}
 We recall that ${\sf v}(z,\tau)$ satisfies
 all assumptions in Proposition \ref{proposition_4.3}:
 \begin{itemize}
 \item
 (k0) for $\tau\in(\sigma,\sigma_1)$
 \quad
 (see \eqref{eq_5.45} and the right hand inequality in \eqref{eq_5.42}),
 \item
 (k1) - (k8) at $\tau=\sigma$
 \quad
 (see ($\tilde{\sf k}1$) - ($\tilde{\sf k}8$))
 and
 \item
 $\langle {\sf v}(z,\tau),H_{1,J}\rangle_\rho=0$
 for $\tau\in(\sigma,\sigma_1)$
 \quad
 (see the definition of $\tau_\text{a}$).
 \end{itemize}
 Therefore
 ${\sf v}(z,\tau)$ satisfies
 (K1) - (K8) stated in Proposition \ref{proposition_4.3}
 for $\tau\in(\sigma,\sigma_1)$,
 which implies that
 ${\sf v}(z,\tau)$ satisfies all assumptions in Proposition \ref{proposition_4.4}
 at $\tau=\sigma_1$:
 \begin{itemize}
 \item
 (j1) - (j8)
 at $\tau=\sigma_1$,
 \item
 (i2)
 at $\tau=\sigma_1$
 \quad
 (see \eqref{eq_5.45})
 and
 \item
 (i3) at $\tau=\sigma_1$
 \quad
 (see the definition of $\tau_\text{a}$).
 \end{itemize}
 We again define
 \begin{align*}
 {\sf w}_1(z,\tau)
 =
 w(z+e^\frac{\tau-\sigma_1}{2}\xi(\sigma_1),\tau).
 \end{align*}
 This function ${\sf w}_1(z,\tau)$ satisfies
 $\pa_\tau{\sf w}_1=A_z{\sf w}_1-\frac{{\sf w}_1}{p-1}+|{\sf w}_1|^{p-1}{\sf w}_1$
 and
 ${\sf w}_1(z,\sigma_1)={\sf v}(z,\sigma_1)$.
 Therefore
 Proposition \ref{proposition_4.4} implies that
 there exists $\sigma_\text{b}>\sigma_1$
 such that
 ${\sf w}_1(z,\sigma_\text{b})<\kappa$
 for $z\in\R^n$.
 As a consequence,
 $w(z,\tau)\to0$ uniformly for $z\in\R^n$,
 namely $u(x,t)$ does not blow up at $t=T(u)$.
 However this contradicts the definition of $T(u)$,
 which proves the left hand inequality of \eqref{eq_5.42}.
 Therefore \eqref{eq_5.42} is proved.
 Considering \eqref{eq_5.42} as a condition (k0)
 in Proposition \ref{proposition_4.3},
 we again apply Proposition \ref{proposition_4.3} to
 ${\sf v}(z,\tau)=w(z+\xi(\tau),\tau)$
 in $\tau\in(\sigma,\tau_\text{a})$.
 As a consequence,
 we see that
 ${\sf v}(z,\tau)=w(z+\xi(\tau),\tau)$ satisfies (K1) - (K8)
 stated in Proposition \ref{proposition_4.3}
 for $\tau\in(\sigma,\tau_\text{a})$.
 We remark that
 constants in (K1)- (K8) do not depend on $\tau_\text{a}$.
 Therefore
 we can verify that
 \begin{align}
 \label{eq_5.46}
 \|
 \tau_\text{a}
 \{w(z+\xi(\tau_\text{a}),\tau_\text{a})
 -
 \kappa
 \}
 +
 \tfrac{1}{c_p}
 \sum_{J=1}^n
 H_{2,JJ}
 \|_\rho
 <
 C\nu_0
 \ll
 \nu.
 \end{align}
 We now put
 \begin{align*}
 q(z,\tau)
 =
 w(z+\xi(\tau_\text{a}),\tau)
 -
 \kappa.
 \end{align*}
 We note that
 \eqref{eq_5.46} is equivalent to
 $\|
 \tau_\text{a}
 q(z,\tau_\text{a})
 +
 \tfrac{1}{c_p}
 \sum_{J=1}^n
 H_{2,JJ}
 \|_\rho
 \ll
 \nu$.
 Hence
 by the continuity of solutions,
 there exists $h_1\in(0,1)$ such that
 $\|
 \tau
 q(z,\tau)
 +
 \tfrac{1}{c_p}
 \sum_{J=1}^n
 H_{2,JJ}
 \|_\rho
 <\frac{\nu}{2}$
 for $\tau\in(\tau_\text{a}-h_1,\tau_\text{a}+h_1)$.
 Therefore
 from Lemma \ref{lemma_4.1},
 there exists
 $\xi_1(\tau)\in C^1((\tau_\text{a}-h_1,\tau_\text{a}-h_1);\R^n)$
 such that
 \begin{align}
 \label{eq_5.47}
 \langle q(z+\xi_1(\tau),\tau),H_{1,J}\rangle_\rho
 &=
 0
 \quad
 \text{for }
 \tau\in(\tau_\text{a}-h_1,\tau_\text{a}+h_1),
 \\
 \nonumber
 \xi_1(\tau_\text{a})
 &=
 0.
 \end{align}
 We can write \eqref{eq_5.47} as
 \begin{align*}
 \langle
 w(z+\xi(\tau_\text{a})+\xi_1(\tau),\tau)-\kappa,H_{1,J}
 \rangle_\rho
 =
 0
 \qquad
 \text{for }
 \tau\in(\tau_\text{a}-h_1,\tau_\text{a}+h_1).
 \end{align*}
 From the uniqueness part of Lemma \ref{lemma_4.1},
 it holds that
 \begin{align}
 \label{eq_5.48}
 \xi(\tau)
 =
 \xi(\tau_\text{a})+\xi_1(\tau)
 \qquad
 \text{for }
 \tau\in(\tau_\text{a}-h_1,\tau_\text{a}).
 \end{align}
 We define
 \begin{align*}
 \xi_\text{p}(\tau)
 =
 \begin{cases}
 \xi(\tau) & \text{for } \tau\in(\sigma,\tau_\text{a}),
 \\
 \xi(\tau_\text{a})+\xi_1(\tau)
 & \text{for } \tau\in(\tau_\text{a},\tau_\text{a}+h_1).
 \end{cases}
 \end{align*}
 From \eqref{eq_5.48},
 it is clear that
 $\xi_\text{p}(\tau)\in C^1((\sigma,\tau_\text{a}+h_1);\R^n)$
 and
 \begin{align*}
 \langle w(z+\xi_\text{p}(\tau),\tau),H_{1,J}\rangle_\rho=0
 \qquad
 \text{for }
 \tau\in(\sigma,\tau_\text{a}+h_1).
 \end{align*}
 This contradicts the definition of $\tau_\text{a}$ (see \eqref{eq_4.52}).
 Therefore
 (claim 3) is proved.
 For the same reason as \eqref{eq_5.42},
 the following estimate holds true.
 \begin{align*}
 -\tfrac{\nu_0^2}{pH_0\tau}
 <
 \langle w(z&+\xi(\tau),\tau)-\kappa,H_0\rangle_\rho
 \\
 &<
 \tfrac{4p\kappa^{p-1}}{H_0}
 e^{-\frac{3\tau}{4(p-1)}}
 \qquad
 \text{for }
 \tau\in(\sigma,\infty).
 \end{align*}
 Therefore
 from Proposition \ref{proposition_4.3},
 we conclude that
 ${\sf v}(z,\tau)
 =
 w(z+\xi(\tau),\tau)$
 satisfies
 (K1) - (K8) for $\tau\in(\sigma,\infty)$.
 We rewrite the relation for $\xi(\tau)$
 (see \eqref{eq_5.27} and \eqref{eq_5.31}).
 \begin{align}
 \label{eq_5.49}
 \tfrac{d\xi}{d\tau}
 -
 \tfrac{\xi}{2}
 =
 {\bf c}(\tau)
 \qquad
 \text{with }
 |{\bf c}(\tau)|<C\tau^{-1}.
 \end{align}
 Integrating \eqref{eq_5.49} over $\tau\in(\tau,\infty)$,
 we get
 \begin{align}
 \label{eq_5.50}
 \underbrace{
 \lim_{s\to\infty}
 e^{-\frac{s}{2}}
 \xi(s)
 }_{=\xi_\infty}
 -
 e^{-\frac{\tau}{2}}
 \xi(\tau)
 =
 \int_\tau^\infty
 e^{-\frac{s}{2}}
 {\bf c}(s)
 ds.
 \end{align}
 We now put
 \begin{align*}
 {\sf w}_\infty(z,\tau)
 &=
 e^{-\frac{\tau}{p-1}}
 u(ze^{-\frac{\tau}{2}}+\xi_\infty,T(u)-e^{-\tau}).
 \end{align*}
 From \eqref{eq_5.50},
 ${\sf w}_\infty(z,\tau)$ can be expressed as
 \begin{align*}
 {\sf w}_\infty(z,\tau)
 &=
 w(z+e^{\frac{\tau}{2}}\xi_\infty,T(u)-e^{-\tau})
 \\
 &=
 w(z+\xi(\tau)+
 \underbrace{e^\frac{\tau}{2}\int_\tau^\infty e^{-\frac{s}{2}}{\bf c}(s)ds
 }_{<C\tau^{-1}},
 T(u)-e^{-\tau}).
 \end{align*}
 Since
 ${\sf v}(z,\tau)=w(z+\xi(\tau),\tau)$
 satisfies (K1) - (K8) for $\tau\in(\sigma,\infty)$,
 ${\sf w}_\infty(z,\tau)$ also satisfies (K1) - (K8) for $\tau\in(\sigma,\infty)$.
 This proves Theorem \ref{thm_2}.
 A blowup point of $u(x,t)$ is given by $\xi_\infty$.

 \appendix
 \section{Elementary computation for nonlinear term}
 We here provide elementary computations
 to derive \eqref{eq_4.39} - \eqref{eq_4.40}
 in the proof of Proposition of \ref{proposition_4.3}.
 \begin{lem}
 \label{lemma_A.1}
 It holds that
 \begin{align}
 \label{equation_A.1}
 \langle H_{2,JJ},H_{2,JJ}^2 \rangle_\rho
 &=
 2
 \sqrt{2}
 {\sf k}_0^n
 \qquad
 \text{\rm for all } J,
 \\
 \label{equation_A.2}
 \langle H_{2,JJ},H_{2,IJ}^2 \rangle_\rho
 &=
 \sqrt{2}
 {\sf k}_0^n
 \qquad
 \text{\rm for all } I<J.
 \end{align}
 \end{lem}
 \begin{proof}
 We recall that
 (see Section \ref{sec_2})
 \[
 H_{2IJ}(z)
 =
 \begin{cases}
 \frac{{\sf k}_0^n}{2\sqrt{2}}
 (z_I^2-2)
 &
 \text{if } I=J,
 \\[2mm]
 \frac{{\sf k}_0^n}{2}
 z_Iz_J
 &
 \text{if } I<J.
 \end{cases}
 \]
 We compute \eqref{equation_A.1} and \eqref{equation_A.2}
 as follows.
 \begin{align}
 \nonumber
 \langle H_{2,JJ}&,H_{2,JJ}^2 \rangle_\rho
 =
 \tfrac{{\sf k}_0^{3n}}{(2\sqrt{2})^3}
 \langle
 (z_J^2-2),(z_J^2-2)^2
 \rangle_\rho
 \\
 \nonumber
 &=
 \tfrac{{\sf k}_0^{3n}}{(2\sqrt{2})^3}
 \int_{\R^n}
 (z_J^6-6z_J^4+12z_J^2-8)
 \rho
 dz
 \\
 \label{equation_A.3}
 &=
 \tfrac{{\sf k}_0^{3n}}{(2\sqrt{2})^3}
 \left(\int_{\R}e^{-\frac{s^2}{4}}ds\right)^{n-1}
 \int_{\R}
 (t^6-6t^4+12t^2-8)
 e^{-\frac{t^2}{4}}
 dt,
 \\
 \nonumber
 \langle H_{2,JJ}&,H_{2,IJ}^2 \rangle_\rho
 =
 \tfrac{{\sf k}_0^{3n}}{2^3\sqrt{2}}
 \langle
 (z_J^2-2),(z_Iz_J)^2
 \rangle_\rho
 \\
 \label{equation_A.4}
 &=
 \tfrac{{\sf k}_0^{3n}}{2^3\sqrt{2}}
 \left(\int_{\R}e^{-\frac{s^2}{4}}ds\right)^{n-2}
 \left(\int_{\R}t^2e^{-\frac{t^2}{4}}dt\right)
 \left(\int_{\R}(t^4-2t^2)e^{-\frac{t^2}{4}}dt\right).
 \end{align}
 Here we note that
 \begin{align}
 \label{equation_A.5}
 \begin{cases}
 \int_{\R}e^{-\frac{s^2}{4}}ds
 =
 \sqrt{4\pi}
 =
 k_0^{-2},
 \\
 \int_{\R}s^2e^{-\frac{s^2}{4}}ds
 =
 2\int_{\R}e^{-\frac{s^2}{4}}ds
 =
 2k_0^{-2},
 \\
 \int_{\R}s^4e^{-\frac{s^2}{4}}ds
 =
 6\int_{\R}s^2e^{-\frac{s^2}{4}}ds
 =
 12k_0^{-2},
 \\
 \int_{\R}s^6e^{-\frac{s^2}{4}}ds
 =
 10\int_{\R}s^4e^{-\frac{s^2}{4}}ds
 =
 120k_0^{-2}.
 \end{cases}
 \end{align}
 Substituting \eqref{equation_A.5} to \eqref{equation_A.3},
 we obtain \eqref{equation_A.1}.
 In the same way,
 we can derive \eqref{equation_A.2} from \eqref{equation_A.4}.
 \end{proof}

 \begin{lem}
 \label{lemma_A.2}
 For ${\bf b}_2=({\bf b}_{2,JJ},{\bf b}_{2,IJ})\in\R^{n+\frac{n^2-n}{2}}$,
 it holds that
 \begin{align}
 \label{equation_A.6}
 \langle({\bf b}_2\cdot{\bf H}_2)^2,H_{2,JJ}\rangle_\rho
 &=
 2\sqrt{2}{\sf k}_0^n
 b_{2,JJ}^2
 +
 O(|{\bf b}_{2,IJ}|^2)
 \qquad
 \text{\rm for all }
 J,
 \\
 \label{equation_A.7}
 \langle({\bf b}_2\cdot{\bf H}_2)^2,H_{2,IJ}\rangle_\rho
 &=
 2\sqrt{2}
 {\sf k}_0^n
 (b_{2,II}+b_{2,JJ})b_{2,IJ}
 +
 O(|{\bf b}_{2,IJ}|^2)
 \\
 \nonumber
 &\quad
 \text{\rm for all }
 I<J.
 \end{align}
 \end{lem}
 \begin{proof}
 We begin with \eqref{equation_A.6}.
 We write
 ${\bf b}_2=({\bf b}_{2,LL},{\bf b}_{2,MN})\in\R^{n+\frac{n^2-n}{2}}$
 in this proof.
 \begin{align*}
 \langle&({\bf b}_2\cdot{\bf H}_2)^2,H_{2,JJ}\rangle_\rho
 \\
 &=
 \langle
 ({\bf b}_{2,LL}\cdot{\bf H}_{2,LL}+{\bf b}_{2,MN}\cdot{\bf H}_{2,MN})^2,
 H_{2,JJ}
 \rangle_\rho
 \\
 &=
 \langle
 ({\bf b}_{2,LL}\cdot{\bf H}_{2,LL})^2,H_{2,JJ}
 \rangle_\rho
 +
 \langle
 ({\bf b}_{2,MN}\cdot{\bf H}_{2,MN})^2,H_{2,JJ}
 \rangle_\rho
 \\
 &\quad
 +
 2
 \underbrace{
 \langle
 ({\bf b}_{2,LL}\cdot{\bf H}_{2,LL})({\bf b}_{2,MN}\cdot{\bf H}_{2,MN}),H_{2,JJ}
 \rangle_\rho
 }_{=0}
 \\
 &=
 \sum_{L=1}^n
 b_{2,LL}^2
 \langle
 H_{2,LL}^2,H_{2,JJ}
 \rangle_\rho
 +
 \underbrace{
 2\sum_{L_1<L_2}^n
 b_{2,L_1L_1}b_{2,L_2L_2}
 \langle
 H_{2,L_1L_1}H_{2,L_2L_2},H_{2,JJ}
 \rangle_\rho
 }_{=0}
 \\
 &\quad
 +
 \langle
 ({\bf b}_{2,MN}\cdot{\bf H}_{2,MN})^2,H_{2,JJ}
 \rangle_\rho
 \\
 &=
 b_{2,JJ}^2
 \langle
 H_{2,JJ}^2,H_{2,JJ}
 \rangle_\rho
 +
 \langle
 ({\bf b}_{2,MN}\cdot{\bf H}_{2,MN})^2,H_{2,JJ}
 \rangle_\rho.
 \end{align*}
 Lemma \ref{lemma_A.1} implies \eqref{equation_A.6}.
 In the same way,
 we see that
 \begin{align*}
 \langle&({\bf b}_2\cdot{\bf H}_2)^2,H_{2,IJ})_\rho
 \\
 &=
 \langle
 ({\bf b}_{2,LL}\cdot{\bf H}_{2,LL}+{\bf b}_{2,MN}\cdot{\bf H}_{2,MN})^2,H_{2,IJ}
 \rangle_\rho
 \\
 &=
 \underbrace{
 \langle({\bf b}_{2,LL}\cdot{\bf H}_{2,LL})^2,H_{2,IJ}
 \rangle_\rho
 }_{=0}
 +
 \langle
 ({\bf b}_{2,MN}\cdot{\bf H}_{2,MN})^2,H_{2,IJ}
 \rangle_\rho
 \\
 &\quad
 +
 2
 \langle
 ({\bf b}_{2,LL}\cdot{\bf H}_{2,LL})({\bf b}_{2,MN}\cdot{\bf H}_{2,MN}),H_{2,IJ}
 \rangle_\rho.
 \end{align*}
 The last term can be given by
 \begin{align*}
 \langle
 ({\bf b}_{2,LL}&\cdot{\bf H}_{2,LL})
 ({\bf b}_{2,MN}\cdot{\bf H}_{2,MN}),
 H_{2,IJ}
 \rangle_\rho
 \\
 &=
 b_{2,IJ}
 \langle
 ({\bf b}_{2,LL}\cdot{\bf H}_{2,LL})H_{2,IJ},H_{2,IJ}
 \rangle_\rho
 \\
 &=
 b_{2,IJ}
 \langle
 (b_{2,II}H_{2,II}+b_{2,JJ}H_{2,JJ})H_{2,IJ},H_{2,IJ}
 \rangle_\rho
 \\
 &=
 \sqrt{2}
 {\sf k}_0^n
 b_{2,IJ}
 (b_{2,II}+b_{2,II}).
 \end{align*}
 The proof is completed.
 \end{proof}

 Define $N(W)$ by
 \begin{align}
 \nonumber
 N(W)
 &=
 |\kappa+W|^{p-1}(\kappa+W)-\kappa^p-p\kappa^{p-1}W
 \\
 \nonumber
 &=
 \tfrac{p(p-1)}{2}
 \kappa^{p-2}
 W^2
 +
 (N-\tfrac{p(p-1)}{2}\kappa^{p-2}W^2)
 \\
 \label{equation_A.8}
 &=
 \tfrac{pW^2}{2\kappa}
 +
 (N-\tfrac{p}{2\kappa}W^2).
 \end{align}

 \begin{lem}
 \label{lemma_A.3}
 Let $p>1$ and $K_1>1$.
 There exists $C_1>1$ {\rm(}depends on $p,K_1${\rm)} such that
 \begin{align*}
 |N(W)|
 <
 C_1W^2
 \qquad
 \text{\rm for }
 |W|<K_1.
 \end{align*}
 \end{lem}
 \begin{proof}
 We easily verify this inequality from an elementary computation.
 \end{proof}

 \begin{lem}
 \label{lemma_A.4}
 Let $W(z)\in L_\rho^2(\R^n)$.
 We decompose it as
 $W(z)=b_0H_0+{\bf b}_1\cdot{\bf H}_1+{\bf b}_2\cdot{\bf H}_2+W^\bot(z)$.
 There exists $C>0$ independent of $W(z)$ and $r_1>1$ such that
 \begin{align}
 \label{equation_A.9}
 |
 \langle N(W&),H_{2,JJ}\rangle_\rho
 -
 c_p
 b_{2,JJ}^2
 -
 \tfrac{pH_0}{\kappa}
 b_0b_{2,JJ}
 |
 \\
 \nonumber
 &<
 C
 (
 |{\bf b}_1|^2
 +
 |{\bf b}_{2,IJ}|^2
 +
 |{\bf b}_2|\cdot\|W^\bot\|_\rho
 +
 \|W^\bot\|_{H_\rho^1}^2
 )
 \\
 \nonumber
 &\quad
 +
 |\langle\{N(W)-\tfrac{pW^2}{2\kappa}\}{\bf 1}_{|z|<r_1},H_{2,JJ}\rangle_\rho|
 \\
 \nonumber
 &\quad
 +
 |\langle N(W){\bf 1}_{|z|>r_1},H_{2,JJ}\rangle_\rho|
 +
 |\langle \tfrac{pW^2}{2\kappa}{\bf 1}_{|z|>r_1},H_{2,JJ}\rangle_\rho|
 \\
 \nonumber
 &
 \text{\rm for all }
 J,
 \end{align}
 \begin{align}
 \label{equation_A.10}
 |
 \langle N(W&),H_{2,IJ}\rangle_\rho
 -
 c_p
 (b_{2,II}+b_{2,JJ})
 b_{2,IJ}
 -
 \tfrac{pH_0}{\kappa}
 b_0b_{2,IJ}
 |
 \\
 \nonumber
 &<
 C
 (|{\bf b}_1|^2
 +
 |{\bf b}_{2,IJ}|^2
 +
 |{\bf b}_2|\cdot\|W^\bot\|_\rho
 +
 \|W^\bot\|_{H_\rho^1}^2
 )
 \\
 \nonumber
 &\quad
 +
 |\langle\{N(W)-\tfrac{pW^2}{2\kappa}\}{\bf 1}_{|z|<r_1},H_{2,IJ}\rangle_\rho|
 \\
 \nonumber
 &\quad
 +
 |\langle N(W){\bf 1}_{|z|>r_1},H_{2,IJ}\rangle_\rho|
 +
 |\langle \tfrac{pW^2}{2\kappa}{\bf 1}_{|z|>r_1},H_{2,IJ}\rangle_\rho|
 \\
 \nonumber
 &
 \text{\rm for all }
 I<J.
 \end{align}
 \end{lem}

 \begin{proof}
 A direct computation shows
 \begin{align*}
 \nonumber
 \tfrac{pW^2}{2\kappa}
 &
 =
 \tfrac{p}{2\kappa}
 (
 b_0
 H_0
 )^2
 +
 \tfrac{p}{2\kappa}
 (
 {\bf b}_1
 \cdot
 {\bf H}_1
 )^2
 +
 \tfrac{p}{2\kappa}
 (
 {\bf b}_2
 \cdot
 {\bf H}_2
 )^2
 +
 \tfrac{p}{2\kappa}
 (
 W^\bot
 )^2
 \\
 \nonumber
 &
 +
 \tfrac{p}{\kappa}
 (
 b_0
 H_0
 )
 (
 {\bf b}_1
 \cdot
 {\bf H}_1
 )
 +
 \tfrac{p}{\kappa}
 (
 b_0
 H_0
 )
 (
 {\bf b}_2
 \cdot
 {\bf H}_2
 )
 +
 \tfrac{p}{\kappa}
 (
 b_0
 H_0
 )
 W^\bot
 \\
 &
 +
 \tfrac{p}{\kappa}
 (
 {\bf b}_1
 \cdot
 {\bf H}_1
 )
 (
 {\bf b}_2
 \cdot
 {\bf H}_2
 )
 +
 \tfrac{p}{\kappa}
 (
 {\bf b}_1
 \cdot
 {\bf H}_1
 )
 W^\bot
 +
 \tfrac{p}{\kappa}
 (
 {\bf b}_2
 \cdot
 {\bf H}_2
 )
 W^\bot.
 \end{align*}
 Hence we get
 \begin{align}
 \label{equation_A.11}
 |
 \langle N(W),&\,H_{2,IJ}\rangle_\rho
 -
 \tfrac{p}{2\kappa}
 \langle({\bf b}_2\cdot{\bf H}_2)^2,H_{2,IJ}\rangle_\rho
 -
 \tfrac{pH_0}{\kappa}
 b_0b_{2,IJ}
 |
 \\
 \nonumber
 &\quad
 <
 C
 (|{\bf b}_1|^2
 +
 |{\bf b}_2|\cdot\|W^\bot\|_\rho
 +
 \|W^\bot\|_{H_\rho^1}^2
 )
 \\
 \nonumber
 &\qquad
 +
 |\langle \tfrac{pW^2}{2\kappa}{\bf 1}_{|z|>r_1},H_{2,IJ}\rangle_\rho|
 +
 |\langle (N-\tfrac{pW^2}{2\kappa}){\bf 1}_{|z|<r_1},H_{2,IJ}\rangle_\rho|
 \\
 \nonumber
 &\qquad
 +
 |\langle N{\bf 1}_{|z|>r_1},H_{2,IJ}\rangle_\rho|.
 \end{align}
 Combining \eqref{equation_A.11}, Lemma \ref{lemma_A.2} and
 the definition of $c_p$ (see \eqref{eq_2.7}),
 we obtain \eqref{equation_A.9} - \eqref{equation_A.10}.
 \end{proof}

 \section{Estimates for rescaled solutions}
 In this section,
 we investigate the asymptotic behavior of solutions to
 \begin{align*}
 \begin{cases}
 \varphi_t=\Delta \varphi+|\varphi|^{p-1}\varphi
 \qquad
 \text{\rm for }
 (x,t)\in\R^n\times(0,T),
 \\
 \varphi(x,0)
 =
 \varphi_0\in C(\R^n)\cap L^\infty(\R^n).
 \end{cases}
 \end{align*}
 We assume
 \begin{enumerate}[({a}1)]
 \setlength{\leftskip}{5mm}
 \item
 $\varphi(x,t)$ blows up in a finite time $t=T$,
 \item
 $\dis\sup_{t\in(0,T)}\|\varphi(x,t)\|_{L^\infty(|x|>r)}<\infty$ \
 for any fixed $r>0$,
 \item
 there exists $M>0$ such that
 $\varphi(x,t)>-M$ \
 for $x\in\R^n$ and $t\in(0,T)$,
 \item
 $\dis\sup_{t\in(0,T)}(T-t)^\frac{1}{p-1}\|\varphi(x,t)\|_{L_x^\infty(\R^n)}<\infty$,
 \item
 $\dis\lim_{\tau\to\infty}\|\Phi(z,\tau)-\kappa\|_{L_\rho^2(\R^n)}=0$
 \quad
 {\rm(}$\kappa=(p-1)^{-\frac{1}{p-1}}${\rm)},
 \item
 $\dis
 \lim_{\tau\to\infty}
 \tau
 \|
 \Phi(z,\tau)
 -
 \kappa
 +
 \tfrac{1}{c_p\tau}
 \sum_{J=1}^n
 H_{2,JJ}(z)
 \|_{H_\rho^1(\R^n)}
 =0$.
 \end{enumerate}
 The assumptions (a1) - (a6) are the same as in Theorem \ref{thm_2}.
 \begin{pro}
 \label{proposition_B.1}
 Let $\varphi(x,t)$ be a solution of
 $\varphi_t=\Delta_x\varphi+|\varphi|^{p-1}\varphi$
 satisfying {\rm(a1)} - {\rm(a6)}.
 Let $T$ be the blowup time of $\varphi(x,t)$,
 and put
 \begin{align*}
 \Phi(z,\tau)
 =
 e^{-\frac{\tau}{p-1}}
 \varphi(e^{-\frac{\tau}{2}}z,T-e^{-\tau}).
 \end{align*}
 Then
 there exists $\sigma_1\in(-\log T,\infty)$
 such that
 \begin{enumerate}[\rm(s1) ]
 \setlength{\leftskip}{5mm}
 \item 
 $\dis\lim_{\tau\to\infty}
 \tau\|\Phi(z,\tau)-\kappa+\tfrac{1}{c_p\tau}\sum_{J=1}^nH_{2,JJ}
 \|_{H_\rho^1(\R^n)}=0$,

 \item 
 $-e^{-\frac{7\tau}{8(p-1)}}<\Phi(z,\tau)
 <
 \kappa+\tfrac{2n\kappa}{p\tau}
 $
 \quad
 \text{\rm for }
 $(z,\tau)\in\R^n\times(\sigma_1,\infty)$,

 \item 
 there exists $c_1\in(0,1)$ such that
 \begin{align*}
 \sup_{|z|>\sqrt{3n}}
 \Phi(z,\tau)
 &<
 \kappa
 -
 \tfrac{c_1}{\tau}
 \qquad
 \text{\rm for }
 \tau\in(\sigma_1,\infty),
 \end{align*}

 \item 
 there exists $c_2>1$ such that
 \[
 \|\Phi^\bot(z,\tau)\|_{L_\rho^2(\R^n)}<c_2\tau^{-2}
 \qquad
 \text{\rm for }
 \tau\in(\sigma_1,\infty),
 \]

 \item 
 $\dis\sup_{|z|<\tau^\frac{1}{16}}
 |\Phi^\bot(z,\tau)|<\tau^{-\frac{3}{2}}$
 \quad
 \text{\rm for }
 $\tau\in(\sigma_1,\infty)$.
 \end{enumerate}
 \end{pro}

 \begin{rem}
 Since we could not find references that explicitly provide the proof of
 {\rm(s2) - (s5)} in Proposition {\rm\ref{proposition_B.1}},
 we include the proof below.
 The assertion {\rm(s3)} plays a crucial role in the proofs of
 Proposition {\rm\ref{proposition_4.3}} - Proposition {\rm\ref{proposition_4.4}}.
 \end{rem}

 \begin{proof}
 To obtain (s4),
 we repeat the argument \eqref{eq_4.43} - \eqref{eq_4.51}
 in the proof of Proposition \ref{proposition_4.3}.
 Since $\Phi(z,\tau)\to\kappa$ in $L_\rho^2(\R^n)$ (see (a5)),
 for any $\epsilon\in(0,1)$ and $R_1>1$
 there exists $\sigma_1=\sigma_1(\epsilon,R_1)\gg1$ such  hat
 \begin{align*}
 \sup_{\tau>\sigma_1}
 \sup_{|z|<R_1}
 |\Phi(z,\tau)-\kappa|
 <
 \epsilon.
 \end{align*}
 We use this estimate to derive \eqref{eq_4.47}
 instead of \eqref{eq_4.30}.
 Then
 we can derive (s4) (see \eqref{eq_4.51}).
 Furthermore
 we put $\Psi(z,\tau)=\Phi(z,\tau)-\kappa$.
 We easily verify that
 $\Psi(z,\tau)$ solves
 \begin{align}
 \label{eq_BB.1}
 \Psi_\tau
 =
 A_z\Psi
 +
 \Psi
 +
 N(\Psi),
 \end{align}
 where
 $N(\Psi)=|\kappa+\Psi|^{p-1}(\kappa+\Psi)-\kappa^p-p\kappa^{p-1}\Psi$.
 For simplicity,
 we put
 \begin{itemize}
 \item
 $b_0=\langle \Psi(z,\tau),H_0\rangle_\rho$,
 \item
 $b_{1,J}=\langle \Psi(z,\tau),H_{1,J}\rangle_\rho$
 and
 \item
 $b_{2,IJ}=\langle \Psi(z,\tau),H_{2,IJ}\rangle_\rho$.
 \end{itemize}
 Multiplying \eqref{eq_BB.1} by $H_0$ and $H_{1,J}$,
 we obtain ODEs for $b_0$ and $b_1$, respectively.
 Integrating these ODEs,
 we obtain
 \begin{align}
 \label{eq_BB.2}
 |b_0(\tau)|+|{\bf b}_1(\tau)|
 <
 \tfrac{C}{\tau^2}.
 \end{align}
 Combining (s4) and \eqref{eq_BB.2},
 we have
 \begin{align}
 \label{eq_BB.3}
 \|\Psi(z,\tau)-{\bf b}_2(\tau)\cdot{\bf H}_2\|_\rho
 <
 \tfrac{C}{\tau^2}.
 \end{align}
 We now prove (s5).
 We introduce a profile function.
 \begin{align}
 \label{eq_BB.4}
 g(z,\tau)
 =
 \kappa
 \left(
 1
 -
 \tfrac{p-1}{\kappa}
 {\bf b}_2(\tau)\cdot{\bf H}_2(z)
 \right)^{-\frac{1}{p-1}}.
 \end{align}
 From \eqref{eq_4.63},
 we note that
 \begin{align*}
 \nonumber
 {\bf b}_2
 \cdot
 {\bf H}_2
 &=
 -
 \tfrac{1}{c_p\tau}
 \sum_{J=1}^n
 H_{2,JJ}
 +
 \underbrace{
 \sum_{J=1}^n
 (
 b_{2,JJ}
 +
 \tfrac{1}{c_p\tau}
 )
 H_{2,JJ}
 +
 \sum_{I<J}
 b_{2,IJ}
 H_{2,IJ}
 }_{=o(\frac{1}{\tau})(1+|z|^2)}
 \\
 &=
 -
 \tfrac{\kappa}{4p\tau}
 (|z|^2-2n)
 +
 o(\tfrac{1}{\tau})(1+|z|^2)
 \qquad
 \text{for }
 z\in\R^n.
 \end{align*}
 Hence
 it holds that
 \begin{align}
 \label{eq_BB.5}
 \tfrac{\kappa}{4p\tau}
 (-\tfrac{9}{8}|z|^2+\tfrac{15n}{8})
 <
 {\bf b}_2
 \cdot
 {\bf H}_2
 <
 \tfrac{\kappa}{4p\tau}
 (-\tfrac{7}{8}|z|^2+\tfrac{17n}{8})
 \end{align}
 for
 $z\in\R^n$ and $\tau\gg1$.
 From \eqref{eq_BB.5},
 $g(z,\tau)$ is well defined for $z\in\R^n$.
 The function $g(z,\tau)$ satisfies
 \begin{align}
 \label{eq_BB.6}
 g_\tau
 &=
 A_zg
 -
 \tfrac{g}{p-1}
 +
 g^p
 +
 \underbrace{
 (\tfrac{d{\bf b}_2}{d\tau}\cdot{\bf H}_2)
 \left(
 1
 -
 \tfrac{p-1}{\kappa}
 {\bf b}_2\cdot{\bf H}_2
 \right)^{-\frac{p}{p-1}}
 }_{F_1(z,\tau)}
 \\
 \nonumber
 &\quad
 +
 \underbrace{
 \tfrac{-p}{\kappa}
 |\nabla_z({\bf b}_2\cdot{\bf H}_2)|^2
 \left(
 1
 -
 \tfrac{p-1}{\kappa}
 {\bf b}_2\cdot{\bf H}_2
 \right)^{-\frac{p}{p-1}-1}
 }_{=F_2(z,\tau)}.
 \end{align}
 We put
 \begin{align*}
 \Upsilon(z,\tau)
 =
 \Phi(z,\tau)
 -
 g(z,\tau),
 \end{align*}
 then
 \begin{align}
 \label{eq_BB.7}
 \Upsilon_\tau
 &=
 A_z\Upsilon
 -
 \tfrac{1}{p-1}
 \Upsilon
 +
 V{\bf 1}_{|z|<\tau}\Upsilon
 +
 \underbrace{
 V{\bf 1}_{|z|>\tau}\Upsilon
 }_{=F_3}
 -
 F_1-F_2,
 \end{align}
 here we write
 \begin{align*}
 V=
 \tfrac{|\Phi|^{p-1}\Phi-g^p}{\Phi-g}
 =
 p
 \int_0^1
 |\theta\Phi+(1-\theta)g|^{p-1}
 d\theta.
 \end{align*}
 From the definition of $g(z,\tau)$ (see \eqref{eq_BB.4}) and \eqref{eq_BB.5},
 we note that
 \begin{align}
 \label{eq_BB.8}
 0
 < 
 g(z,\tau)
 <
 \kappa
 +
 C\tau^{-1}.
 \end{align}
 Set
 $\Omega_+
 =
 \{(z,\tau)\in\R^n\times(-\log T,\infty);\Phi(z,\tau)>0\}$
 and
 $\Omega_-
 =
 \{(z,\tau)\in\R^n\times(-\log T,\infty);\Phi(z,\tau)<0\}$.
 From Lemma 3.3 in \cite{Harada} p. 4231,
 there exist $r,\gamma,c>0$ such that
 \begin{align}
 \label{eq_BB.9}
 \Phi(z,\tau)<\kappa+ce^{-\gamma\tau}
 \qquad
 \text{for }
 r<|z|<e^\frac{\tau}{4}.
 \end{align}
 Combining \eqref{eq_BB.9} and (a6),
 we get
 \begin{align}
 \label{eq_BB.10}
 0<\Phi(z,\tau)<\kappa+\tfrac{c}{\tau}
 \qquad
 \text{for }
 (z,\tau)\in\Omega_+\cap\{|z|<\tau\}.
 \end{align}
 Therefore
 from \eqref{eq_BB.8} and \eqref{eq_BB.10},
 it follows that
 \begin{align}
 \nonumber
 V
 &=
 p
 \int_0^1
 (\theta\Phi+(1-\theta)g)^{p-1}
 d\theta
 \\
 \nonumber
 &<
 p
 (\kappa+\tfrac{c}{\tau})^{p-1}
 <
 p\kappa^{p-1}
 +
 \tfrac{D}{\tau}
 \\
 \label{eq_BB.11}
 &<
 \tfrac{p}{p-1}
 +
 \tfrac{D}{\tau}
 \qquad
 \text{for }
 (z,\tau)\in\Omega_+\cap\{|z|<\tau\}.
 \end{align}
 From (a3),
 we note that
 \begin{align*}
 -Me^{-\frac{\tau}{p-1}}<\Phi(z,\tau)<0
 \qquad
 \text{for }
 (z,\tau)\in\Omega_-.
 \end{align*}
 Hence
 it holds that
 \begin{align}
 \nonumber
 V
 &<
 p
 \max\{
 |\Phi|^{p-1},
 g^{p-1}
 \}
 <
 p
 \kappa^{p-1}
 +
 \tfrac{D}{\tau}
 \\
 \label{eq_BB.12}
 &<
 \tfrac{p}{p-1}
 +
 \tfrac{D}{\tau}
 \qquad
 \text{for }
 (z,\tau)\in\Omega_-.
 \end{align}
 We recall that
 there exist $p_0(z)\in V_0$, $p_1(z)\in V_2$
 (see Section \ref{sec_2} for the definition of $V_k$)
 such that
 \begin{align}
 \label{eq_BB.13}
 |z|^2
 <
 p_0(z)+p_1(z)
 \qquad
 \text{for }
 z\in\R^n.
 \end{align}
 Hence
 using \eqref{eq_BB.5},
 we can obtain the bound of $F_1(z,\tau),F_2(z,\tau)$ (see \eqref{eq_BB.6}).
 \begin{align}
 \nonumber
 |F_1(z,\tau)|
 &<
 |\tfrac{d{\bf b}_2}{d\tau}\cdot{\bf H}_2|
 \underbrace{
 (
 1
 -
 \tfrac{p-1}{\kappa}
 {\bf b}_2\cdot{\bf H}_2
 )^{-\frac{p}{p-1}}
 }_{<C}
 \\
 \label{eq_BB.14}
 &<
 C
 |\tfrac{d{\bf b}_2}{d\tau}|
 |z|^2
 <
 C
 |\tfrac{d{\bf b}_2}{d\tau}|
 (p_0(z)+p_1(z)),
 \\
 \nonumber
 |F_2(z,\tau)|
 &<
 2
 |{\bf b}_2|^2
 |\nabla_z{\bf H}_2|^2
 <
 C|{\bf b}_2|^2|z|^2
 \\
 \label{eq_BB.15}
 &<
 C
 |{\bf b}_2|^2
 (p_0(z)+p_1(z)).
 \end{align}
 Multiplying
 \eqref{eq_BB.1}
 by $H_{2,IJ}(z)$,
 we get
 \begin{align*}
 \tfrac{db_{2,IJ}}{d\tau}
 =
 \langle N,H_{2,IJ}\rangle_\rho.
 \end{align*}
 From (a6),
 we easily see that
 \begin{align*}
 |\tfrac{db_{2,IJ}}{d\tau}|
 &<
 C
 \|\Psi\|_{H_\rho^1(\R^n)}^2
 =
 C
 \|(\Phi-\kappa)\|_{H_\rho^1(\R^n)}^2
 \\
 &<
 C\tau^{-2}
 \qquad
 \text{for all }
 I\leq J.
 \end{align*}
 Therefore
 \eqref{eq_BB.14} - \eqref{eq_BB.15} imply
 \begin{align}
 \label{eq_BB.16}
 |F_1(z,\tau)|
 +
 |F_2(z,\tau)|
 <
 \tfrac{C_1}{\tau^2}
 (p_0(z)+p_1(z))
 \end{align}
 for
 $(z,\tau)\in\R^n\times(-\log T,\infty)$.
 Furthermore
 since $\Phi(z,\tau)$ and $g(z,\tau)$ is uniformly bounded on
 $(z,\tau)\in\R^n\times(-\log\tau,\infty)$ (see (a4) and \eqref{eq_BB.8}),
 we can estimate $F_3(z,\tau)$ as
 \begin{align*}
 |F_3(z,\tau)|
 =
 |V|{\bf 1}_{|z|>\tau}|\Upsilon|
 =
 ||\Phi|^{p-1}\Phi-g^p|
 {\bf 1}_{|z|>\tau}
 <
 C_2
 {\bf 1}_{|z|>\tau}.
 \end{align*}
 We now construct a comparison function of \eqref{eq_BB.7}.
 \begin{align}
 \label{eq_BB.17}
 \begin{cases}
 \dis
 \bar\Upsilon_\tau
 =
 A_z
 \bar\Upsilon
 +
 (1+\tfrac{D}{\tau})
 \bar\Upsilon
 +
 \tfrac{C_1}{\tau^2}
 (p_0(z)+p_2(z))
 +
 C_2
 {\bf 1}_{|z|>\tau}
 &
 \text{for }
 \tau>\tau_1,
 \\
 \bar\Upsilon|_{\tau=\tau_1}
 =
 |\Upsilon(\tau_1)|.
 \end{cases}
 \end{align}
 From \eqref{eq_BB.3} and \eqref{eq_BB.4},
 we can verify that
 \begin{align}
 \label{eq_BB.18}
 \|\Upsilon(z,\tau)\|_\rho
 <
 C
 \tau^{-2}
 \qquad
 \text{for }
 \tau\in(-\log T,\infty).
 \end{align}
 We define $\tau_1,\tau_2$ by
 \begin{align*}
 \tau_1(\tau)
 &=
 \tau-\tfrac{1}{8}\log\tau,
 \\
 \tau_2(\tau)
 &=
 \tau-\log\tau.
 \end{align*}
 We use $\tau_1$ for the proof of (s5),
 and $\tau_2$ for the proof of (s3).
 We easily see that
 \begin{align}
 \label{eq_BB.19}
 1<\tfrac{\tau}{\tau_1}<\tfrac{\tau}{\tau_2}<2.
 \end{align}
 The Duhamel formula gives
 \begin{align}
 \label{eq_BB.20}
 \bar\Upsilon(\tau)
 &=
 (\tfrac{\tau}{\tau_i})^D
 e^{\tau-\tau_i}
 e^{A_z(\tau-\tau_i)}
 \bar\Upsilon(\tau_i)
 \\
 \nonumber
 &\quad
 +
 \int_{\tau_i}^\tau
 (\tfrac{\tau}{s})^D
 e^{\tau-s}
 e^{A_z(\tau-s)}
 [\tfrac{C_1}{s^2}(p_0+p_1)]
 ds
 \
 \\
 \nonumber
 &\quad
 +
 \int_{\tau_i}^\tau
 (\tfrac{\tau}{s})^D
 e^{\tau-s}
 e^{A_z(\tau-s)}
 [C_2{\bf 1}_{|z|>s}]
 ds
 \qquad
 (i=1,2).
 \end{align}
 From Lemma \ref{lem_3.2} and \eqref{eq_BB.18},
 it holds that
 \begin{align*}
 \sup_{|z|<2e^\frac{\tau-\tau_i}{2}}
 |e^{A_z(\tau-\tau_i)}\bar\Upsilon(\tau_i)|
 &<
 C\|\bar\Upsilon(\tau_i)\|_\rho
 <
 C\tau_i^{-2}
 \qquad
 (i=1,2).
 \end{align*}
 Hence
 by the choice of $\tau_1,\tau_2$ and \eqref{eq_BB.19},
 we obtain
 \begin{align}
 \label{eq_BB.21}
 \sup_{|z|<2\tau^\frac{1}{16}}
 (\tfrac{\tau}{\tau_1})^D
 e^{\tau-\tau_1}
 |e^{A_z(\tau-\tau_1)}\bar\Upsilon(\tau_1)|
 &<
 e^{\frac{1}{8}\log\tau}
 \tfrac{C}{\tau^2}
 <
 C\tau^{-\frac{15}{8}},
 \\
 \label{eq_BB.22}
 \sup_{|z|<2\tau^\frac{1}{2}}
 (\tfrac{\tau}{\tau_2})^D
 e^{\tau-\tau_2}
 |e^{A_z(\tau-\tau_2)}\bar\Upsilon(\tau_2)|
 &<
 \tfrac{C}{\tau}.
 \end{align}
 We next compute the second term on the right-hand side of \eqref{eq_BB.20}.
 \begin{align*}
 &\int_{\tau_i}^\tau
 (\tfrac{\tau}{s})^D
 e^{\tau-s}
 e^{A_z(\tau-s)}
 [\tfrac{p_0+p_1}{s^2}
 ]
 ds
 \\
 &=
 p_0(z)
 \int_{\tau_i}^\tau
 (\tfrac{\tau}{s})^D
 e^{\tau-s}
 \tfrac{ds}{s^2}
 +
 p_1(z)
 \int_{\tau_i}^\tau
 (\tfrac{\tau}{s})^D
 \tfrac{ds}{s^2}.
 \\
 &<
 C|p_0(z)|
 e^{\tau-\tau_i}
 \tfrac{ds}{\tau_i^2}
 +
 C|p_1(z)|
 \tfrac{\tau-\tau_i}{\tau_i^2}.
 \end{align*}
 Hence
 it follows that
 \begin{align}
 \nonumber
 \int_{\tau_1}^\tau
 (\tfrac{\tau}{s})^D
 e^{\tau-s}
 e^{A_z(\tau-s)}
 [\tfrac{p_0+p_1}{s^2}]
 ds
 &<
 \tfrac{C\tau^\frac{1}{8}}{\tau^2}
 |p_0(z)|
 +
 \tfrac{C\log\tau}{\tau^2}
 |p_1(z)|
 \\
 \label{eq_BB.23}
 &<
 \tfrac{C}{\tau^\frac{15}{8}}
 |p_0(z)|
 +
 \tfrac{C\log\tau}{\tau^2}
 |p_1(z)|,
 \\
 \label{eq_BB.24}
 \int_{\tau_2}^\tau
 (\tfrac{\tau}{s})^D
 e^{\tau-s}
 e^{A_z(\tau-s)}
 [\tfrac{p_0+p_1}{s^2}]
 ds
 &<
 \tfrac{C}{\tau}
 |p_0(z)|
 +
 \tfrac{C\log\tau}{\tau^2}
 |p_1(z)|.
 \end{align}
 We apply Lemma \ref{lem_3.3} to estimate the last term in \eqref{eq_BB.20}.
 \begin{align}
 \label{eq_BB.25}
 \sup_{|z|<2\sqrt{\tau}}
 \int_{\tau_i}^\tau
 (\tfrac{\tau}{s})^D
 e^{\tau-s}
 e^{A_z(\tau-s)}
 [{\bf 1}_{|z|>s}]
 ds
 <
 \tfrac{C}{\tau^2}.
 \end{align}
 Substituting \eqref{eq_BB.21} - \eqref{eq_BB.25} to \eqref{eq_BB.20},
 we obtain
 \begin{align}
 \nonumber
 \sup_{|z|<2\tau^\frac{1}{16}}
 \bar\Upsilon(\tau)
 &<
 \tfrac{C}{\tau^\frac{15}{8}}
 +
 \sup_{|z|<2\tau^\frac{1}{16}}
 \{
 \tfrac{C}{\tau^\frac{15}{8}}
 |p_0(z)|
 +
 \tfrac{C\log\tau}{\tau^2}
 |p_1(z)|
 \}
 +
 \tfrac{C}{\tau^2}
 \\
 \label{eq_BB.26}
 &<
 \tfrac{C\log\tau}{\tau^\frac{15}{8}},
 \\
 \nonumber
 \sup_{|z|<2\tau^\frac{1}{2}}
 \bar\Upsilon(\tau)
 &<
 \tfrac{C}{\tau}
 +
 \sup_{|z|<2\tau^\frac{1}{2}}
 \{
 \tfrac{C}{\tau}
 |p_0(z)|
 +
 \tfrac{C\log\tau}{\tau^2}
 |p_1(z)|
 \}
 +
 \tfrac{C}{\tau^2}
 \\
 \label{eq_BB.27}
 &
 <
 \tfrac{C\log\tau}{\tau}.
 \end{align}
 Since
 $\Upsilon(z,\tau)<\bar\Upsilon(z,\tau)$
 by a comparison argument,
 we conclude
 \begin{align}
 \label{eq_BB.28}
 \sup_{|z|<2\tau^\frac{1}{16}}
 |\Upsilon(z,\tau)|
 &<
 \tfrac{C\log\tau}{\tau^\frac{15}{8}},
 \\
 \label{eq_BB.29}
 \sup_{|z|<2\tau^\frac{1}{2}}
 |\Upsilon(z,\tau)|
 &<
 \tfrac{C\log\tau}{\tau}.
 \end{align}
 Since
 $\Phi^\bot$ can be written as
 \begin{align*}
 &|\Phi^\bot|
 =
 |(\Phi-\kappa)^\bot|
 \\
 &=
 |\Phi-\kappa-b_0H_0-{\bf b}_1\cdot{\bf H}_1-{\bf b}_2\cdot{\bf H}_2|
 \\
 &<
 \underbrace{
 |\Phi-g|
 }_{=|\Upsilon|}
 +
 |g-\kappa-{\bf b}_2\cdot{\bf H}_2|
 +
 \underbrace{
 |b_0H_0-{\bf b}_1\cdot{\bf H}_1|
 }_{<C\tau^{-2}(1+|z|)},
 \end{align*}
 we get from \eqref{eq_BB.29} that
 \begin{align}
 \label{eq_BB.30}
 \sup_{|z|<2\tau^\frac{1}{16}}
 |\Phi^\bot|
 &<
 \tfrac{C\log\tau}{\tau^\frac{15}{8}}
 +
 \sup_{|z|<2\tau^\frac{1}{16}}
 |g-\kappa-{\bf b}_2\cdot{\bf H}_2|
 +
 \tfrac{C}{\tau^\frac{31}{16}}.
 \end{align}
 From the definition of $g(z,\tau)$ (see \eqref{eq_BB.4}),
 we note that
 \begin{align}
 \nonumber
 |g-\kappa-{\bf b}_2\cdot{\bf H}_2|
 &=
 \kappa
 |
 (1-\tfrac{p-1}{\kappa}{\bf b}_2\cdot{\bf H}_2)^{-\frac{1}{p-1}}
 -1-\tfrac{1}{\kappa}{\bf b}_2\cdot{\bf H}_2
 |
 \\
 \label{eq_BB.31}
 &<
 C({\bf b}_2\cdot{\bf H}_2)^2
 <
 \tfrac{C}{\tau^\frac{7}{4}}
 \qquad
 \text{for }
 |z|<2\tau^\frac{1}{16}.
 \end{align}
 Therefore
 \eqref{eq_BB.30} - \eqref{eq_BB.31} prove (s5).
 We turn to the proof of (s3),
 which is a main part of this proposition.
 We begin with the estimates for  $\sqrt{3n}<|z|<2\sqrt\tau$.
 From \eqref{eq_BB.28}  - \eqref{eq_BB.29},
 we see that 
 \begin{align}
 \label{eq_BB.32}
 \Phi
 =
 \underbrace{
 \Phi-g
 }_{=\Upsilon}
 +
 g
 <
 \begin{cases}
 g
 +
 \tfrac{C}{\tau^\frac{7}{4}}
 & \text{for }
 |z|<2\tau^\frac{1}{16},
 \\
 g
 +
 \tfrac{C\log\tau}{\tau}
 & \text{for }
 |z|<2\sqrt\tau.
 \end{cases}
 \end{align}
 From \eqref{eq_BB.5},
 we note that ${\bf b}_2\cdot{\bf H}_2<0$ for $|z|>\sqrt{3n}$ and $\tau\gg1$.
 Hence
 we observe
 \begin{align}
 \nonumber
 g(z,\tau)
 &=
 \kappa(1-\tfrac{p-1}{\kappa}{\bf b}_2\cdot{\bf H}_2)^{-\frac{1}{p-1}}
 \\
 \label{eq_BB.33}
 &<
 \kappa
 +
 {\bf b}_2\cdot{\bf H}_2
 +
 \tfrac{p}{2\kappa}
 ({\bf b}_2\cdot{\bf H}_2)^2
 \qquad
 \text{for }
 |z|>\sqrt{3n}.
 \end{align}
 Therefore
 from \eqref{eq_BB.32} - \eqref{eq_BB.33} and \eqref{eq_BB.5},
 we obtain
 \begin{align}
 \nonumber
 \Phi
 &<
 \kappa
 +
 {\bf b}_2\cdot{\bf H}_2
 +
 \tfrac{p}{2\kappa}
 ({\bf b}_2\cdot{\bf H}_2)^2
 +
 \tfrac{C}{\tau^\frac{7}{4}}
 \\
 \nonumber
 &<
 \kappa
 +
 \tfrac{1}{2}
 {\bf b}_2\cdot{\bf H}_2
 +
 \tfrac{C}{\tau^\frac{7}{4}}
 <
 \kappa
 +
 \tfrac{1}{4}
 {\bf b}_2\cdot{\bf H}_2
 \\
 \nonumber
 &<
 \kappa
 +
 \tfrac{\kappa}{16p\tau}
 (-\tfrac{7}{8}|z|^2+\tfrac{17n}{8})
 \\
 \label{eq_BB.34}
 &<
 \kappa
 -
 \tfrac{\kappa}{16p\tau}
 \tfrac{4n}{8}
 \qquad
 \text{for }
 \sqrt{3n}<|z|<2\tau^\frac{1}{16}.
 \end{align}
 From \eqref{eq_BB.5},
 we can choose $\nu\in(0,1)$ such that
 $\frac{p}{2\kappa}|{\bf b_2}\cdot{\bf H}_2|<\tfrac{1}{4}$
 for $|z|<\sqrt{\nu\tau}$.
 Then 
 we have
 \begin{align}
 \nonumber
 \Phi
 &<
 \kappa
 +
 {\bf b}_2\cdot{\bf H}_2
 +
 \tfrac{p}{2\kappa}
 ({\bf b}_2\cdot{\bf H}_2)^2
 +
 \tfrac{C\log\tau}{\tau}
 \\
 \nonumber
 &<
 \kappa
 +
 \tfrac{3}{4}
 {\bf b}_2\cdot{\bf H}_2
 +
 \tfrac{C\log\tau}{\tau}
 \\
 \nonumber
 &<
 \kappa
 +
 \tfrac{3}{4}
 \tfrac{\kappa}{4p\tau}
 (-\tfrac{7}{8}|z|^2+\tfrac{17n}{8})
 +
 \tfrac{C\log\tau}{\tau}
 \\
 \nonumber
 &<
 \kappa
 +
 \tfrac{3}{4}
 \tfrac{\kappa}{4p\tau}
 \underbrace{
 (-\tfrac{7}{8}\tau^\frac{1}{8}+\tfrac{17n}{8}+C\log\tau)
 }_{<-\tfrac{6}{8}\tau^\frac{1}{8}}
 \\
 \label{eq_BB.35}
 &<
 \kappa
 -
 \tfrac{9}{16}
 \tfrac{\kappa}{4p}
 \tau^{-\frac{7}{8}}
 \qquad
 \text{for }
 \tau^\frac{1}{16}<|z|<\sqrt{\nu\tau}.
 \end{align}
 Furthermore
 we note from \eqref{eq_BB.5} that
 \begin{align*}
 \nonumber
 g(z,\tau)
 &=
 \kappa(1-\tfrac{p-1}{\kappa}{\bf b}_2\cdot{\bf H}_2)^{-\frac{1}{p-1}}
 \\
 \nonumber
 &<
 \kappa
 \{
 1+\tfrac{\kappa}{4p\tau}(\tfrac{7}{8}|z|^2-\tfrac{17}{8})
 \}^{-\frac{1}{p-1}}
 \\
 \nonumber
 &<
 \kappa
 \{
 1+\tfrac{\kappa}{4p}(\tfrac{7\nu}{8}-\tfrac{17}{8\tau})
 \}^{-\frac{1}{p-1}}
 \}^{-\frac{1}{p-1}}
 \\
 &<
 \kappa
 \{
 1+\tfrac{\kappa}{4p}\tfrac{3\nu}{4}
 \}^{-\frac{1}{p-1}}
 \qquad
 \text{for }
 |z|>\sqrt{\nu\tau}.
 \end{align*}
 Hence
 \eqref{eq_BB.32} implies
 \begin{align}
 \nonumber
 \Phi
 &<
 g
 +
 \tfrac{C\log\tau}{\tau}
 \\
 \nonumber
 &<
 \kappa
 \{
 1+\tfrac{3\nu\kappa}{16p}
 \}^{-\frac{1}{p-1}}
 +
 \tfrac{C\log\tau}{\tau}
 \\
 \label{eq_BB.36}
 &<
 \kappa
 \{
 1+\tfrac{3\nu\kappa}{32p}
 \}^{-\frac{1}{p-1}}
 \qquad
 \text{for }
 \sqrt{\nu\tau}<|z|<2\sqrt\tau.
 \end{align}
 We next provide estimates for $|z|>2\sqrt\tau$.
 It is proved in \cite{Velazquez_higher} p. 1589 - 1590 that
 a rescaled function
 \begin{align*}
 \Theta(\omega,\zeta,s)
 =
 (T-s)^{\frac{1}{p-1}}
 \varphi(\sqrt{(T-s)|\log(T-s)|}\omega,s+\zeta(T-s))
 \end{align*}
 behaves like
 \begin{align}
 \label{eq_BB.37}
 \lim_{s\to T}
 \sup_{\omega\in S^{n-1}}
 \sup_{\zeta\in(0,1)}
 |
 \Theta(\omega,\zeta,s)
 -
 \kappa
 (
 1-\zeta
 +
 \tfrac{p-1}{4p}
 )^{-\frac{1}{p-1}}
 |
 =
 0.
 \end{align}
 We fix $\eta\in(0,1)$ such that
 \begin{align*}
 \kappa
 (
 1-\zeta
 +
 \tfrac{p-1}{4p}
 )^{-\frac{1}{p-1}}
 +
 \eta
 <
 \kappa
 (
 1-\zeta
 +
 \tfrac{p-1}{8p}
 )^{-\frac{1}{p-1}}
 \qquad
 \text{for }
 \zeta\in(0,1).
 \end{align*}
 From \eqref{eq_BB.37},
 there exists  $s_1\in(0,T)$ such that
 \begin{align*}
 \sup_{\omega\in S^{n-1}}
 &\Theta(\omega,\zeta,s)
 <
 \kappa
 (
 1-\zeta+\tfrac{p-1}{4p}
 )^{-\frac{1}{p-1}}
 +
 \eta
 \\
 &<
 \kappa
 (
 1-\zeta+\tfrac{p-1}{8p}
 )^{-\frac{1}{p-1}}
 \qquad
 \text{for }
 \zeta\in(0,1)
 \text{ and }
 s\in(s_1,T).
 \end{align*}
 We rewrite this inequality in the original variable $(x,t)$.
 Put $t=s+\zeta(T-s)$.
 If $s\in(s_1,T)$,
 then it holds that
 \begin{align}
 \label{eq_BB.38}
 \sup_{\omega\in S^{n-1}}
 &\varphi(\sqrt{(T-s)|\log(T-s)|}\omega,t)
 \\
 \nonumber
 &\quad
 <
 \kappa
 (T-t)^{-\frac{1}{p-1}}
 (
 1+\tfrac{p-1}{8p}\tfrac{T-s}{T-t}
 )^{-\frac{1}{p-1}}
 \qquad \text{for }
 t\in(s,T).
 \end{align}
 We apply a comparison argument in the region $U(s)$ defined by
 \begin{align*}
 U(s)
 =
 \{(x,t)&\in\R^n\times(s,T);
 \sqrt{(T-t)|\log(T-t)|}
 \\
 \nonumber
 &<|x|<\sqrt{(T-s)|\log(T-s)|}
 \}.
 \end{align*}
 We now take $\zeta=0$ in \eqref{eq_BB.38}.
 Since $t|_{\zeta=0}=\{s+\zeta(T-s)\}|_{\zeta=0}=s$,
 we get
 \begin{align}
 \label{eq_BB.39}
 \sup_{\omega\in S^{n-1}}
 \varphi(&\sqrt{(T-t)|\log(T-t)|}\omega,t)
 \\
 \nonumber
 &\qquad
 <
 \kappa
 (T-t)^{-\frac{1}{p-1}}
 (
 1+\tfrac{p-1}{8p}
 )^{-\frac{1}{p-1}}
 \qquad
 \text{for }
 t\in(s_1,T).
 \end{align}
 It is easily checked that
 $\varphi_1(x,t)=\kappa(T-t)^{-\frac{1}{p-1}}(1+\tfrac{(p-1)}{16p})^{-\frac{1}{p-1}}$
 satisfies
 $\pa_{\tau}\varphi_1>\Delta\varphi_1+\varphi_1^p$.
 Hence
 from \eqref{eq_BB.38} - \eqref{eq_BB.39},
 a comparison argument shows
 \begin{align}
 \label{eq_BB.40}
 \varphi(x,t)
 <
 \kappa
 (T-t)^{-\frac{1}{p-1}}
 (1+\tfrac{(p-1)}{16p})^{-\frac{1}{p-1}}
 \qquad
 \text{for }
 (x,t)\in U(s_1).
 \end{align}
 From (a2),
 we can choose $t_1\in(s_1,T)$ such that
 \begin{align}
 \label{eq_BB.41}
 \sup_{|x|>\sqrt{(T-s_1)|\log(T-s_1)|}}
 \varphi(x,t)
 &<
 \kappa
 (T-t)^{-\frac{1}{p-1}}
 (1+\tfrac{(p-1)}{16p})^{-\frac{1}{p-1}}
 \end{align}
 for $t\in(t_1,T)$.
 Estimates \eqref{eq_BB.40} - \eqref{eq_BB.41}
 imply
 \begin{align}
 \label{eq_BB.42}
 \sup_{|z|>\sqrt\tau}
 \Phi(z,\tau)
 &<
 \kappa
 (1+\tfrac{(p-1)}{16p})^{-\frac{1}{p-1}}
 \qquad
 \text{for }
 \tau>-\log(T-t_1).
 \end{align}
 Combining \eqref{eq_BB.42} and \eqref{eq_BB.34} - \eqref{eq_BB.36},
 we conclude (s3).
 Te right hand inequality in (s2) is immediately obtained from (s3).
 Assumption (a3) directly implies the left hand inequality in (s2).
 The proof is completed.
 \end{proof}

 \section{Continuity of the blowup time.}
 \begin{lem}
 \label{lemma_C}
 Assume $n\geq1$ and $p>1$.
 Let $\varphi(x,t)$ be a solution of
 $\varphi_t=\Delta_x\varphi+|\varphi|^{p-1}\varphi$
 satisfying
 \begin{enumerate}[\rm(q1)]
 \setlength{\leftskip}{5mm}
 \item $\varphi(x,t)$ blows up in a finite time $t=T$,
 \item $\varphi(x,t)>-C_1$ \quad {\rm for} $(x,t)\in\R^n\times(0,T)$,
 \item $\dis\sup_{x\in\R^n}\|u(x,t)\|_{L_x^\infty(\R^n)}<C_2(T-t)^{-\frac{1}{p-1}}$
 \quad {\rm for} $t\in(0,T)$,
 \item a rescaled function
 $\Phi(z,\tau)=e^{-\frac{\tau}{p-1}}\varphi(z\sqrt{T-t},T-e^{-\tau})$
 satisfies
 \begin{align*}
 \lim_{\tau\to\infty}\|\Phi(z,\tau)-\kappa\|_\rho=0.
 \end{align*}
 \end{enumerate}
 Let $u(x,t)$ be a solution of $u_t=\Delta_xu+|u|^{p-1}u$ with $u|_{t=0}=u_0$,
 and $T(u)$ be the blowup time of $u(x,t)$.
 For any $\epsilon>0$
 there exists $\delta_1\in(0,1)$
 such that if $\|u_0-\varphi_0\|_{L_x^\infty(\R^n)}<\delta_1$,
 then
 $|T(u)-T|<\epsilon$.
 \end{lem}
 \begin{proof}
 From a standard argument,
 we easily see that
 for any $\eta\in(0,1)$ there exists $h_1\in(0,1)$
 such that if $\|u_0-\varphi_0\|_{L_x^\infty(\R^n)}<h_1$,
 then
 $T(u)>T-\eta$.
 Hence
 it is sufficient to show the continuity of $T(u)$ from above.
 Let $\eta\in(0,1)$ and set
 \begin{align*}
 \Phi(z,\tau;\eta)
 =
 e^{-\frac{\tau}{p-1}}
 \varphi(e^{-\frac{\tau}{2}}z,T+\eta-e^{-\tau}).
 \end{align*}
 This function is defined on
 $\tau\in(-\log(T+\eta),-\log\eta)$.
 Define
 $s=\tilde s(\tau,\eta)$ by
 \begin{align*}
 T+\eta
 -
 e^{-\tau}
 =
 T
 -
 e^{-\tilde s}
 &
 \Leftrightarrow
 \quad
 e^{-\tilde s}=-\eta+e^{-\tau}
 \\
 &
 \Leftrightarrow
 \quad
 e^\tau=\tfrac{e^{\tilde s}}{1+\eta e^{\tilde s}}.
 \end{align*}
 From (q3) - (q4),
 for any $R>1$ and $\epsilon_1\in(0,1)$,
 there exists $\sigma_1\in(-\log T,\infty)$ such that
 \begin{align}
 \label{eq_C.1}
 |\varphi(x,t)-\kappa&(T-t)^{-\frac{1}{p-1}}|
 <
 \epsilon_1
 (T-t)^{-\frac{1}{p-1}}
 \\
 \nonumber
 &
 \text{for }
 |x|<R\sqrt{T-t}
 \text{ and }
 t\in(T-e^{-\sigma_1},T).
 \end{align}
 For $\sigma\in(\sigma_1,\infty)$,
 we define $\eta(\sigma)$, $\tau(\sigma)$ and $s(\sigma)$ by
 \begin{itemize}
 \item 
 $\eta(\sigma)=\frac{1}{4}e^{-\sigma}$,
 \item
 $e^{\tau(\sigma)}=\frac{4}{5}e^{\sigma}$,
 \item
 $s(\sigma)=\sigma$.
 \end{itemize}
 From this relation,
 it holds that
 \begin{itemize}
 \item
 $\tilde s(\tau(\sigma),\eta(\sigma))=\sigma=s(\sigma)$.
 \end{itemize}
 From \eqref{eq_C.1},
 we see that
 for $\sigma\in(\sigma_1,\infty)$
 \begin{align*}
 |
 \Phi(z&,\tau(\sigma);\eta(\sigma))
 -
 \kappa
 e^{-\frac{\tau(\sigma)}{p-1}}
 e^\frac{s(\sigma)}{p-1}
 |
 \\
 &=
 |
 e^{-\frac{\tau(\sigma)}{p-1}}
 \varphi(e^{-\frac{\tau(\sigma)}{2}}z,T+\eta(\sigma)-e^{-\tau(\sigma)})
 -
 \kappa
 e^{-\frac{\tau(\sigma)}{p-1}}
 e^\frac{s(\sigma)}{p-1}
 \\
 &=
 |
 e^{-\frac{\tau(\sigma)}{p-1}}
 \varphi(e^{-\frac{\tau(\sigma)}{2}}z,T-e^{-s(\sigma)})
 -
 \kappa
 e^{-\frac{\tau(\sigma)}{p-1}}
 e^\frac{s(\sigma)}{p-1}
 |
 \\
 &=
 e^{-\frac{\tau(\sigma)}{p-1}}
 |
 \varphi(e^{-\frac{\tau(\sigma)}{2}}z,T-e^{-s(\sigma)})
 -
 \kappa
 e^\frac{s(\sigma)}{p-1}
 |
 \\
 &=
 e^{-\frac{\tau(\sigma)}{p-1}}
 |
 \varphi(e^{-\frac{\tau(\sigma)}{2}}z,T-e^{-\sigma})
 -
 \kappa
 e^\frac{\sigma}{p-1}
 |
 \\
 &<
 \epsilon_1
 e^{-\frac{\tau(\sigma)}{p-1}}
 e^\frac{\sigma}{p-1}
 \qquad
 \text{for }
 |e^{-\frac{\tau(\sigma)}{2}}z|
 <
 Re^{-\frac{\sigma}{2}}.
 \end{align*}
 Since $e^{\tau(\sigma)}e^{-s(\sigma)}=e^{\tau(\sigma)}e^{-\sigma}=\frac{4}{5}$,
 it follows that
 for
 $\sigma\in(\sigma_1,\infty)$ 
 \begin{align*}
 |
 \Phi(z,\tau(\sigma);\eta(\sigma))
 -
 \kappa
 (\tfrac{5}{4})^{\frac{1}{p-1}}|
 &<
 \epsilon_1
 (\tfrac{5}{4})^{\frac{1}{p-1}}
 \qquad
 \text{for }
 |z|
 <
 R\sqrt{\tfrac{4}{5}}.
 \end{align*}
 Therefore
 if we choose $\epsilon_1\in(0,1)$ small enough,
 then
 it holds that
 for
 $\sigma\in(\sigma_1,\infty)$ 
 \begin{align}
 \label{eq_C.2}
 \Phi(z,\tau(\sigma);\eta(\sigma))
 >
 (1+\tfrac{\theta}{2})
 \kappa
 \qquad
 \text{for }
 |z|
 <
 R\sqrt{\tfrac{4}{5}},
 \end{align}
 here $\theta=(\tfrac{5}{4})^\frac{1}{p-1}-1>0$.
 We define a rescaled function $w(z,\tau)$ by
 \begin{align*}
 w(z,\tau;\sigma)
 =
 e^{-\frac{\tau}{p-1}}
 u(ze^{-\frac{\tau}{2}},T+\eta(\sigma)-e^{-\tau}).
 \end{align*}
 The function $w(z,\tau;\sigma)$ is defined on
 $\tau
 \in
 (-\log(T+\eta(\sigma)),\tau_\infty(\sigma))$,
 where
 \begin{align*}
 \tau_\infty(\sigma)
 =
 \begin{cases}
 -\log(T+\eta(\sigma)-T(u)) & \text{if } T+\eta(\sigma)-T(u)>0,
 \\
 \infty & \text{if } T+\eta(\sigma)-T(u)=0,
 \\
 \infty & \text{if } T+\eta(\sigma)-T(u)<0.
 \end{cases}
 \end{align*}
 We easily see that
 \begin{align*}
 &
 \sup_{z\in\R^n}
 |
 w(z,\tau;\sigma)
 -
 \Phi(z,\tau;\eta(\sigma))
 |
 \\
 &=
 e^{-\frac{\tau}{p-1}}
 \sup_{z\in\R^n}
 |
 u(ze^{-\frac{\tau}{2}},T+\eta(\sigma)-e^{-\tau})
 -
 \varphi(ze^{-\frac{\tau}{2}},T+\eta(\sigma)-e^{-\tau})
 |
 \\
 &=
 e^{-\frac{\tau}{p-1}}
 \sup_{x\in\R^n}
 |
 u(x,t)
 -
 \varphi(x,t)
 |
 \qquad
 \text{with }
 t=T+\eta(\sigma)-e^{-\tau}.
 \end{align*}
 From the definition of $\tau(\sigma),\eta(\sigma)$, 
 we note that
 $T+\eta(\sigma)-e^{-\tau(\sigma)}=T-e^{-\sigma}$.
 By the continuous dependence of solutions on the initial data,
 for any fixed $\sigma\in(\sigma_1,\infty)$ and any $\epsilon\in(0,1)$,
 there exists $\delta_1=\delta_1(\sigma)\in(0,1)$ such that
 if $\|u_0-\varphi_0\|_{L_x^\infty(\R~n)}<\delta_1$,
 then
 it holds that
 \begin{align}
 \nonumber
 &
 \sup_{z\in\R^n}|
 w(z,\tau(\sigma);\sigma)
 -
 \Phi(z,\tau(\sigma);\eta(\sigma))
 |
 \\
 \label{eq_C.3}
 &<
 e^{-\frac{\tau(\sigma)}{p-1}}
 \sup_{x\in\R^n}
 |
 u(x,T-e^{-\sigma})
 -
 \varphi(x,T-e^{-\sigma})
 |
 <
 \epsilon.
 \end{align}
 Combining \eqref{eq_C.2} - \eqref{eq_C.3},
 we deduce that
 \begin{align}
 \label{eq_C.4}
 w(z,\tau(\sigma);\sigma)
 >
 (1+\tfrac{\theta}{4})
 \kappa
 \qquad
 \text{for }
 |z|
 <
 R\sqrt{\tfrac{4}{5}}.
 \end{align}
 Furthermore
 from (q2) - (q3),
 it holds that
 \begin{align}
 \nonumber
 |\Phi(z,\tau(\sigma);\eta(\sigma))|
 &=
 |e^{-\frac{\tau(\sigma)}{p-1}}
 \varphi(e^{-\frac{\tau(\sigma)}{2}}z,T+\eta(\sigma)-e^{-\tau(\sigma)})|
 \\
 \nonumber
 &=
 e^{-\frac{\tau(\sigma)}{p-1}}
 |\varphi(e^{-\frac{\tau(\sigma)}{2}}z,T-e^{-\sigma})|
 \\
 \label{eq_C.5}
 &<
 C_2
 e^{-\frac{\tau(\sigma)}{p-1}}
 e^\frac{\sigma}{p-1}
 =
 C_2
 (\tfrac{5}{4})^\frac{1}{p-1}
 \qquad
 \text{for }
 z\in\R^n.
 \end{align}
 Hence
 from \eqref{eq_C.5} and \eqref{eq_C.3},
 we have
 \begin{align}
 \label{eq_C.6}
 \sup_{z\in\R^n}
 |w(z,\tau(\sigma);\sigma)|
 <
 2C_2
 (\tfrac{5}{4})^\frac{1}{p-1}.
 \end{align}
 Using \eqref{eq_C.4} and \eqref{eq_C.6},
 we can verify that
 \begin{align*}
 \langle w(z,\tau(\sigma)&;\sigma)-\kappa,H_0\rangle_\rho
 =
 \langle
 \{w(z,\tau(\sigma);\sigma)-\kappa\}{\bf 1}_{|z|<R\sqrt{\frac{4}{5}}},H_0
 \rangle_\rho
 \\
 &\quad
 +
 \langle \{w(z,\tau(\sigma);\sigma)-\kappa\}{\bf 1}_{|z>R\sqrt{\frac{4}{5}}},H_0\rangle_\rho
 \\
 &>
 \tfrac{\theta\kappa}{4}
 \langle {\bf 1}_{|z|<R\sqrt{\frac{4}{5}}},H_0\rangle_\rho
 -
 2C_2
 (\tfrac{5}{4})^\frac{1}{p-1}
 \langle {\bf 1}_{|z>R\sqrt{\frac{4}{5}}},H_0\rangle_\rho.
 \end{align*}
 We recall that $\theta$ is a positive constant independent of $R$.
 Therefore
 if $R$ is large enough,
 then
 it holds that
 \begin{align*}
 \langle w(z,\tau(\sigma);\sigma)-\kappa,H_0\rangle_\rho
 &>
 \tfrac{\theta\kappa}{8H_0}.
 \end{align*}
 Furthermore
 from the definition of $w(z,\tau;\sigma)$ and (q2),
 it is clear that
 \begin{align*}
 w(z,\tau;\sigma)
 &=
 e^{-\frac{\tau}{p-1}}
 u(ze^{-\frac{\tau}{2}},T+\eta(\sigma)-e^{-\tau})
 \\
 &>
 -
 C_1
 e^{-\frac{\tau}{p-1}}
 \qquad
 \text{for }
 \tau\in(-\log(T+\eta(\sigma)),\tau_\infty(\sigma)).
 \end{align*}
 Therefore
 from Proposition \ref{proposition_4.2},
 $w(z,\tau;\sigma)$ blows up in a finite time $\tau=\tau_\text{b}<\infty$,
 which implies that $\tau_\infty(\sigma)<\infty$.
 This shows $T(u)<T+\eta(\sigma)=T+\frac{1}{4}e^{-\sigma}$.
 Since we can choose $\sigma\in(\sigma_1,\infty)$ arbitrary large in this argument,
 we complete the proof.
 \end{proof}

\section*{Acknowledgement}
The author is partly supported by JSPS KAKENHI Grant Number 23K03161.

\section*{Data availability}
Data sharing is not applicable to this article
as no datasets were generated or analyzed during the current study.

\section*{Conflict of interest}
The author declares that he has no conflict of interest.

 
\end{document}